\documentclass{amsart}

\usepackage{hyperref}
\hypersetup{
    colorlinks=true,
    linkcolor=blue,
    citecolor=blue,
    urlcolor=blue,
}

\makeatletter
\newcommand\org@maketitle{}
\newcommand\@authors{}
\let\org@maketitle\maketitle
\def\maketitle{%
	\let\@authors\authors
	\nxandlist{; }{ and }{; }\@authors
	\hypersetup{
		linktocpage=true,
		pdftitle={\@title},
                pdfauthor={\@authors},
                pdfsubject={\subjclassname. \@subjclass},
		pdfkeywords={\@keywords}
	}%
	\org@maketitle
}
\makeatother
\usepackage{comment}
\usepackage{amssymb}
\usepackage{rsfso}
\usepackage{mathtools}
\usepackage{microtype}
\usepackage[alphabetic,msc-links]{amsrefs}
\usepackage{enumitem}
\usepackage{amsfonts}
\usepackage{color}

\usepackage{doi}
\renewcommand{\PrintDOI}[1]{\doi{#1}}

\newcommand{\arxiv}[1]{arXiv:\href{https://arxiv.org/abs/#1}{#1}}
\usepackage[
  hmarginratio={1:1},     % equal left and right margins
  vmarginratio={1:1},     % equal top and bottom margins
  textwidth=17cm,        % new text width
  textheight=21cm,
  heightrounded,          % always useful
]{geometry}

\numberwithin{equation}{section}

\newtheorem{theorem}{Theorem}[section]
\newtheorem{teo}[theorem]{Theorem}
\newtheorem{lem}[theorem]{Lemma}
\newtheorem{cor}[theorem]{Corollary}
\newtheorem{pro}[theorem]{Proposition}
\theoremstyle{definition}
\newtheorem{defi}[theorem]{Definition}

\theoremstyle{remark}
\newtheorem{oss}[theorem]{Remark}

\newcommand{\e}{\varepsilon}

\newcommand{\R}{\mathbb{R}}
\newcommand{\N}{\mathbb{N}}

\newcommand{\D}{\nabla}

\newcommand{\supp}{\operatorname{spt}}
\renewcommand{\div}{\operatorname{div}}

\newcommand{\dist}{\operatorname{dist}}

\DeclareMathOperator*{\esssup}{ess\,sup}

\def\XXint#1#2#3{{\setbox0=\hbox{$#1{#2#3}{\int}$}
    \vcenter{\hbox{$#2#3$}}\kern-.5\wd0}}

\DeclareRobustCommand{\rchi}{{\mathpalette\irchi\relax}}
\newcommand{\irchi}[2]{\raisebox{\depth}{$#1\chi$}}

\mathchardef\ordinarycolon\mathcode`\:
\mathcode`\:=\string"8000
\begingroup \catcode`\:=\active
  \gdef:{\mathrel{\mathop\ordinarycolon}}
\endgroup

\author{A. Audrito}
\author{G. Fioravanti}
\author{S. Vita}
\address{Alessandro Audrito\newline\indent
Dipartimento di Scienze Matematiche ``Giuseppe Luigi Lagrange''
\newline\indent
Politecnico di Torino
\newline\indent
Corso Duca degli Abruzzi 24, 10129, Torino, Italy}
\email{alessandro.audrito@polito.it}
\address{Gabriele Fioravanti\newline\indent
Dipartimento di Matematica ``Giuseppe Peano''
\newline\indent
Universit\`a degli Studi di Torino
\newline\indent
Via Carlo Alberto 10, 10124, Torino, Italy}
\email{gabriele.fioravanti@unito.it}
\address{Stefano Vita\newline\indent
Dipartimento di Matematica ``Giuseppe Peano''
\newline\indent
Universit\`a degli Studi di Torino
\newline\indent
Via Carlo Alberto 10, 10124, Torino, Italy}
\email{stefano.vita@unito.it}

\title[Schauder estimates for parabolic equations with degenerate or singular weights]{Schauder estimates for parabolic equations with degenerate or singular weights}


\allowdisplaybreaks

\begin{document}

\begin{abstract} 
We establish some $C^{0,\alpha}$ and $C^{1,\alpha}$ regularity estimates for a class of weighted parabolic problems in divergence form. The main novelty is that the weights may vanish or explode on a characteristic hyperplane $\Sigma$ as a power $a > -1$ of the distance to $\Sigma$. The estimates we obtain are sharp with respect to the assumptions on coefficients and data. Our methods rely on a regularization of the equation and some uniform regularity estimates combined with a Liouville theorem and an approximation argument. As a corollary of our main result, we obtain similar $C^{1,\alpha}$ estimates when the degeneracy/singularity of the weight occurs on a regular hypersurface of cylindrical type.\end{abstract}

\maketitle
%\tableofcontents

\section{Introduction}
In this paper we prove some H\"older and Schauder regularity estimates for solutions to a special class of weighted parabolic equations: the weights appearing in the equations degenerate or explode on a characteristic hyperplane $\Sigma$ as $\mathrm{dist}(\cdot,\Sigma)^a$, where $a>-1$ is a fixed parameter. More precisely, we establish some local regularity estimates ``up to'' $\Sigma$ for weak solutions to the problem
\begin{equation}\label{eq:1}
\begin{cases}
    y^a \partial_tu - {\div }(y^{a} A\nabla u) = y^a f + {\div}(y^{a} F) \quad &\text{in } Q_1^+,\\
    \displaystyle{\lim_{y\to0^+}} y^a(A\nabla u+F)\cdot e_{N+1}=0              &\text{on }\partial^0Q_1^+.
\end{cases}
\end{equation}
Here $N\ge1$, $(z,t)=(x,y,t)\in\R^N\times\R\times\R$, $\Sigma=\{y=0\}$ and $\mathrm{dist}(P,\Sigma)^a=y^a$ is locally integrable whenever $a>-1$. $B_1\subset\R^{N+1}$ denotes the unit ball with center at $0$ and $B_1^+ := B_1\cap\{y>0\}$ the unit upper-half ball. Similar, if $I_1:=(-1,1)$, then $Q_1 := B_1\times I_1$ is the unit parabolic cylinder, $Q_1^+:=B_1^+\times I_1$ is the unit upper-half cylinder, while $\partial^0 Q_1^+ = Q_1 \cap\{y=0\}$. The operators $\nabla$ and $\div$ denote the gradient and the divergence w.r.t. the spatial variable $z$, respectively.  Furthermore, $A:Q_1^+\to \R^{N+1,N+1}$ is a symmetric $(N+1)$-dimensional matrix satisfying the following ellipticity condition: there exist $0 < \lambda \le \Lambda < +\infty$ such that
\begin{equation}\label{eq:UnifEll}
    \lambda|\xi|^2\le A(z,t)\xi\cdot\xi\le\Lambda|\xi|^2,
\end{equation}
for all $\xi\in\mathbb{R}^{N+1}$ and a.e. $(z,t)\in Q_1^+$, while the forcing terms in the r.h.s. $f:Q_1^+ \to \R$ and $F:Q_1^+ \to \R^{N+1}$ are given functions belonging to some suitable functional spaces.

In the simplest case where $A = I$ and $f=|F|=0$, problem \eqref{eq:1} (posed in the whole space) is nothing more than the gradient flow of the energy
\[
\int_{\R^{N+1}_+} y^a |\nabla v|^2 dz, \quad v \in H^1(\R^{N+1}_+,y^a).
\]
So, as one may imagine,
%Although the precise notion of weak solutions of \eqref{eq:1} will be given later in Section \ref{section2},  
the natural functional setting involves the weighted Sobolev spaces involving time. Precisely, we say that $u$ is a weak solution to \eqref{eq:1} if $u \in L^2(I_1;H^1(B_1^+,y^a))\cap L^\infty(I_1;L^2(B_1^+,y^a))$ and satisfies
$$\int_{Q_1^+} y^a \big(-u\partial_t \phi+A\nabla u\cdot\nabla\phi\big)dzdt=\int_{Q_1^+} y^a\big(f\phi-F\cdot\nabla\phi\big)dzdt,$$
for every test function $\phi\in C_c^\infty(Q_1)$ (cf. Definition \ref{def.solution-mezza-palla}). Notice that weak solutions formally satisfy the \textsl{conormal boundary condition} 
\begin{equation}\label{naturalconormal}
\lim_{y\to0^+}y^a(A\nabla u+F)\cdot e_{N+1}=0 \quad \text{ on } \Sigma,
\end{equation}
appearing in \eqref{eq:1}: as standard in Neumann-type problems, one can easily check this integrating by parts and use the fact that the test functions $\phi$ need not to vanish on $\Sigma$. Actually, as a consequence of our main theorem (see \eqref{eq:boundary}), we will obtain that, under suitable regularity assumptions on the data, weak solutions satisfy 
\[
(A\nabla u+F)\cdot e_{N+1}=0 \quad \text{ on } \Sigma,
\]
which is stronger than \eqref{naturalconormal} (at least when $a > 0$) and, when $A = I$ and $|F| = 0$, reduces to the \emph{classical} Neumann boundary condition.

\smallskip

The regularity theory for \textsl{uniformly parabolic} equations is nowadays classical, see for instance \cites{LSU,Lie96}. By \textsl{uniformly parabolic} we mean that the second order leading term of the equation possesses uniformly elliptic coefficients in the sense of \eqref{eq:UnifEll}.
Then, many efforts have been made to prove regularity results for non-uniformly parabolic equations; that is, whenever at least one of the two bounds in \eqref{eq:UnifEll} fails.
Among all the papers on this topic, we quote the pioneering works \cite{FabKenSer82} for the elliptic case and then \cite{ChiSer85} for the parabolic counterpart. In these papers the authors established some Harnack inequalities and H\"older estimates for weak solutions to a family of second order equations with degenerate or singular weights, in which the uniform ellipticity condition fails: the weights may vanish or explode somewhere. Such results cover the case of weights $\omega$, either coming from quasiconformal mapping or belonging to the $A_2$-Muckenhoupt class, that is,
\begin{equation}\label{muckA2}
        \sup_{B} \Big(\frac{1}{|B|}\int_{B} \omega \Big)\Big(\frac{1}{|B|}\int_{B} \omega^{-1}\Big)\le C,
\end{equation}
where the supremum is taken over every ball $B\subset\mathbb{R}^{N+1}$ (see also \cite{GutWhe91} and the more recent \cites{banerjee,audritoterracini,audrito}). The weight term $|y|^a$ we are considering here is $A_2$-Muckenhoupt when $a\in(-1,1)$ only and thus, in this range, part of our theory falls into \cites{FabKenSer82,ChiSer85}: in particular, the H\"older continuity of solutions is already available at least for some implicit H\"older exponent.

However, the peculiar geometry of the degeneracy/singularity set of our weight - the characteristic hyperplane $\Sigma$ - allows us to get more information compared to the general theory quoted above. In fact, as already done in \cites{SirTerVit21a,TerTorVit22} in the elliptic setting, the approach we follow in this paper allows us to cover the full range $a>-1$ and eventually will allow us to show Schauder $C^{k,\alpha}_p$ estimates for any $k\in\mathbb N$ (we will treat the case $k \geq 2$ in a forthcoming paper). It is important to remark here that the regularity we obtain strongly relies on the \textsl{natural conormal boundary condition} \eqref{naturalconormal} we impose on the characteristic hyperplane $\Sigma$. As one may imagine, different boundary conditions lead to different regularity estimates: for instance, $v = y^{1-a}$ weakly solves ${\div }(y^{a} \nabla v)=0$ in $B_1^+$ with homogeneous Dirichlet boundary condition at $\Sigma$ whenever $a<1$, but it is no more than $C^{\lfloor 1-a\rfloor,1-a-\lfloor 1-a\rfloor}$ regular.

On the other hand, the study of weighted problems like \eqref{eq:1} is strongly related to the theory of \textsl{edge operators} \cites{Maz91,MazVer14}, \textsl{nonlocal operators} and \textsl{nonlocal diffusion}. The latter is the major motivation for the parabolic theory we develop here: it relies in the connection between a class of fractional heat operators like $(\partial_t - \Delta)^{\frac{1-a}{2}}$ - possibly with variable coefficients - and their extension theories \cites{NysSan16,StiTor17,banerjee}, which represent the parabolic counterpart of \cite{CafSil07}. Such kind of operators have been widely investigated in the last years, in many different contexts: we quote \cite{CAFFA-MELLET-SIRE} for reaction-diffusion equations with nonlocal diffusion, \cites{DGPT17,AthCafMil18,BanDanGarPet21} for obstacle type problems, \cite{audritoterracini} for the nodal set analysis of sign-changing solutions, and \cite{HYDER} where a new nonlocal harmonic maps flow was recently introduced.
A special mention goes to \cite{BisSti21}, where the authors proved some Schauder estimates for solutions to fractional parabolic equations involving $(\partial_t-\div_x(A(x)\nabla_x))^{\frac{1-a}{2}}$: respect to our notation, this corresponds to estimates in the $(x,t)$-variables on $\Sigma$ and $a\in(-1,1)$.

We also mention \cites{DongDirichlet,DonPha23} where the authors deal with parabolic weighted equations in divergence form as in \eqref{eq:1} (and in non divergence form as well). They require weaker assumptions on coefficients, and obtain estimates for solutions in weighted Sobolev $W^{1,q}$ spaces. This kind of result is comparable to our $C^{0,\alpha}_p$ regularity theory when $q$ is large but not with the higher $C^{1,\alpha}_p$ regularity we are able to obtain. However, as we will explain in a moment, we need to establish some H\"older regularity estimates which are stable with respect to a perturbation of the weight term, and this does not follow from \cite{DonPha23}. The stability is crucial in order to establish the higher order estimates with blow-up techniques.

\subsection*{Main results}
The main goal of the paper is to prove some local $C^{0,\alpha}_p$ and $C^{1,\alpha}_p$ regularity estimates - up to the characteristic hyperplane $\Sigma$ - for weak solutions to \eqref{eq:1} (see Section \ref{section:parabolic:holder;spaces} for the definitions of the H\"older space of parabolic type), under suitable assumptions on the matrix $A$ and the right hand sides. As already mentioned, the higher regularity of solutions, which is quite different from its elliptic counterpart \cites{SirTerVit21a,TerTorVit22}, will be treated in a forthcoming paper.

The main idea is to extend to the parabolic framework the regularization argument used in \cites{SirTerVit21a,SirTerVit21b} in the elliptic one: for $\varepsilon\in(0,1)$, we introduce the family of smooth weights
$$
\rho_\varepsilon^a(y):=(\varepsilon^2+y^2)^{a/2},
$$
and we consider weak solutions to  
\begin{equation}\label{eq:1:reg}
\begin{cases}
    \rho_\varepsilon^a \partial_t u - {\div }(\rho_\varepsilon^a A\nabla u) = \rho_\varepsilon^a f + {\div}(\rho_\varepsilon^a F) \quad &\text{in } Q_1^+\\
    \displaystyle{\lim_{y\to0^+}} \rho_\varepsilon^a(A\nabla u+F)\cdot e_{N+1}=0                                                        &\text{on } \partial^0Q_1^+,
\end{cases}
\end{equation}
which corresponds to the problem associated to the regularized weight (notice that by construction $\rho_\varepsilon^a(y)\to y^a$ almost everywhere as $\varepsilon\to0^+$). Since $\rho_\varepsilon^a$ is (locally) bounded and bounded away from zero, problem \eqref{eq:1:reg} is uniformly parabolic.  Consequently, the classical Schauder regularity theory applies (see for instance \cites{Lie96,LSU}) and one obtains $C^{0,\alpha}_p$ and $C^{1,\alpha}_p$ regularity estimates with constants \emph{possibly depending} on $\e$. The main result of the paper shows that such estimates are \emph{uniform} in $\e \in (0,1)$ and pass to the limit as $\e \to 0^+$. We refer to this property as $\varepsilon$-stability of the estimates. The latter, together with a fine approximation procedure (see Section \ref{section4}) yields our main result: 
\begin{teo}\label{teo:C^1,alpha}
    Let $a>-1$ and $u$ be a weak solution to 
    \eqref{eq:1}, in the sense of Definition \ref{def.solution-mezza-palla}. Then
\begin{enumerate}
    \item [(i)] If $A$ is a continuous matrix satisfying \eqref{eq:UnifEll}, $f\in L^p(Q_1^+,y^a )$ with $p>\frac{N+3+a^+}{2}$, $F\in L^q(Q_1^+, y^a )^{N+1}$ with $q>N+3+a^+$, $\alpha\in(0,1)\cap(0,2-\frac{N+3+a^+}{p}]\cap(0,1-\frac{N+3+a^+}{q}]$, then there exists a constant $C>0$, depending on $N$, $a$, $\lambda$, $\Lambda$, $p$, $q$ and $\alpha$ such that
    \begin{equation}\label{eq:C0,alpha}
        \|u\|_{C^{0,\alpha}_p(Q_{1/2}^+)}\le C\Big(
    \|u\|_{L^2(Q_1^+,y^a)}+
    \|f\|_{L^p(Q_1^+,y^a)}+
    \|F\|_{L^q(Q_1^+,y^a)}
    \Big).
    \end{equation}
    \item [(ii)] If $A,F\in C^{0,\alpha}_p(Q_1^+)$, with $A$ satisfying \eqref{eq:UnifEll}, $f\in L^p(Q_1^+,y^a )$ with $p>{N+3+a^+}$, $\alpha\in(0,1)\cap(0,1-\frac{N+3+a^+}{p}]$, then there exists a constant $C>0$, depending on $N$, $a$, $\lambda$, $\Lambda$, $p$, $\alpha$ and $\|A\|_{C^{0,\alpha}_p(Q_1^+)}$ such that
    \begin{equation}\label{eq:C1,alpha}
    \|u\|_{C^{1,\alpha}_p(Q_{1/2}^+)}\le C\Big(
    \|u\|_{L^2(Q_1^+,y^a)}+
    \|f\|_{L^p(Q_1^+,y^a)}+
    \|F\|_{C^{0,\alpha}_p(Q_1^+)}
    \Big).
    \end{equation}
    In addition, $u$ satisfies the conormal boundary condition
    \begin{equation}\label{eq:boundary}
    (A\nabla u+F)\cdot e_{N+1}=0\quad\mathrm{ on \  }\partial^0Q_{1/2}^+.
    \end{equation}
\end{enumerate}
Moreover, the estimates \eqref{eq:C0,alpha} and \eqref{eq:C1,alpha} are $\varepsilon$-stable in the sense of Theorems \ref{teo C^0-alpha} and \ref{teo-C1 alpha epsilon}.
\end{teo}
We would like to remark that the $\varepsilon$-stability of the $C^{k,\alpha}_p$ estimates with respect to the regularization described above can not be valid when $k\geq2$, see \cite{SirTerVit21a}*{Remark 5.4}. Moreover, let us stress again the fact that the $C^{0,\alpha}_p$ regularity above is comparable to the regularity theory in \cite{DonPha23}. However, our approach for the $C^{1,\alpha}_p$ regularity requires the $\varepsilon$-stability of the $C^{0,\alpha}_p$ estimate above: to the best of our knowledge, this is completely new in the parabolic setting.

The proof of our main theorem is based on a contradiction argument combined with a blow-up procedure (in the spirit of the classic paper by Simon \cite{simon}), which crucially exploit the following Liouville-type theorem for entire solutions having a certain growth-control at infinity.
\begin{teo}\label{teo liouville 1}
Let $a>-1$, $\varepsilon\in[0,1)$, $\gamma \in [0,2)$ and let $u$ be an entire solution to
\begin{equation}\label{eq-liouville-1}
       \begin{cases}
       \rho_\varepsilon^{a}\partial_t u-{\div}(\rho_\varepsilon^a \nabla u)=0&\mathrm{ in \ } \mathbb{R}^{N+1}_+\times\mathbb{R} \\
       \rho_\varepsilon^a\partial_y u=0&\mathrm{ on \ } \partial \mathbb{R}_+^{N+1} \times\mathbb{R}.
   \end{cases}
\end{equation}
Assume that
\begin{equation}\label{growth}
       |u(z,t)|\le C(1+(|z|^2+|t|)^\gamma)^{1/2} \quad  \mathrm{ for \  a.e. \ } (z,t) \in \mathbb{R}^{N+1}_+\times\mathbb{R}.
\end{equation}
Then $u$ is a linear function depending only on $x$. Moreover, if $\gamma\in[0,1)$, then $u$ is constant.
\end{teo}

The proof of the Liouville theorem above is obtained with an iteration of a (parabolic) Caccioppoli-type inequality \eqref{caccioppoli}, in the spirit of \cite{TerTorVit22} and by a duality principle between $u$ and its weighted derivative $\rho_\varepsilon^a\partial_yu$ which solve respectively equations with weights $\rho_\varepsilon^a$ and $\rho_\varepsilon^{-a}$ as in \cite{CafSalSil08}.

\smallskip

Finally, as a consequence of our main theorem, we can treat more general equations with weights behaving as \emph{distance functions} to a $C^{1,\alpha}$ hypersurface $\Gamma \subset \R^{N+1}$ (curved characteristic manifolds) that we introduce below. Such equations are set in cylindrical domains $\Omega^+\times(-1,1)$ of $\R^{N+2}$ which ``live'' on one side of $\Gamma\times(-1,1)$. Specifically, up to rotations and dilations, $0\in\Gamma$ and there exist a spacial direction $y$ and a function $\varphi\in C^{1,\alpha}(B_1\cap\{y=0\})$ with $\varphi(0)=0$ and $\nabla_x\varphi(0)=0$ such that
\begin{equation}\label{phi}
\Omega^+\cap B_1=\{y>\varphi(x)\}\cap B_1,\qquad \Gamma\cap B_1=\{y=\varphi(x)\}\cap B_1.
\end{equation}
Then, the family of weights $\delta = \delta(z)$ we consider behave as a distance function to $\Gamma$ in the sense that $\delta\in C^{1,\alpha}(\Omega^+\cap B_1)$, and
\begin{equation}\label{delta}
\begin{cases}
\delta>0 &\mathrm{in \ } \Omega^+\cap B_1\\
|\nabla\delta|\geq c_0>0 &\mathrm{in \ } \Omega^+\cap B_1\\
\delta=0 &\mathrm{on \ } \Gamma\cap B_1,
\end{cases}
\end{equation}
and we consider weighted equations of the form
\begin{equation}\label{eq:1:curve}
\begin{cases}
    \delta^a \partial_tu - {\div }(\delta^{a} A\nabla u) = \delta^a f + {\div}(\delta^{a} F) \quad &\text{in } (\Omega^+\cap B_1)\times(-1,1),\\
    \delta^a(A\nabla u+F)\cdot \nu=0              &\text{on }(\Gamma\cap B_1)\times(-1,1),
\end{cases}
\end{equation}
where $\nu$ is the unit outward normal vector to $\Omega^+$ on $\Gamma$. The precise definition of solutions to \eqref{eq:1:curve} will be given later in Section \ref{sec:curve}, see Definition \ref{def.solution-curve}.

\begin{cor}\label{cor:C^1,alpha}
    Let $a>-1$ and $u$ be a weak solution to 
    \eqref{eq:1:curve}, in the sense of Definition \ref{def.solution-curve}. Let $\varphi\in C^{1,\alpha}(B_1\cap\{y=0\})$ be the parametrization defined in \eqref{phi} and $\delta\in C^{1,\alpha}(\Omega^+\cap B_1)$ satisfying \eqref{delta}. 
    
    Let $A,F\in C^{0,\alpha}_p((\Omega^+\cap B_1)\times(-1,1))$, with $A$ satisfying \eqref{eq:UnifEll}, $f\in L^p((\Omega^+\cap B_1)\times(-1,1),\delta^a )$ with $p>{N+3+a^+}$, $\alpha\in(0,1)\cap(0,1-\frac{N+3+a^+}{p}]$. Then, there exists a constant $C>0$, depending on $N$, $a$, $\lambda$, $\Lambda$, $p$, $\alpha$, $c_0$, $\|A\|_{C^{0,\alpha}_p((\Omega^+\cap B_1)\times(-1,1))}$, $\|\varphi\|_{C^{1,\alpha}(B_1\cap\{y=0\})}$ and $\|\delta\|_{C^{1,\alpha}(\Omega^+\cap B_1)}$ such that
    \begin{equation}\label{eq:C1,alpha:curve}
    \begin{aligned}
    \|u\|_{C^{1,\alpha}_p((\Omega^+\cap B_{1/2})\times(- 1/2,1/2))}\le C\Big(
    &\|u\|_{L^2((\Omega^+\cap B_1)\times(-1,1),\delta^a)} \\
    & \quad + \|f\|_{L^p((\Omega^+\cap B_1)\times(-1,1),\delta^a)}+
    \|F\|_{C^{0,\alpha}_p((\Omega^+\cap B_1)\times(-1,1))}
    \Big).
    \end{aligned}
    \end{equation}
    In addition, $u$ satisfies the conormal boundary condition
    \begin{equation}\label{eq:boundary:curve}
    (A\nabla u+F)\cdot \nu=0\quad\mathrm{ on \ }(\Gamma\cap B_1)\times(-1,1),
    \end{equation}
where $\nu$ is the unit outward normal vector to $\Omega^+$ on $\Gamma$.
\end{cor}

\subsection*{Structure of the paper}

The paper is organized as follows: in Section \ref{section2} we set up the problem introducing the energy spaces and the definition of weak solutions to \eqref{eq:1}.
In Section \ref{section3} we prove some uniform estimates, namely the Caccioppoli's inequality and the $L^\infty$ bounds, by using the De Giorgi's iterative technique. Section \ref{section4} is devoted to the proof of the approximation results, that is, the convergence (in suitable energy spaces) of the regularized solutions to weak solutions to \eqref{eq:1}. 
In Section \ref{section5}, we prove the Liouville Theorem \ref{teo liouville 1}. Finally, in Section \ref{section6} and Section \ref{section7}, we show the $\e$-stability of $C^{0,\alpha}_p$ and $C^{1,\alpha}_p$ regularity estimates mentioned above. This, together with the approximation argument in Section \ref{section4}, will prove our main Theorem \ref{teo:C^1,alpha} and Corollary \ref{cor:C^1,alpha}.

\section{Functional setting}\label{section2}
The present section is mostly devoted to the functional setting of the problem.

\subsection{Functional spaces}
\subsubsection{$C^k$ spaces} Let $k \in \mathbb{N}$ and $r>0$. We set 
\begin{align*}
C^k(\overline{B}_r\setminus\Sigma) &:= \{u \in C^k(B_r\setminus\Sigma): D^\alpha u \text{ is uniformly continuous on every } U \subset B_r \\
& \quad\;\;\; \text{ with } \dist(U,\Sigma) > 0, \text{ for every multiindex } |\alpha| \leq k \}, \\
C_c^k(\overline{B}_r\setminus\Sigma) &:= \{u \in C^k(\overline{B}_r\setminus\Sigma): \supp u \subset\subset  \overline{B}_r\setminus\Sigma  \}, \\
C_c^\infty(\overline{B}_r\setminus\Sigma) &:=  \bigcap_{k=1}^\infty C_c^k(\overline{B}_r\setminus\Sigma). 
\end{align*}
\subsubsection{Sobolev spaces} Let $a\in\mathbb{R}$, $\varepsilon\in[0,1)$ and $r>0$. For $p > 1$, we set 
\[
L^p(B_r,\rho_\varepsilon^a) := \{u:B_r \to \R \text{ measurable: } \int_{B_r}\rho_\varepsilon^a |u|^p dz < +\infty \},
\]
equipped with the norm 
\[
\|u\|_{L^p(B_r,\rho_\varepsilon^a )} := \left( \int_{B_r}\rho_\varepsilon^a |u|^p dz\right)^{{1}/{p}}.
\]
For fields, we define
\[
L^p(B_r,\rho_\varepsilon^a)^{N+1} := \{U:B_r \to \R^{N+1} \text{ measurable: } \int_{B_r}\rho_\varepsilon^a |U|^p dz < +\infty \},
\]
normed by $\|U\|_{L^p(B_r,\rho_\varepsilon^a)^{N+1}} := \||U|\|_{L^p(B_r,\rho_\varepsilon^a)}$.

The space $H^1(B_r,\rho_\varepsilon^a)$ is defined as the completion of $C^\infty(\overline{B}_r)$ w.r.t. the norm 
\begin{equation}\label{eq:WH1Norm}
\|v\|_{H^1(B_r,\rho_{\varepsilon}^{a})} = \left(\int_{B_r}\rho_{\varepsilon}^a v^2dz +\int_{B_r}\rho_{\varepsilon}^a|\nabla v|^2 dz \right)^{{1}/{2}},
\end{equation}
while the space $H^1_0(B_r,\rho_\varepsilon^a)$ is the completion of $C_c^\infty({B_r})$ w.r.t. the seminorm 
\begin{equation}\label{eq:WH10Norm}
\|v\|_{{H}_0^1(B_r,\rho_{\varepsilon}^{a})}=\left(\int_{B_r}\rho_{\varepsilon}^a|\nabla v|^2dz\right)^{{1}/{2}}.
\end{equation}
When $\e = 0$, we write $H^1(B_r,|y|^a)$ and $H_0^1(B_r,|y|^a)$, instead of $H^1(B_r,\rho_0^a)$ and $H_0^1(B_r,\rho_0^a)$, respectively. The symbol $H^{-1}(B_r,\rho_\varepsilon^a)$ denotes the topological dual space of $H^1_0(B_r,\rho_\varepsilon^a)$.

As observed in \cite{SirTerVit21a}, when $\e = 0$, the nature of such spaces is intrinsically related to the degeneracy/singularity of the weight $|y|^a$. Heuristically, when $a\le-1$ the weight $|y|^{a}$ is not locally integrable and thus the functions in $H^1(B_r,|y|^a)$ are forced to have zero trace on $\Sigma$. Conversely, when $a\ge1$, the weight has a strong degeneracy and the traces on $\Sigma$ of functions in $H^1(B_r,|y|^a)$ have no sense in general (this is due to the zero $H^{1}(B_r,|y|^a)$-capacity of $\Sigma$).

These observations suggest to introduce the space $\Tilde{H}^1(B_r,\rho_\varepsilon^a)$, defined as the completion of $C_c^\infty(\overline{B}_r\setminus\Sigma)$ w.r.t. \eqref{eq:WH1Norm} and similarly $\Tilde{H}^1_0(B_r,\rho_\varepsilon^a)$ as the completion of $C_c^\infty({B_r}\setminus\Sigma)$ w.r.t. \eqref{eq:WH10Norm}. As above, when $\e = 0$, we set $\Tilde{H}^1(B_r,|y|^a) := \Tilde{H}^1(B_r,\rho_0^a)$ and $\Tilde{H}_0^1(B_r,|y|^a) := \Tilde{H}_0^1(B_r,\rho_0^a)$.

When $a \in (-1,1)$, as we have previously remarked in the introduction, the weight $|y|^a$ belongs to the $A_2$ Muckenhoupt class; that is, \eqref{muckA2} holds true.

The following proposition characterizes the space $H^1(B_r,|y|^a)$. 
\begin{pro}\label{prop:CharH1a} (\cite{kil}*{Theorem 2.5} and \cite{SirTerVit21a}*{Proposition 2.2})

If $a \in (-1,1)$, then:
$$H^1(B_r,|y|^a) = W^{1,2}(B_r,|y|^a),$$
where 
\[
W^{1,2}(B_r,|y|^a) := \{ u \in L^2(B_r,|y|^a) : \nabla u \in L^2(B_r,|y|^a)^{N+1} \},
\]
and $\nabla u$ denotes the weak gradient of $u$.

If $a\in(-\infty,-1]\cup[1,+\infty)$, then:
\begin{align*}
H^1(B_r,|y|^a) &=\tilde{H}^1(B_r,|y|^a), \\
H_0^1(B_r,|y|^a) &=\tilde{H}_0^1(B_r,|y|^a);
\end{align*}
in particular, $C_c^\infty(\overline{B}_r\setminus\Sigma)$ is dense in $H^1(B_r,|y|^a)$ and $C_c^\infty(B_r\setminus\Sigma)$ is dense in $H_0^1(B_r,|y|^a)$.
\end{pro}
The spaces introduced above enjoy interesting Sobolev embedding properties, depending on the value of the parameter $a$. 
\begin{teo} (\cite{Haj96}*{Theorem 6} and \cite{SirTerVit21a}*{Theorem 2.4}) 

Let  $\varepsilon\in[0,1)$ and $r\in[1/2,1]$. Assume either $a > -1$ and $N\ge2$, or $a>0$ and $N=1$. Then there exists $C > 0$ depending only on $N$ and $a$ such that
    $$\left(\int_{B_r}\rho_\varepsilon^a|u|^{2^*_a}dz\right)^{2/2^*_a}\le C\left(\int_{B_r}\rho_\varepsilon^a u^2 dz+\int_{B_r}\rho_\varepsilon^a|\nabla u|^2dz\right),$$
for every $u\in H^1(B_r,\rho_\varepsilon^a)$, where 
\[
2^\ast_a := \frac{2(N+1+a^+)}{N+a^+-1}.
\]
Further, if $N=1$ and $a\in(-1,0]$, the above inequality holds with $2^*_a$ replaced with any $p\in[1,+\infty)$ and a constant $C>0$ depending only on $N$, $a$ and $p$.    

\end{teo}
\begin{teo} (\cite{SirTerVit21a}*{Theorem 2.5} and \cite{SirTerVit21b}*{Lemma B.5})

    Let $a\le-1$, $N\ge2$, $\varepsilon\in[0,1)$ and $r\in[1/2,1]$. Then there exists a constant $C>0$ depending only on $N$ and $a$ such that
\begin{equation}\label{eq:embedding:non:int}
\left(\int_{B_r}(\rho_\varepsilon^a)^{2^*/2}|u|^{2^*}dz\right)^{2/2^*}\le  C\left(\int_{B_r}\rho_\varepsilon^a u^2dz+\int_{B_r}\rho_\varepsilon^a |\nabla u|^2dz\right),
\end{equation}    
for every $u\in \tilde{H}^1(B_r,\rho_\varepsilon^a)$, where 
$$2^* := \frac{2(N+1)}{N-1}.$$ 
Moreover, the inequality \eqref{eq:embedding:non:int} implies that
$$\left(\int_{B_r}\rho_\varepsilon^a|u|^{2^*}dz\right)^{2/2^*}\le  C\left(\int_{B_r}\rho_\varepsilon^a u^2dz+\int_{B_r}\rho_\varepsilon^a |\nabla u|^2dz\right). 
$$
Further, when $N=1$, the above inequalities hold with $2^*$ replaced by any $p\in[1,+\infty)$ and a constant $C>0$ depending only on $N,a,p$.
\end{teo}
\begin{oss}\label{rem:IsoEll} It is worth mentioning that the theorem above (range $a \leq -1$) follows as a consequence of a fine analysis of the isometry
\begin{equation}\label{eq:IsoEll}
T_\varepsilon^a:\Tilde{H}^1(B_r,\rho^a_\varepsilon)\to \Tilde{H}^1(B_r) \quad\; u\rightarrow v := \sqrt{\rho_{\varepsilon}^{a}}u,
\end{equation}
where $\Tilde{H}^1(B_r)$ is the completion of $C_c^\infty(\overline{B}_r\setminus\Sigma)$ w.r.t. the norm 
$$Q_{\varepsilon}(v)=\int_{B_r} |\nabla v|^2+\int_{B_r} \left[\left( \frac{\partial_y\rho_{\varepsilon}^{a}}{2\rho_{\varepsilon}^{a}}\right)^2+\partial_y\left( \frac{\partial_y\rho_{\varepsilon}^{a}}{2\rho_{\varepsilon}^{a}}\right)\right]v^2-\int_{\partial B_r}\frac{\partial_y\rho_{\varepsilon}^{a}}{2\rho_{\varepsilon}^{a}}yv^2,$$
which turns out to be equivalent to the classical $H^1(B_r)$-norm, uniformly in $\e \in [0,1)$ (see \cite{SirTerVit21b}*{Lemma B.4}). This fact allows to apply the classical Sobolev inequality to $v = T_\varepsilon^a u$ and recover \eqref{eq:embedding:non:int} in terms of $u$.   
\end{oss}
\begin{oss}\label{rem:omega} Notice that the definitions and theorems above hold true replacing the ball with any open bounded domain $\Omega$, including the case of the half balls $B_r^+$. \end{oss}
\begin{oss}\label{rem:h01h01rhoe} 
Let $\Omega \subset \R^{N+1}$ be an open bounded set such that $\Omega \subset\subset \R^{N+1}\setminus\Sigma$ and, for every $\e \in [0,1)$, let $H_0^1(\Omega,\rho_\e^a)$ be the completion of $C_c^\infty(\Omega)$ w.r.t. the seminorm 
\[
\|u\|_{H_0^1(\Omega,\rho_\e^a)} := \left(\int_\Omega \rho_\e^a |\nabla u|^2 dz\right)^{{1}/{2}}.
\]
Then, for every $\e \in [0,1)$,
\begin{equation}\label{eq:EquivH10}
H_0^1(\Omega,\rho_\e^a) = H_0^1(\Omega).
\end{equation}
Indeed, $\dist(\Omega,\Sigma) \geq \delta$ for some $\delta > 0$ depending only on $\Omega$, and thus $\delta \leq \rho_\e^a \leq \delta^{-1}$ in $\R^{N+1}$ uniformly in $\e$, up to taking $\delta$ smaller. This shows that $\|\cdot\|_{H_0^1(\Omega,\rho_\e^a)} \sim \|\cdot\|_{H_0^1(\Omega)}$ which, in turn, readily implies \eqref{eq:EquivH10} by the definition of $H_0^1(\Omega)$ and $H_0^1(\Omega,\rho_\e^a)$. 
\end{oss}
\subsubsection{Sobolev spaces involving time.}\label{sobolev:spaces:time} Let $a\in\mathbb{R}$, $\varepsilon\in[0,1)$, $r>0$, $I_r := (-r^2,r^2)$ and $p,q \in (1,+\infty)$. We define
\[
L^q(I_r;L^p(B_r,\rho_\varepsilon^a)) := \{u : I_r \to L^p(B_r,\rho_\varepsilon^a) \text{ strongly measurable: } \int_{I_r} \|u(t)\|_{L^p(B_r,\rho_\varepsilon^a)}^q dt < +\infty \},
\]
equipped with the norm
\[
\|u\|_{L^q(I_r;L^p(B_r,\rho_\varepsilon^a))} := \left( \int_{I_r} \|u(t)\|_{L^p(B_r,\rho_\varepsilon^a)}^q dt \right)^{{1}/{q}}.
\]
The special case $p = q$ is the most relevant for the paper. In such case, we set $L^p(Q_r,\rho_\varepsilon^a) := L^p(I_r;L^p(B_r,\rho_\varepsilon^a))$ and
\[
\|u\|_{L^p(Q_r,\rho_\varepsilon^a)} := \|u\|_{L^p(I_r;L^p(B_r,\rho_\varepsilon^a))} = \left(\int_{Q_r}\rho_{\varepsilon}^a |u|^p dzdt \right)^{{1}/{p}}.
\]
Similarly, for fields we define
\[
L^q(I_r;L^p(B_r,\rho_\varepsilon^a)^{N+1}) := \{U : I_r \to L^p(B_r,\rho_\varepsilon^a )^{N+1} \text{ strongly measurable: } \int_{I_r} \|U(t)\|_{L^p(B_r,\rho_\varepsilon^a)^{N+1}}^q dt < +\infty \},
\]
normed by
\[
\|U\|_{L^q(I_r;L^p(B_r,\rho_\varepsilon^a)^{N+1})} := \left( \int_{I_r} \|U(t)\|_{L^p(B_r,\rho_\varepsilon^a)^{N+1}}^q dt \right)^{{1}/{q}}.
\]
As above, when $p = q$, we set $L^p(Q_r,\rho_\varepsilon^a)^{N+1} := L^p(I_r;L^p(B_r,\rho_\varepsilon^a)^{N+1})$ and
\[
\|U\|_{L^p(Q_r,\rho_\varepsilon^a)^{N+1}} := \|U\|_{L^p(I_r;L^p(B_r,\rho_\varepsilon^a )^{N+1})} = \left( \int_{Q_r}\rho_{\varepsilon}^a |U|^p dzdt \right)^{{{1}/{p}}}.
\]
We set
\[
L^\infty(I_r;L^p(B_r,\rho_\varepsilon^a)) := \{u : I_r \to L^p(B_r,\rho_\varepsilon^a) \text{ strongly measurable: } \esssup_{t\in I_r} \|u(t)\|_{L^p(B_r,\rho_\varepsilon^a)} < +\infty \},
\]
equipped with the norm
\[
\|u\|_{L^\infty(I_r;L^p(B_r,\rho_\varepsilon^a))} := \esssup_{t\in I_r} \|u(t)\|_{L^p(B_r,\rho_\varepsilon^a)},
\]
and 
\[
C(\overline{I}_r;L^p(B_r,\rho_\varepsilon^a)) := \{u : \overline{I}_r \to L^p(B_r,\rho_\varepsilon^a ) \text{ continuous: } \max_{t\in \overline{I}_r} \|u(t)\|_{L^p(B_r,\rho_\varepsilon^a )} < +\infty \},
\]
normed by
\[
\|u\|_{C(\overline{I}_r;L^p(B_r,\rho_\varepsilon^a ))} := \max_{t\in \overline{I}_r} \|u(t)\|_{L^p(B_r,\rho_\varepsilon^a )}.
\]
The space $L^2(I_r;H^1(B_r,\rho_{\varepsilon}^a))$ is defined as the completion of $C^\infty(\overline{Q}_r)$ w.r.t. the norm 
\begin{equation}\label{eq:L2H1Norm}
\|u\|_{L^2(I_r;H^1(B_r,\rho_{\varepsilon}^a ))} := \left(\int_{I_r}\|u(t)\|_{{H}^1(B_r,\rho_{\varepsilon}^a)}^2 dt  \right)^{{1}/{2}} = \left(\int_{Q_r}\rho_{\varepsilon}^a u^2dzdt + \int_{Q_r}\rho_{\varepsilon}^a|\nabla u|^2 dzdt \right)^{{1}/{2}},
\end{equation}
while $L^2(I_r;H_0^1(B_r,\rho_{\varepsilon}^a))$ is the completion of $C_c^\infty(Q_r) $ w.r.t. the seminorm 
$$
\|u\|_{L^2(I_r;H^1_0(B_r,\rho_{\varepsilon}^a))} := \left(\int_{I_r}\|u(t)\|_{H^1_0(B_r,\rho_{\varepsilon}^a)}^2 dt  \right)^{{1}/{2}}= \left(\int_{Q_r}\rho_{\varepsilon}^a|\nabla u|^2 dzdt \right)^{{1}/{2}}.
$$
Notice that by the Riesz's representation theorem, the topological dual space of $L^2(I_r;H^1_0(B_r,\rho_\varepsilon^a))$ satisfies
\[
L^2(I_r;H^1_0(B_r,\rho_\varepsilon^a))^\star = L^2(I_r;H^{-1}(B_r,\rho_\varepsilon^a)).
\]
\begin{oss}\label{remark 2.3}
Later on in the paper we will use the following classical fact, see \cite{lions}*{Proposition 2.1, Theorem 3.1}: there exists $C > 0$ depending only on $r$, such that
\[
\|u\|_{C(\bar{I}_r;L^2(B_r,\rho_\varepsilon^a))} \leq
  C \big(  \|u\|_{L^2(\overline{I}_r;H^1_0(B_r,\rho_\varepsilon^a ))} + \|\partial_t u\|_{L^2(I_r;H^{-1}(B_r,\rho_\varepsilon^a ))}  \big),
\]
where $\partial_t u$ denotes the weak time derivative of $u$. That is, if $u \in L^2(I_r;H^1_0(B_r,\rho_\varepsilon^a))$ and $\partial_t u \in L^2(I_r;H^{-1}(B_r,\rho_\varepsilon^a ))$, then $u \in C(\overline{I}_r; L^2(B_r,\rho_\varepsilon^a))$.
\end{oss}
Exploiting the Sobolev inequalities above, one can prove their ``parabolic versions'' (see for instance \cite{ChiSer85}, or the more recent \cite{audrito}).
\begin{teo}
Let $\varepsilon\in[0,1)$ and $r\in[1/2,1]$. Assume either $a > -1$ and $N\ge2$, or $a>0$ and $N=1$. Then there exists $C > 0$ depending only on $N$ and $a$ such that
\begin{equation}\label{audrito A.3}
\int_{Q_r}\rho_\varepsilon^a |u|^{2\gamma}dzdt \le C \left(\int_{Q_r}\rho_\varepsilon^a \left(u^2 +|\nabla u|^2\right)dzdt\right) \esssup_{t\in I_r }\left(\int_{Q_r}\rho_\varepsilon^a u^2 dzdt\right)^{\gamma-1},
\end{equation}
for every $u\in L^2(I_r;H^1(B_r,\rho_\varepsilon^a))$, where 
$$
\gamma := 2 \cdot \frac{2^*_a-1}{2^*_a}.
$$
Further, if $N=1$ and $a\in(-1,0]$, the above inequality holds with $\gamma$ replaced with any $p\in[1,2)$ and a constant $C>0$ depending only on $N$, $a$ and $p$.    
\end{teo}
The same isometry in \eqref{eq:IsoEll} can be easily extended to the parabolic setting as a map
\begin{equation}\label{eq:IsoParabolic}
\overline{T}_\varepsilon^a : L^2(I_r;\Tilde{H}^1(B_r,\rho_\varepsilon)) \to L^2(I_r;\Tilde{H}^1(B_r))\quad\; u\rightarrow v := \sqrt{\rho_{\varepsilon}^{a}}u.
\end{equation}
Notice that $\overline{T}_\varepsilon^a$ is still an isometry if $L^2(I_r;\Tilde{H}^1(B_r))$ is normed by
$$
\|v\|_{L^2(I_r;\Tilde{H}^1(B_r))} = \left(\int_{I_r} Q_\varepsilon(v(t)) dt\right)^{ 1/2},
$$
and $L^2(I_r;\Tilde{H}^1(B_r,\rho_\varepsilon))$ stands for the completion of $C_c^\infty(\overline{Q}_r\setminus\Sigma)$ w.r.t. \eqref{eq:L2H1Norm}. Working as in the stationary (time-independent) framework, one can see that such a norm is equivalent to the standard $L^2(I_r;H^1(B_r))$-norm, uniformly in $\varepsilon\in[0,1)$. As a consequence, we obtain the following Sobolev embeddings when $a\le-1$.
\begin{teo}
    Let $a\le-1$, $N\ge2$, $\varepsilon\in[0,1)$, $r\in[1/2,1]$. Then there exists $C>0$ depending only on $N$ and $a$ such that
    \begin{equation*}\label{audrito A.3-a<-1}
        \int_{Q_r}\rho_\varepsilon^a|u|^{2\gamma}dzdt\le C\left(
        \int_{Q_r} \rho_\varepsilon^a\left(u^2 +|\nabla u|^2\right)dzdt
        \right)\esssup_{t\in I_r}\left(\int_{Q_r}\rho_\varepsilon^a u^2dzdt\right)^{\gamma-1},
    \end{equation*}
for every $u\in L^2(I_r;\tilde{H}^1(B_r,\rho_\varepsilon^a))$, where 
$$
\gamma := 2 \cdot\frac{2^*-1}{2^*}.
$$
Further, if $N=1$, the above inequality holds with $\gamma$ replaced with any $p\in[1,2)$ and a constant $C>0$ depending only on $N$, $a$ and $p$.
\end{teo}
\begin{oss} As explained in Remark \ref{rem:omega}, one can define the spaces $L^q(I_r;L^p(B_r^+,\rho_\varepsilon^a))$, $L^q(I_r;L^p(B_r^+,\rho_\varepsilon^a)^{N+1})$ and $C(\overline{I}_r;L^p(B_r^+,\rho_\varepsilon^a))$, and the Sobolev spaces $L^2(I_r;H^1(B_r^+,\rho_{\varepsilon}^a))$ and $L^2(I_r;\tilde{H}^1(B_r^+,\rho_{\varepsilon}^a))$.
\end{oss}
%
%
%
%
%
%%%%%%%%%%%%%%%%%%%%%%%%%%%%%%%%%%%%%%%%%%%%%%%%%%%%%%%%%%%%%%%%%%%%%%%%%%%%%%%%%%%%%%%%%%%%%%%%%%%%%%%%%%%%%%%%%%%%%%%%%%%%%%%%%%%%%%%%%%%%%%%%%%%%%%%%%%%%%%%%%%%%%%
%
%
%
%%%%%%%%%%%%%%%%%%%%%%%%%%%%%%%%%%%%%%%%%%%%%%%%%%%%%%%%%%%%%%%%%%%%%%%%%%%%%%%%%%%%%%%%%%%%%%%%%%%%%%%%%%%%%%%%%%%%%%%%%%%%%%%%%%%%%%%%%%%%%%%%%%%%%%%%%%%%%%%%
%
\subsection{Parabolic Hölder spaces}\label{section:parabolic:holder;spaces}
In this section we recall the definitions of the H\"older spaces of parabolic type we use later on in the paper. We follow \cite{Lie96}*{Chapter 4} (see also \cite{LSU}*{Chapter 1}).

Let $\Omega\subset\mathbb{R}^{N+1}\times\R$ be an open subset and $u:\Omega\to\mathbb{R}$. The parabolic distance $d_p:\Omega\times\Omega\to\mathbb{R}$ is defined by
\begin{equation}\label{eq:par:dist}
d_p((z,t),(\zeta,\tau)) := (|z-\zeta|^2+|t-\tau|)^{1/2}, 
\end{equation}
for all $(z,t),(\zeta,\tau)\in\Omega$, where $z,\zeta\in\mathbb{R}^{N+1}$, $t,\tau\in\mathbb{R}$. Notice that $d_p$ is parabolically $1$-homogeneous, in the sense that 
$$
d_p((rz,r^2t),(r\zeta,r^2\tau))=rd_p((z,t),(\zeta,\tau)), \quad \forall r \in \mathbb{R}.
$$
For $\alpha\in(0,1]$, we define the seminorms
$$
[u]_{C^{0,\alpha}_p(\Omega)} := \sup_{\substack{(z,t),(\zeta,\tau) \in\Omega \\ (z,t)\not=(\zeta,\tau)}}\frac{|u(z,t)-u(\zeta,\tau)|}{(|z-\zeta|^2+|t-\tau|)^{\alpha/2}}, \qquad [u]_{C^{\alpha}_t(\Omega)} := \sup_{\substack{(z,t),(z,\tau) \in\Omega \\ t\not=\tau}}\frac{|u(z,t)-u(z,\tau)|}{|t-\tau|^\alpha},
$$
and
$$
[u]_{C^{1,\alpha}_p(\Omega)} := \sum_{i=1}^{N+1}[\partial_i u]_{C^{0,\alpha}_p(\Omega)} + [u]_{C^{\frac{1+\alpha}{2}}_t(\Omega)}.
$$
We also define the norms
\[
\|u\|_{C^{0,\alpha}_p(\Omega)} := \|u\|_{L^\infty(\Omega)} + [u]_{C^{0,\alpha}_p(\Omega)}, \qquad \|u\|_{C^{1,\alpha}_p(\Omega)} := \|u\|_{L^\infty(\Omega)} + \| \nabla u \|_{L^\infty(\Omega)} + [u]_{C^{1,\alpha}_p(\Omega)},
\]
and the spaces
\[
C^{0,\alpha}_p(\Omega) := \{u:\Omega \to \R: \|u\|_{C^{0,\alpha}_p(\Omega)} < +\infty \}, \qquad C^{1,\alpha}_p(\Omega) := \{u:\Omega \to \R: \|u\|_{C^{1,\alpha}_p(\Omega)} < +\infty \}.
\]
More generally, if $\beta\in\mathbb{N}^{N+1}$ is a multi-index, $\alpha\in(0,1]$ and $k \geq 2$, we define the seminorm 
$$
[u]_{C_p^{k,\alpha}(\Omega)}:=\sum_{|\beta|+2j=k}[\partial_x^\beta\partial_t^j u]_{C^{0,\alpha}_p(\Omega)}+\sum_{|\beta|+2j=k-1}[\partial_x^\beta\partial_t^j u]_{C^{\frac{1+\alpha}{2}}_t(\Omega)},
$$
the norm
$$
\|u\|_{{C_p^{k,\alpha}(\Omega)}}=\sum_{|\beta|+2j\le k}\sup_{\Omega}|\partial_x^\beta\partial_t^j u|+[u]_{C_p^{k,\alpha}(\Omega)},
$$
and the space
\[
C^{k,\alpha}_p(\Omega) := \{u:\Omega \to \R: \|u\|_{C^{k,\alpha}_p(\Omega)} < +\infty \}.
\]
%

%
%%%%%%%%%%%%%%%%%%%%%%%%%%%%%%%%%%%%%%%%%%%%%%%%%%%%%%%%%%%%%%%%%%%%%%%%%%%%%%%%%%%%%%%%%%%%%%%%%%%%%%%%%%%%%%%%%%%%%%%%%%%%%%%%%%%%%%%%%%%%%%%%%%%%%%%%%%%%%%%%%
%
\subsection{Weak solutions}
The energy spaces introduced above allow us to give the notion of weak solutions for our class of problems. Before that, we introduce the space of test functions we will use in the definitions below: such space takes into account the integrability/non-integrability of the weight $|y|^a$ when $\e = 0$.
\begin{defi}\label{def:TestF}
Let $a\in\mathbb{R}$, $N\ge1$, $r>0$ and $\e \in [0,1)$. We define
\[
\mathcal{D}_c^\infty(Q_r) :=
\begin{cases}
C_c^\infty(Q_r) \quad &\text{ if either } \e \in (0,1), \text{ or } \e = 0 \text{ and } a \in (-1,1) \\
C_c^\infty(Q_r\setminus\Sigma) \quad &\text{ if } \e = 0 \text{ and } a \in (-\infty,1]\cup[1,+\infty).
\end{cases}
\]
Notice that, in light of Proposition \ref{prop:CharH1a}, $\mathcal{D}_c^\infty(Q_r)$ is dense in $L^2(I_r;H_0^1(B_r,\rho_\varepsilon^a ))$ for every $\e \in [0,1)$. 
\end{defi}
\begin{defi}\label{def.solution}
Let $a\in\mathbb{R}$, $N\ge1$, $r > 0$, $\varepsilon\in[0,1)$ and $f \in L^2(Q_r,\rho_\e^a )$, $F \in L^2(Q_r,\rho_\e^a )^{N+1}$. We say that $u$ is a weak solution to
\begin{equation}\label{solution}
    \rho_\varepsilon^a \partial_t u - \div (\rho_\varepsilon^{a} A\nabla u) = \rho_\varepsilon^a f + \div(\rho_\varepsilon^{a} F) \quad \text{ in }  Q_r,
\end{equation}
if $u \in L^2(I_r;H^1(B_r,\rho_\varepsilon^a ))\cap L^\infty(I_r;L^2(B_r,\rho_\varepsilon^a))$ and satisfies
\begin{equation}\label{eq.solution}
    -\int_{Q_r}\rho_\varepsilon^a u \partial_t\phi dzdt +\int_{Q_r}\rho_\varepsilon^a A\nabla u\cdot\nabla\phi dzdt = \int_{Q_r}\rho_\varepsilon^a( f\phi -F\cdot\nabla\phi )dzdt,
\end{equation}
for every $\phi\in \mathcal{D}_c^\infty(Q_r)$. We say that $u$ is an entire solution to 
\[
\rho_\varepsilon^a \partial_t u - \div (\rho_\varepsilon^{a} A\nabla u) = \rho_\varepsilon^a f + \div(\rho_\varepsilon^{a} F) \quad \text{ in }  \R^{N+1}\times\R,
\]
if, for every $r > 0$, $u$ is a weak solution to \eqref{solution}.
\end{defi}
\begin{defi}\label{def.solution1}
Let $a\in\mathbb{R}$, $N\ge1$, $r > 0$, $\varepsilon\in[0,1)$ and $f \in L^2(Q_r,\rho_\e^a)$, $F \in L^2(Q_r,\rho_\e^a)^{N+1}$, $u_0 \in L^2(B_r,\rho_\e^a)$. We say that $u$ is a weak solution to
\begin{equation}\label{solution1}
\begin{cases}
\rho_\varepsilon^a \partial_t u - \div (\rho_\varepsilon^{a} A\nabla u) = \rho_\varepsilon^a f + \div(\rho_\varepsilon^{a} F) \quad &\text{ in }  Q_r \\
u = 0    \quad &\text{ in }  \partial B_r\times I_r \\
u = u_0   \quad &\text{ in }  B_r\times\{-r^2\},
\end{cases}
\end{equation}
if $u \in L^2(I_r;H_0^1(B_r,\rho_\varepsilon^a))\cap L^\infty(I_r;L^2(B_r,\rho_\varepsilon^a ))$, satisfies \eqref{eq.solution} for every $\phi\in \mathcal{D}_c^\infty(Q_r)$ and $u(-r^2) = u_0$ in $L^2(B_r,\rho_\e^a)$.
\end{defi}
\begin{oss}\label{remark 2.3 bis} Let $\e \in [0,1)$ and let $u$ be a weak solution to \eqref{solution1}. Then, by the H\"older inequality, \eqref{eq:UnifEll} and the Poincaré inequality (for the degenerate/singular case we refer to \cite{JeoVit23}*{Lemma 3.2}, \cite{SirTerVit21b}*{Lemma B.5} and \cite{FabKenSer82}*{Theorem 1.3}), we have 
\[
-\int_{Q_r}\rho_\varepsilon^a u \partial_t\phi dzdt \leq C \big( \Lambda \|u\|_{L^2(I_r;H_0^1(B_r,\rho_\varepsilon^a))} + \|f\|_{L^2(Q_r,\rho_\varepsilon^a )} + \|F\|_{L^2(Q_r,\rho_\varepsilon^a )^{N+1}}  \big)  \|\phi\|_{L^2(I_r;H_0^1(B_r,\rho_\varepsilon^a ))}
\]
for every $\phi \in \mathcal{D}_c^\infty(Q_r)$, for some $C > 0$ depending on $N$, $a$ and $\e$.
Consequently, a standard density argument, shows that the distribution
\[
\langle \partial_tu,\phi \rangle := -\int_{Q_r}\rho_\varepsilon^a  u \partial_t\phi dzdt, \quad \phi \in L^2(I_r;H_0^1(B_r,\rho_\varepsilon^a))
\]
is well-defined and $\partial_t u\in L^2(I_r;H^{-1}(B_r,\rho_\varepsilon^a))$. In particular, $u\in C(\overline{I}_r;L^2(B_r,\rho_\varepsilon^a))$ by Remark \ref{remark 2.3} and thus the equation $u(-r^2) = u_0$ in $L^2(B_r,\rho_\e^a)$ makes sense. 
\end{oss}
\begin{defi}\label{def.solution-mezza-palla}
Let $a>-1$, $N\ge1$, $r > 0$, $\varepsilon\in[0,1)$ and $f \in L^2(Q_r^+,\rho_\e^a )$, $F \in L^2(Q_r^+,\rho_\e^a )^{N+1}$. We say that $u$ is a weak solution to
\begin{equation}\label{solution-mezza-palla}
\begin{cases}
\rho_\varepsilon^a \partial_t u - \div (\rho_\varepsilon^{a} A\nabla u) = \rho_\varepsilon^a f + \div(\rho_\varepsilon^{a} F) \quad &\text{ in }  Q_r^+\\
\rho_\varepsilon^a\left(A\nabla u+F\right)\cdot e_{N+1} = 0 \quad &\text{ in }\partial^0Q_r^+,
    \end{cases}
\end{equation}
if $u \in L^2(I_r;H^1(B_r^+,\rho_\varepsilon^a)) \cap L^\infty(I_r;L^2(B_r^+,\rho_\varepsilon^a))$ and satisfies
\[
-\int_{Q_r^+}\rho_\varepsilon^a u\partial_t\phi dzdt + \int_{Q_r^+}\rho_\varepsilon^a A\nabla u\cdot\nabla\phi dzdt = \int_{Q_r^+}\rho_\varepsilon^a( f\phi -F\cdot\nabla\phi )dzdt,
\]
for every $\phi\in C_c^\infty(Q_r)$. We say that $u$ is an entire solution to
\[
\begin{cases}
\rho_\varepsilon^a \partial_t u - \div (\rho_\varepsilon^{a} A\nabla u) = \rho_\varepsilon^a f + \div(\rho_\varepsilon^{a} F) \quad &\text{ in }  \R_+^{N+1}\times\R\\
\rho_\varepsilon^a\left(A\nabla u+F\right)\cdot e_{N+1} = 0 \quad &\text{ in }\partial \R_+^{N+1}\times\R,
\end{cases}
\]
if, for every $r > 0$, $u$ is a weak solution to \eqref{solution-mezza-palla}.
\end{defi}
\begin{oss}\label{rem:Stekelov}
A key tool in the study of weak solutions are the Steklov averages, defined as
$$
u_h(z,t) := \frac{1}{h}\int_t^{t+h}u(z,s)ds, \qquad  u_{-h}(z,t) := \frac{1}{h} \int_{t-h}^t u(z,s)ds,
$$
where $h > 0$ and $u$ is a given function. It is well-known that if $\e \in (0,1)$, $u \in L^2(I_r;H^1(B_r,\rho_\e^a))$ and $\delta > 0$, then
\[
u_h \to u, \quad \nabla u_h \to \nabla u \qquad \text{in } L^2(B_r \times (-r^2,r^2-\delta),\rho_\e^a),
\]
as $h \to 0$ (see for instance \cite{Lie96}*{Lemma 3.2 and Lemma 3.3}). Furthermore, if $u$ is a weak solution to \eqref{solution}, then $u_h$ satisfies
\begin{equation}\label{solution.steklov}
    \int_{Q_r}\rho_\varepsilon^a (\partial_t u_h \phi + (A\nabla u)_h\cdot\nabla\phi)dzdt = \int_{Q_r}\rho_\varepsilon^a (f_h\phi-F_h\cdot\nabla\phi)dzdt,
\end{equation}
for every $\phi\in C_c^\infty(B_r\times(-r^2,r^2-h))$: the proof is a standard adaptation of the classical framework (see for instance \cite{Lie96}*{Theorem 6.1}). Similar for the case $\e = 0$ and for weak solutions to \eqref{solution1} or \eqref{solution-mezza-palla}. 
We quote \cite{LSU}, \cite{Lie96} and the more recent \cite{steklov average} for further properties of Steklov averages.
\end{oss}    
%
%
%

%
%%%%%%%%%%%%%%%%%%%%%%%%%%%%%%%%%%%%%%%%%%%%%%%%%%%%%%%%%%%%%%%%%%%%%%%%%%%%%%%%%%%%%%%%%%%%%%%%%%%%%%%%%%%%%%%%%%%%%%%%%%%%%%%%%%%%%%%%%%%%%%%%%%%%%%%%%%%%%%%%%%%%%%
%
%%%%%%%%%%%%%%%%%%%%%%%%%%%%%%%%%%%%%%%%%%%%%%%%%%%%%%%%%%%%%%%%%%%%%%%%%%%%%%%%%%%%%%%%%%%%%%%%%%%%%%%%%%%%%%%%%%%%%%%%%%%%%%%%%%%%%%%%%%%%%%%%%%%%%%%%%%%%%%%%%%%%%%
%
\section{Local boundedness of weak solutions}\label{section3}
In this section we prove a local $L^2 \to L^\infty$ estimate for weak solutions to \eqref{solution} using a De Giorgi-Nash-Moser iteration. Analogous statements have been previously obtained in \cites{moser64,ChiSer85,vasseur,banerjee,audrito}. However, our setting is slightly more general and, even though the proof is quite standard, we present it for completeness.
\begin{pro}\label{moser-3.5}
Let $a\in\mathbb{R}$, $N\ge1$, $\varepsilon\in[0,1)$, $p>\frac{N+a^++3}{2}$, $q>N+a^++3$ and $A$ satisfying \eqref{eq:UnifEll}. Let $f\in L^p(Q_1,\rho_\varepsilon^a)$ and $F\in L^q(Q_1,\rho_\varepsilon^a)^{N+1}$ and let $u$ be a weak solution to \eqref{solution}. Then there exists $C > 0$ depending only on $N$, $a$, $\lambda$, $\Lambda$, $p$ and $q$ such that 
$$
\|u\|_{L^\infty(Q_{1/2})}\le C\left( \|u\|_{L^2(Q_1,\rho_\varepsilon^a)}+\|f\|_{L^p(Q_1,\rho_\varepsilon^a)}+\|F\|_{L^q(Q_1,\rho_\varepsilon^a)}\right).
$$
\end{pro}
The proof of Proposition \ref{moser-3.5} will be obtained in a couple of steps. The first one is a Cacciopoli-type inequality.
\begin{lem}\label{caccioppoli}
Let $a\in\mathbb{R}$, $N\ge1$, $\varepsilon\in[0,1)$, $p,q\ge2$ and $A$ satisfying \eqref{eq:UnifEll}. Let $f\in L^p(Q_1,\rho_\varepsilon^a)$ and $F\in L^q(Q_1,\rho_\varepsilon^a)^{N+1}$ and let $u$ be a weak solution to \eqref{solution}. Then there exists $C > 0$ depending only on $N$, $a$, $\lambda$ and $\Lambda$ such that for every $\frac{1}{2}\le r' < r < 1$ there holds
\begin{align}\label{caccioppoli.inequality}
    \begin{aligned}
        &\esssup_{t\in(-r'^2,r'^2)}\int_{B_{r'}}\rho_\varepsilon^a u^2+\int_{Q_{r'}}\rho_\varepsilon^a |\nabla u|^2\\
        & \qquad \le C\left[    \frac{1}{(r-r')^2}\int_{Q_r}\rho_\varepsilon^a u^2+\|f\|_{L^p(Q_r,\rho_\varepsilon^a)}\|u\|_{L^{p'}(Q_r,\rho_\varepsilon^a)}    +\int_{Q_r}\rho_\varepsilon^a |F|^2\rchi_{\{|u|>0\}}\right].
    \end{aligned}
\end{align}
Moreover, for every $k\in\mathbb{R}$ and for functions of the form $v:=(u-k)_+=\max\{u-k,0\}$ and $v:=(u-k)_-=\max\{-u+k,0\}$, the following inequality holds
\begin{align}\label{caccioppoli.inequality-subsolutions}
    \begin{aligned}
        &\esssup_{t\in(-r'^2,r'^2)}\int_{B_{r'}}\rho_\varepsilon^a v^2+\int_{Q_{r'}}\rho_\varepsilon^a |\nabla v|^2\\
        & \qquad \le C\left[    \frac{1}{(r-r')^2}\int_{Q_r}\rho_\varepsilon^a v^2+\|f\|_{L^p(Q_r,\rho_\varepsilon^a)}\|v\|_{L^{p'}(Q_r,\rho_\varepsilon^a)}    +\int_{Q_r}\rho_\varepsilon^a |F|^2\rchi_{\{v>0\}}\right].
    \end{aligned}
\end{align}
\end{lem}
\begin{proof} To simplify the notation, let $\rho=\rho_\varepsilon^a$. As in \cite{banerjee}, we may work with the Steklov average $u_h$ of $u$ and later take the limit as $h \to 0$: equivalently, we may assume that $\partial_t u \in L^2(Q_1,\rho)$ and directly work with $u$ which is what we do next.

Fix $\frac{1}{2}\le r'<r<1$. We test the equation of $u$ with $\eta^2 u$, where $\eta$ is a smooth cut-off function we will define later. Then:
\begin{align*}
\int_{Q_1}\rho\Big(   \frac{1}{2}\partial_t(u^2\eta^2)+\eta^2A\nabla u\cdot \nabla u     \Big)= \int_{Q_1}\rho\Big(      \frac{1}{2}u^2\partial_t(\eta^2)-2\eta u A\nabla u\cdot\nabla\eta    + f\eta^2u-\eta^2F\cdot\nabla u-2\eta uF\cdot\nabla \eta    \Big).
\end{align*}
By \eqref{eq:UnifEll}, the Hölder's and the Young's inequalities, we get
\begin{align*}
& \frac{1}{2}\int_{Q_1}\rho \partial_t(u^2\eta^2)+\lambda \int_{Q_1}\rho  \eta^2|\nabla u|^2\\
&\le \frac{1}{2}\int_{Q_1}\rho      u^2\partial_t(\eta^2) + 2\Lambda \Big(\int_{Q_1}\rho  \eta^2 |\nabla u|^2   \Big)^{{1}/{2}} \Big(\int_{Q_1}\rho  u^2  |\nabla \eta|^2      \Big)^{{1}/{2}} +\|\eta f\|_{L^p(Q_1,\rho)}\|\eta u\|_{L^{p'}(Q_1,\rho)} \\
&+ \Big(\int_{Q_1}\rho  \eta^2 |F|^2 \rchi_{\{|u|>0\}}  \Big)^{{1}/{2}}  \Big(\int_{Q_1}\rho  \eta^2 |\nabla u|^2   \Big)^{{1}/{2}}  + 2\Big(\int_{Q_1}\rho  \eta^2 |F|^2 \rchi_{\{|u|>0\}}  \Big)^{{1}/{2}} \Big(\int_{Q_1}\rho  u^2 |\nabla \eta |^2   \Big)^{{1}/{2}} \\
&\le \frac{1}{2}\int_{Q_1}\rho      u^2\partial_t(\eta^2) +\frac{\lambda}{3}\int_{Q_1}\rho\eta^2|\nabla u|^2+\frac{3\Lambda^2}{\lambda}\int_{Q_1}\rho u^2|\nabla \eta|^2+\|\eta f\|_{L^p(Q_1,\rho)}\|\eta u\|_{L^{p'}(Q_1,\rho)}\\
&+ \frac{1}{2\lambda}\int_{Q_1}\rho  \eta^2 |F|^2 \rchi_{\{|u|>0\}}  +  \frac{\lambda}{2}\int_{Q_1}\rho  \eta^2 |\nabla u|^2 + \int_{Q_1}\rho  \eta^2 |F|^2 \rchi_{\{|u|>0\}}+  \int_{Q_1}\rho u^2|\nabla \eta|^2.
\end{align*}
Hence, we have
\begin{align}\label{caccio1}
\begin{aligned}
&\frac{1}{2}\int_{Q_1}\rho \partial_t(u^2\eta^2)+\frac{1}{6}\lambda \int_{Q_1}\rho  \eta^2|\nabla u|^2\\
&\le \frac{1}{2}\int_{Q_1}\rho      u^2\partial_t(\eta^2)+\Big(\frac{3\Lambda^2}{\lambda}+1\Big)\int_{Q_1}\rho u^2|\nabla \eta|^2+\|\eta f\|_{L^p(Q_1,\rho)}\|\eta u\|_{L^{p'}(Q_1,\rho)}+\Big(1+\frac{1}{2\lambda}\Big)\int_{Q_1}\rho  \eta^2 |F|^2 \rchi_{\{|u|>0\}}.
\end{aligned}
\end{align}
A standard approximation technique (see  \cite{moser64}, \cite{banerjee}*{Theorem 5.1}) allows us to integrate over a cylinder of the form $B_1\times(-1,\hat{t})$, for any $\hat{t}\in(-1,1)$. Now, let $t^*\in(-r',r')$ such that 
  \begin{equation}\label{t^*}
      \frac{1}{2}\esssup_{t\in(-r'^2,r'^2)}\int_{B_{r'}}\rho u^2\le \int_{B_r'}\rho u^2(t^*),
  \end{equation}  
and take the test function $\eta = \psi_1(|z|)\psi_2(t)$, where
  \begin{equation}\label{test1}
      \psi_1\equiv1 \text{ in }B_{r'},\hspace{0.3cm} 0\le\psi_1\le1 \text{ in } B_1,\hspace{0.3cm} \supp(\psi_1)=B_r, \hspace{0.3cm} |\nabla \psi_1|\le\frac{C}{r-r'},
  \end{equation}
  \begin{align}\label{test2}
      \begin{aligned}
          &\psi_2\equiv1 \text{ in }(-r'^2,t^*),\hspace{0.3cm} 0\le\psi_2\le1 \text{ in } (-1,t^*),\\
          &\supp(\psi_2)=(-r^2,t^*), \hspace{0.3cm} |\partial_t \psi_2|\le\frac{C}{(r-r')^2}.
      \end{aligned}
  \end{align}
By \eqref{test1} and \eqref{test2} we have
$$
\partial_t(\eta^2)+|\nabla\eta|^2\le\frac{C}{(r-r')^2},
$$
and thus by \eqref{caccio1}, \eqref{t^*} and the last inequality we obtain
\begin{align*}
    \esssup_{t\in(-r'^2,r'^2)}\int_{B_{r'}}\rho u^2 \le C\left[\frac{1}{(r-r')^2}\int_{Q_r}\rho_\varepsilon^a u^2+\|f\|_{L^p(Q_r,\rho_\varepsilon^a)}\|u\|_{L^{p'}(Q_r,\rho_\varepsilon^a)}    +\int_{Q_r}\rho |F|^2\rchi_{\{|u|>0\}}\right].
\end{align*}
Combining this inequality with \eqref{caccio1}, then \eqref{caccioppoli.inequality} follows.

To prove \eqref{caccioppoli.inequality-subsolutions}, let $v=(u-k)_+$ and test the equation of $u$ with $\eta^2 v$. Since $\partial_t u = \partial_t v$ and $\nabla u = \nabla v$  on $\{v>0\}$, we obtain
\begin{align*}
&\int_{Q_1}\rho\left(\partial_t u (\eta^2 v) +\nabla u\cdot\nabla(\eta^2v)-f(\eta^2v)+F\cdot\nabla(\eta^2 v)\right) \\
&=\int_{Q_1}\rho\left(\partial_t v (\eta^2 v) +\nabla v\cdot\nabla(\eta^2v)-f(\eta^2v)+F\cdot\nabla(\eta^2 v)\right),
\end{align*}
and thus,  \eqref{caccioppoli.inequality-subsolutions} follows from the same argument above. The case $v = (u-k)_-$ is analogue, noticing that $\partial_t u=-\partial_t v$ and $\nabla u=-\nabla v$.
\end{proof}
The second step is to establish a ``no-spikes'' estimate type.
\begin{lem}\label{no-spike}
    Let $N\ge1$,  $\varepsilon\in[0,1)$, $a\in\mathbb{R}$, $p>\frac{N+a^++3}{2}$, $q>N+a^++3$ and $A$ satisfying \eqref{eq:UnifEll}. Then there exists a constant $\delta\in(0,1)$, which dependens on $N$,  $a$,  $\lambda$, $\Lambda$, $p$ and $q$, such that if 
    $$\|f\|_{L^p(Q_1,\rho_\varepsilon^a)}+\|F\|_{L^q(Q_1,\rho_\varepsilon^a)}\le 1,$$
    and $u$ is a weak solution to \eqref{solution} with 
    $$\int_{Q_1}\rho_\varepsilon^a(u_+)^2dzdt\le\delta,$$
    then
    $$u\le1\hspace{0.3cm}\text{in }Q_{1/2}.$$
\end{lem}
\begin{proof}
    Fix $\varepsilon\in[0,1)$, set $\rho =\rho_\varepsilon^a$, and assume either $N\ge2$ and $a>-1$, or $N=1$ and $a>0$ (the other cases are analogous). For every integer $j\ge0$ define 
    $$
    C_j:=1-2^{-j},\hspace{0.5cm}r_j:=\frac{1}{2}+2^{-j-1},\hspace{0.5cm}B_j:=B_{r_j},\hspace{0.5cm} Q_j:=Q_{r_j}.
    $$
    Notice that $C_j\downarrow 1$, $r_j\uparrow \frac{1}{2}$ as $j\to+\infty$, and $r_j-r_{j+1}=2^{-j-2}$. 
    Define 
    $$
    V_j:=(u-C_j)_+, \qquad E_j:=\int_{Q_j}\rho V_j^2dzdt,
    $$
and observe that, for every $j\ge0$, $E_j\le E_0\le\delta$ by assumption. Applying the Caccioppoli inequality \eqref{caccioppoli.inequality-subsolutions} to $V_{j+1}$, with $r'=r_{j+1}$ and $r=r_j$ we have
    \begin{align*}
        \esssup_{t\in({-r^2_{j+1}},{r^2_{j+1}})}\int_{B_{{j+1}}}\rho V_{j+1}^2+\int_{Q_{{j+1}}}\rho |\nabla V_{j+1}|^2
        \le C\left[    2^{2j}\int_{Q_j}\rho V_{j+1}^2+\|V_{j+1}\|_{L^{p'}(Q_j,\rho)}    +\int_{Q_r}\rho |F|^2\rchi_{\{V_{j+1}>0\}}\right].
    \end{align*}
    Consequently, by the Sobolev embedding \eqref{audrito A.3} (with $\gamma = 1+\frac{2}{N+1+a^+}$),
    \begin{align}\label{2gamma}  
    \begin{aligned}
    \left(\int_{Q_{j+1}}\rho|V_{j+1}|^{2\gamma}\right)^{1/\gamma} &\le C\left(\int_{Q_{j+1}}\rho V_{j+1}^2+\int_{Q_{j+1}}\rho |\nabla V_{j+1}|^2)\right)^{1/\gamma}\esssup_{t\in(-{r^2_{j+1}},{r^2_{j+1}})}\left(\int_{B_{{j+1}}}\rho V_{j+1}^2\right)^{(\gamma-1)/\gamma}\\
        &\le C\left[    2^{2j}\int_{Q_j}\rho V_{j+1}^2+\|V_{j+1}\|_{L^{p'}(Q_j,\rho)}    +\int_{Q_r}\rho |F|^2\rchi_{\{V_{j+1}>0\}}\right].    
    \end{aligned}
    \end{align}
    Now, by the H\"older inequality 
    \begin{equation}\label{2-2gamma}
        E_{j+1}=\int_{Q_{j+1}}\rho V_{j+1}^2\le\left(\int_{Q_{j+1}}\rho|V_{j+1}|^{2\gamma}\right)^{1/\gamma}\left(\int_{Q_{j+1}}\rho\rchi_{\{V_{j+1}>0\}}\right)^{1/\gamma'},
    \end{equation}
    where $\gamma'=\frac{N+3+a^+}{2}$ is the conjugate exponent of $\gamma$ and, using the H\"older inequality again, we obtain
    \begin{align}\label{Fq-||}
        \begin{aligned}
            \|V_{j+1}\|_{L^{p'}}(Q_j,\rho)\le\left(\int_{Q_{j+1}}\rho V_{j+1}^2\right)^{1/2}\left(\int_{Q_{j+1}}\rho \rchi_{\{V_{j+1}>0\}}\right)^{(p-2)/2p} \le E_j^{1/2}\left(\int_{Q_{j+1}}\rho \rchi_{\{V_{j+1}>0\}}\right)^{(p-2)/2p}
        \end{aligned}
    \end{align}
    and
    \begin{align}\label{fq-||}
        \begin{aligned}
            \int_{Q_{j+1}}\rho|F|^2\rchi_{\{V_{j+1}>0\}} \le \left(\int_{Q_{j+1}}\rho|F|^q\right)^{2/q}\left(\int_{Q_{j+1}}\rho\rchi_{\{V_{j+1}>0\}}\right)^{(q-2)/q}
            \le\left(\int_{Q_{j+1}}\rho\rchi_{\{V_{j+1}>0\}}\right)^{(q-2)/q}.
        \end{aligned}
    \end{align}
    Further, using the definition of $V_j$, it is easy to see that $V_{j+1}>0$ if and only if $V_j>2^{-j-1}$, for every $j$ and thus
    \begin{equation}\label{|...|}
        \int_{Q_{j+1}}\rho\rchi_{\{V_{j+1}>0\}}=\int_{Q_{j+1}}\rho\rchi_{\{V^2_{j}>2^{-2j-2}\}}\le 2^{2j+2}\int_{Q_j}\rho V_j^2=E_j.
    \end{equation}
    Combining \eqref{2gamma}, \eqref{2-2gamma}, \eqref{Fq-||}, \eqref{fq-||} and \eqref{|...|}, we obtain
    $$E_{j+1}\le C^{1+j}\left(E_j^{1+\frac{1}{\gamma'}}+E_j^{1-\frac{1}{p}+\frac{1}{\gamma'}}+E_j^{1-\frac{2}{q}+\frac{1}{\gamma'}}\right),$$
    where $C$ depends only on $N$ and $a$. By the assumptions on $p$ and $q$ it follows that $\frac{1}{\gamma'}-\frac{1}{p}>0$ and $\frac{1}{\gamma'}-\frac{2}{q}>0$. Let us denote by $\Bar{\gamma}$ the minimum of such two positive numbers. Taking into account that $E_j\le\delta$ for every $j$, we have
    $$
    \begin{cases}
        E_{j+1}\le C^{1+j}E_j^{1+\bar{\gamma}},\\
        E_0\le\delta,
    \end{cases}
    $$
which implies
\begin{align*}
        E_{j} \le C^{\sum_{i=0}^j i(1+\Bar{\gamma})^{j-i}}E_0^{(1+\Bar{\gamma})^j} \le C^{(1+\bar{\gamma})^j\sum_{i=0}^j\frac{i}{(1+\Bar{\gamma})^i}}\delta^{(1+\bar{\gamma})^j}\le (C\delta)^{(1+\bar{\gamma})^j},
\end{align*}
since $\sum_{i=0}^j\frac{i}{(1+\Bar{\gamma})^i} < +\infty$. Now, take $\delta$ such that $C\delta<1$. Then $E_j\to0$, as $j\to+\infty$ and thus, by definition of $V_{j}$, $E_j \to \int_{Q_{1/2}}\rho (u-1)_+^2 = 0$, which yields $u\le1$ in $Q_{{1}/{2}}$, as claimed in the statement.
\end{proof}
\begin{proof}[Proof of Proposition \ref{moser-3.5}] 
    Define $$V_+:=\theta_+u_+,\hspace{0.5cm}\theta_+:=\frac{\sqrt{\delta}}{\|u_+\|_{L^2(Q_1,\rho_\varepsilon^a)}+\|f\|_{L^p(Q_1,\rho_\varepsilon^a)}+\|F\|_{L^q(Q_1,\rho_\varepsilon^a)}},$$
    where $\delta>0$ is as Lemma \ref{no-spike}. The hypothesis of the Lemma \ref{no-spike} are satisfied, so 
    $$\|u_+\|_{L^\infty(Q_{1/2})}\le \frac{1}{\sqrt\delta}\left( \|u_+\|_{L^2(Q_1,\rho_\varepsilon^a)}+\|f\|_{L^p(Q_1,\rho_\varepsilon^a)}+\|F\|_{L^q(Q_1,\rho_\varepsilon^a)}\right).$$
    Repeating the same reasoning with $V_-$ and taking into account that both the estimate \eqref{caccioppoli.inequality-subsolutions} and Lemma \ref{no-spike} hold also for the negative part of solutions, it follows 
    $$\|u_-\|_{L^\infty(Q_{1/2})}\le \frac{1}{\sqrt\delta}\left( \|u_-\|_{L^2(Q_1,\rho_\varepsilon^a)}+\|f\|_{L^p(Q_1,\rho_\varepsilon^a)}+\|F\|_{L^q(Q_1,\rho_\varepsilon^a)}\right).$$
    So, putting together these two inequalities, the thesis follows choosing $C=\frac{4}{\sqrt\delta}$. 
\end{proof}
%
%
%
%
%
%%%%%%%%%%%%%%%%%%%%%%%%%%%%%%%%%%%%%%%%%%%%%%%%%%%%%%%%%%%%%%%%%%%%%%%%%%%%%%%%%%%%%%%%%%%%%%%%%%%%%%%%%%%%%%%%%%%%%%%%%%%%%%%%%%%%%%%%%%%%%%%%%%%%%%%%%%%%%%%%
%%%%%%%%%%%%%%%%%%%%%%%%%%%%%%%%%%%%%%%%%%%%%%%%%%%%%%%%%%%%%%%%%%%%%%%%%%%%%%%%%%%%%%%%%%%%%%%%%%%%%%%%%%%%%%%%%%%%%%%%%%%%%%%%%%%%%%%%%%%%%%%%%%%%%%%%%%%%%%%%
%
%
%
%
%
\section{Approximation results}\label{section4}
The purpose of this section is to establish some approximation results, in the spirit of \cite{SirTerVit21a} (elliptic framework). The main fact is that any weak solution $u$ to \eqref{solution} with $\varepsilon = 0$ can be locally approximated with a family of \emph{classical} solutions $\{u_\e\}_{\e \in (0,1)}$ to \eqref{solution} (with $\e > 0$), that is $u_\e \to u$ as $\e \to 0$, in a suitable sense (see Lemma \ref{rho-to-y} and Lemma \ref{y-to-rho}). This is a key step of our work that will play an important role in the proofs of the H\"older and Schauder estimates.

We begin with the following lemma, in the spirit of \cite{Lie96}*{Theorem 6.1}. 
\begin{lem}\label{liebermann}
Let $a\in\mathbb{R}$, $N\ge1$, $\varepsilon\in[0,1)$ and $f\in L^2(Q_1,\rho_\varepsilon)$, $F\in L^2(Q_1,\rho_\varepsilon^a)^{N+1}$, $u_0\in L^2(B_1,\rho_\varepsilon^a)$. Let $u_\e$ be a weak solution to \eqref{solution1} in $Q_1$. Then there exists $C > 0$ depending only on $N$, $a$ and $\lambda$ such that
\begin{equation}\label{energy.estimates.by.data}
\|u_\e\|_{L^\infty(-1,1;L^2(B_1,\rho_\varepsilon^a))} + \|u_\e\|_{L^2(-1,1;H^1_0(B_1,\rho_\varepsilon^a))} \le C(\|f\|_{L^2(Q_1,\rho_\varepsilon^a)}+\|F\|_{L^2(Q_1,\rho_\varepsilon^a)}+\|u_0\|_{L^2(B_1,\rho_\varepsilon^a)}).
%
   %  &\|\partial_t u\|_{L^2(I_1;H^{-1}(B_1,\rho_\varepsilon^a))} \le C(\|f\|_{L^2(Q_1,\rho_\varepsilon^a)}+\|F\|_{L^2(Q_1,\rho_\varepsilon^a)}+\|u_0\|_{L^2(B_1,\rho_\varepsilon^a)}).
\end{equation}
\end{lem}
\begin{proof} Let us set $u := u_\e$ and notice that $u \in C([-1,1],L^2(B_1,\rho_\varepsilon^a))$ by Remark \ref{remark 2.3 bis}. In what follows, we prove the existence of $C > 0$ depending only on $a$, $N$ and $\lambda$ such that
\begin{equation}\label{eq:EstDataEqTh}
\int_{B_1}\rho_\varepsilon^au^2(\tau)dz + \int_{-1}^{\tau}\int_{B_1}\rho_\varepsilon^a|\nabla u|^2 dz dt \le C \left(\int_{Q_1}\rho_\varepsilon^a(f^2 + |F|^2)dzdt + \int_{B_1}\rho_\varepsilon^a u_0^2 dz\right),
\end{equation}
for every $\tau \in (-1,1)$. The bound in \eqref{energy.estimates.by.data} easily follows by the arbitrariness of $\tau$.

So, let us fix $\tau \in (-1,1)$, $h \in (0,1-\tau)$ and consider the Steklov average $u_h$ (see Remark \ref{rem:Stekelov}). Using a standard approximation procedure (see \cite{Lie96}*{Theorem 6.1}) and recalling that $\mathcal{D}_c^\infty(Q_r)$ is dense in $L^2(I_r;H_0^1(B_r,\rho_\varepsilon^a))$, we may test \eqref{solution.steklov} with $\phi:=u_h {\rchi}_{[-1,\tau]}$ to deduce
$$
\int_{-1}^{\tau}\int_{B_1}\rho_\varepsilon^a (\partial_t u_h u_h + (A\nabla u)_h\cdot\nabla u_h)dzdt = \int_{-1}^{\tau}\int_{B_1}\rho_\varepsilon^a (f_h u_h - F_h\cdot\nabla u_h)dzdt.
$$
Now, using Fubini-Tonelli theorem and integrating w.r.t. $t$, we obtain
\[
\int_{-1}^{\tau}\int_{B_1}\rho_\varepsilon^a \partial_t u_h u_h dzdt = \frac{1}{2}\int_{B_1}\rho_\varepsilon^a\int_{-1}^{\tau}\partial_t(u_h^2)dtdz = \frac{1}{2}\int_{B_1}\rho_\varepsilon^au_h^2(\tau)dz - \frac{1}{2}\int_{B_1}\rho_\varepsilon^au_h^2(-1)dz,
\]
and thus, passing to the limit as $h \to 0$ and recalling that $u \in C([-1,1],L^2(B_1,\rho_\varepsilon^a))$, it follows  
    \begin{align*}
   \frac{1}{2}\int_{B_1}\rho_\varepsilon^au^2(\tau)dz+\int_{-1}^{\tau}\int_{B_1}\rho_\varepsilon^aA\nabla u\cdot\nabla u dz dt = \int_{-1}^{\tau}\int_{B_1}\rho_\varepsilon^a(fu-F\cdot\nabla u)dzdt+\frac{1}{2}\int_{B_1}\rho_\varepsilon^a u_0^2 dz.
    \end{align*}
Recalling that $A$ is uniformly parabolic and applying both H\"older's inequality and Young's inequality, it turns out
    \begin{align*}
        &\frac{1}{2}\int_{B_1}\rho_\varepsilon^au^2(\tau)dz + \lambda\int_{-1}^{\tau}\int_{B_1}\rho_\varepsilon^a|\nabla u|^2 dzdt \\
        &\quad \le \|f\|_{L^2(Q_1,\rho_\e^a)}\left(\int_{-1}^{\tau}\int_{B_1}\rho_\varepsilon^au^{2} dzdt\right)^{{1}/{2}} + \|F\|_{L^2(Q_1,\rho_\e^a)} \left(\int_{-1}^{\tau}\int_{B_1}\rho_\varepsilon^a|\nabla u|^{2} dzdt\right)^{{1}/{2}} + \frac{1}{2} \|u_0\|_{L^2(B_1,\rho_\e^a)}^2\\
        &\quad\le \frac{1}{2} \|f\|_{L^2(Q_1,\rho_\e^a)}^2 + \frac{1}{2}\int_{-1}^{\tau}\int_{B_1}\rho_\varepsilon^au^{2} + \frac{1}{2\lambda} \|F\|_{L^2(Q_1,\rho_\e^a)}^2 + \frac{\lambda}{2}\int_{-1}^{\tau}\int_{B_1}\rho_\varepsilon^a|\nabla u|^{2} + \frac{1}{2} \|u_0\|_{L^2(B_1,\rho_\e^a)}^2,
    \end{align*}
that is
$$
H(\tau) + \lambda\int_{-1}^{\tau}\int_{B_1}\rho_\varepsilon^a|\nabla u|^2 dzdt \le \int_{-1}^\tau H(t)dt + K,
$$
where $H(\tau) := \int_{B_1}\rho_\varepsilon^au^2(\tau)$ and $K := \|f\|_{L^2(Q_1,\rho_\e^a)}^2 + \frac{1}{\lambda} \|F\|_{L^2(Q_1,\rho_\e^a)}^2 + \|u_0\|_{L^2(B_1,\rho_\e^a)}^2$. Finally, since the second term in the r.h.s. is nonnegative, the Gronwall's inequality yields $\int_{-1}^\tau H(t)dt \le K (1+e^\tau)\le K(1+e)$ which, in turn, proves \eqref{eq:EstDataEqTh}.  
\end{proof}
Now, we proceed with the approximation results. In what follows, we will repeatedly use the following elementar fact:
\[
\rho_\e^a \to |y|^a \quad \text{ in } L_{loc}^1(\R^{N+1} \setminus \Sigma), 
\]
as $\e \to 0$.
\begin{lem}\label{rho-to-y}
Let $a\in\mathbb{R}$, $p,q\ge2$, $A$ satisfying \eqref{eq:UnifEll}, $R > 0$ and $I_R := (-R^2,R^2)$. Let $\{f_\varepsilon\}_{\varepsilon\in(0,1)} \subset L^p(Q_R,\rho_\varepsilon^a)$, $\{F_\varepsilon\}_{\varepsilon\in(0,1)} \subset L^q(Q_R,\rho_\varepsilon^a)^{N+1}$ and let $\{u_\varepsilon\}_{\varepsilon\in(0,1)}$ be a family of weak solutions to
\begin{equation}\label{eq-epsilon-solution}
\rho_\varepsilon^a\partial_t u_\varepsilon - {\div }(\rho_\varepsilon^{a} A\nabla u_\varepsilon) = \rho_\varepsilon^a f_\varepsilon + {\div}(\rho_\varepsilon^{a} F_\varepsilon) \quad \text{ in } Q_R.
\end{equation} 
Assume that there exist $C>0$ independent of $\varepsilon$, $f \in L_{loc}^p(Q_R\setminus \Sigma)$ and $F \in L_{loc}^q(Q_R\setminus \Sigma)$ such that 
\begin{equation}\label{estimates-u}
\|u_\varepsilon\|_{L^2(I_R;H^1(B_R,\rho_\varepsilon^a))} + \|u_\varepsilon\|_{L^\infty(I_R;L^2(B_R,\rho_\varepsilon^a))} \le C,
\end{equation} 
\begin{equation}\label{estimates-f}
\|f_\varepsilon\|_{L^p(Q_R,\rho_\varepsilon^a)} + \|F_\varepsilon\|_{L^q(B_R,\rho_\varepsilon^a)^{N+1}} \le C,
\end{equation}     
\begin{equation}\label{eq:ConvRHS}
f_\varepsilon \to f \quad \text{in } L_{loc}^p(Q_R\setminus \Sigma) \qquad \text{and} \qquad F_\varepsilon\to F \quad \text{in } L_{loc}^q(Q_R\setminus \Sigma)^{N+1}
\end{equation}
as $\e \to 0$. Then, $f\in L^p(Q_R,|y|^a)$, $F\in L^q(Q_R,|y|^a)^{N+1}$, and there exist a weak solution $u$ to \eqref{solution} in $Q_R$ (with $\e = 0$) and a sequence $\e_k \to 0$ such that $u_{\varepsilon_k} \to u$ in $L_{loc}^2(I_R;H_{loc}^1(B_R\setminus\Sigma))$ as $k \to +\infty$. Moreover, if we assume that $\{u_\e\}\subset  L^2(I_R;H^1_0(B_R,\rho_\e^a))$, then $u\in L^2(I_R;H^1_0(B_R,|y|^a))$.

\end{lem}
\begin{proof} By scaling, we may assume $R = 1$ and set $I := (-1,1)$.

\

\emph{Step 1.} We have $f\in L^p(Q_1,|y|^a)$ and $F\in L^q(Q_1,|y|^a)$. This easily follows by Fatou's lemma, \eqref{eq:ConvRHS} and $\rho_\e^a \to |y|^a$ a.e. in $Q_1$.

\

\emph{Step 2.} In this step we show the existence of $u \in L^2(I;H_{loc}^1(B_1\setminus\Sigma))$ and a sequence $\e_k \to 0$ such that
\begin{equation}\label{eq:L2locStrongSoln}
u_{\e_k} \to u \quad \text{ in } L^2(I;L_{loc}^2(B_1\setminus\Sigma)),
\end{equation}
as $k \to +\infty$. Further, for every open set $\omega \subset\subset B_1\setminus \Sigma$, $u$ is a weak solution to \eqref{solution} in $\omega\times I$.

Let $\Omega,\tilde{\Omega} \subset \R^{N+1}$ be open sets such that $\tilde{\Omega} \subset\subset \Omega \subset\subset \overline{B}_1\setminus \Sigma$ and let $\xi \in C_c^\infty(\Omega)$ with $0 \le \xi \le 1$, $\xi = 1$ in $\tilde{\Omega}$ and $|\nabla \xi| \leq C_0$, where $C_0 > 0$ depends on $N$, $\Omega$ and $\tilde{\Omega}$.

Define $v_{\varepsilon} := \xi u_\varepsilon$. By \eqref{estimates-u}, we have $v_\e \in L^2(I;H^1_0(\Omega;\rho_\varepsilon^a))\cap L^\infty(I;L^2(\Omega;\rho_\varepsilon^a))$ with
\begin{equation}\label{eq:UniEstve}
\|v_{\e}\|_{L^2(I;H_0^1(\Omega,\rho_\varepsilon^a))} + \|v_{\e}\|_{L^\infty(I;L^2(\Omega,\rho_\varepsilon^a))} \le C,
\end{equation}
for a new $C > 0$ independent of $\e$. Setting $Q := \Omega \times I$ and fixing $\phi \in C_c^\infty(Q)$, we compute
\begin{align*}
        &- \int_Q \rho_\varepsilon^a v_\varepsilon \partial_t\phi dzdt + \int_Q \rho_\varepsilon^a A\nabla v_\varepsilon\cdot\nabla\phi dzdt \\
        =& -\int_Q \rho_\varepsilon^a u_\varepsilon \partial_t (\xi\phi)dz dt+\int_Q \rho_\varepsilon^a \xi A\nabla u_\varepsilon\cdot\nabla\phi dzdt + \int_Q \rho_\varepsilon^a u_\varepsilon A\nabla \xi\cdot\nabla\phi dzdt\\
        =&-\int_Q \rho_\varepsilon^a u_\varepsilon \partial_t (\xi\phi)dz dt + \int_Q \rho_\varepsilon^a A \nabla u_\varepsilon\cdot\nabla(\xi\phi) dzdt -\int_Q \rho_\varepsilon^a \phi A\nabla u_\varepsilon\cdot \nabla\xi dzdt + \int_Q \rho_\varepsilon^a u_\varepsilon A\nabla \xi\cdot\nabla\phi dzdt\\
        =&\int_Q \rho_\varepsilon^a\left( f_\varepsilon\xi\phi-F_\varepsilon\cdot\nabla (\xi\phi)-\phi A\nabla u_\varepsilon\cdot \nabla\xi+u_\varepsilon A\nabla \xi\cdot\nabla\phi\right)dzdt\\
        =&\int_Q \rho_\varepsilon^a\left( f_\varepsilon\xi\phi-\xi F_\varepsilon\cdot\nabla \phi-\phi 
        F_\varepsilon\cdot\nabla \xi-\phi A\nabla u_\varepsilon\cdot \nabla\xi+u_\varepsilon A\nabla \xi\cdot\nabla\phi\right)dzdt,
\end{align*}
that is,
$$
\rho_\varepsilon^a\partial_t v_\varepsilon - {\div }(\rho_\varepsilon^{a} A\nabla v_\varepsilon) = \rho_\varepsilon^a \tilde{f}_\varepsilon + \div(\rho_\varepsilon^{a}\Tilde{F}_\varepsilon ) \quad \text{ in } Q,
$$
in the weak sense, where we have set 
$$\tilde{f}_\varepsilon := f_\varepsilon\xi - F_\varepsilon\cdot\nabla \xi - A\nabla u_\varepsilon\cdot\nabla\xi,  \qquad   \Tilde{F}_\varepsilon := F_\varepsilon\xi -u_\varepsilon A\nabla \xi.   
$$
Proceeding as in Remark \ref{remark 2.3 bis}, one combines the uniform estimates \eqref{estimates-u}, \eqref{estimates-f} and \eqref{eq:UniEstve} with the H\"older's and Young's inequalities, to deduce
\[
-\int_Q \rho_\varepsilon^a v_\e \partial_t\phi dzdt \leq C \|\phi\|_{L^2(I;H_0^1(\Omega,\rho_\e^a))},
\]
for some new $C > 0$ depending only on $N$, $\Omega$, $\tilde{\Omega}$, $a$ and $\Lambda$. Notice that, respect to Remark \ref{remark 2.3 bis}, $C$ is independent of $\e$: this is because $H_0^1(\Omega,\rho_\e^a) = H_0^1(\Omega)$ and we can make use of the Poincar\'e inequality with constant independent of $\e$, see Remark \ref{rem:h01h01rhoe}. As a consequence of the above inequality, it follows $\partial_t v_\varepsilon \in L^2(I;H^{-1}(\Omega,\rho_\varepsilon^a))$ with 
$\|\partial_t v_\e\|_{L^2(I;H^{-1}(\Omega,\rho_\varepsilon^a))}\le C$ and so, since $H^{-1}(\Omega,\rho_\e^a) = H^{-1}(\Omega)$ by Remark \ref{rem:h01h01rhoe}, we obtain
\begin{equation}\label{estimate:t-derivative:step1}
\|\partial_t v_\varepsilon\|_{L^2(I;H^{-1}(\Omega))}\le C.
\end{equation}
At this point, combining \eqref{eq:UniEstve}, \eqref{estimate:t-derivative:step1} and Remark \ref{rem:h01h01rhoe} again, it follows
\[
\|v_{\e}\|_{L^2(I;H_0^1(\Omega))} + \|\partial_t v_\varepsilon\|_{L^2(I;H^{-1}(\Omega))}\le 2C,
\]
and thus the Aubin-Lion lemma (see for instance \cite{Simon87}*{Corollary 8}) yields the existence of $v \in L^2(I;H_0^1(\Omega))$ such that $v_\e \to v$ in $L^2(Q)$, along a suitable sequence. Further, since by \eqref{estimates-u} there is $u \in L^2(I;H^1(\Omega))$ such that $u_\e \rightharpoonup u$ in $L^2(I;H^1(\Omega))$ (along a suitable sequence) and $\xi = 1$ in $\tilde{\Omega}$, we deduce $v = u$ in $L^2(\tilde{\Omega}\times I)$. A standard diagonal argument yields both $u \in L^2(I;H_{loc}^1(B_1\setminus\Sigma))$ and \eqref{eq:L2locStrongSoln} (take for instance $\Omega = \Omega_j := B_1\setminus\{|y|<\frac{1}{j+3}\}$ and $\tilde{\Omega} = \tilde{\Omega}_j := B_{\frac{j+1}{j+2}}\setminus\{|y|<\frac{1}{j+2}\}$, $j \in \N$).

Now, fix $\omega \subset\subset B_1\setminus\Sigma$. Combining \eqref{eq:L2locStrongSoln} and $u_{\e_k} \rightharpoonup u$ in $L^2(I;H^1(\omega))$ and recalling that $\rho_{\e_k}^a \to |y|^a$ in $L^2(\omega)$, and 
testing \eqref{eq-epsilon-solution} with $\phi \in C_c^\infty(\omega\times I)$, we may pass to the limit as $k \to +\infty$ into (the weak formulation of) \eqref{eq-epsilon-solution} and deduce that $u$ is a weak solution to \eqref{solution} in $\omega\times I$.      

\

\emph{Step 3.} Now we prove that 
\begin{equation}\label{eq:L2locStrongNabla}
\D u_{\e_k} \to \D u \quad \text{ in } L_{loc}^2((B_1\setminus\Sigma)\times I),
\end{equation}
as $k \to +\infty$, up to passing to a suitable subsequence.

Let $\Omega \subset\subset B_1\setminus\Sigma$, $\eta \in C_c^\infty(\Omega)$, $-1 < t_1 < t_2 < 1$ and $h \in (0,1-t_2)$ and let $(u_{\e_k})_h$ and $u_h$ be the Steklov averages of $u_{\e_k}$ and $u$, respectively (see Remark \ref{rem:Stekelov}). Similar to the proof of Lemma \ref{liebermann}, we test the equation of $(u_{\e_k})_h$ with $\eta^2 \chi_{[t_1,t_2]} (u_{\e_k})_h$ to obtain 
\begin{equation}
    \begin{aligned} &\int_{t_1}^{t_2}\int_{\Omega}\rho_{\varepsilon_k}^a \Big(\partial_t (u_{\varepsilon_k})_h \eta^2 (u_{\varepsilon_k})_h+(A\nabla u_{\varepsilon_k})_h\cdot\nabla(\eta^2 (u_{\varepsilon_k})_h)\Big)\\
    & \quad = \int_{t_1}^{t_2}\int_{\Omega}\rho_{\varepsilon_k}^a \Big( \tfrac{1}{2}\partial_t(\eta^2 (u_{\varepsilon_k})_h^2)+\eta^2 (A\nabla u_{\varepsilon_k})_h \cdot \nabla (u_{\varepsilon_k})_h+ 2\eta (u_{\varepsilon_k})_h (A\nabla u_{\varepsilon_k})_h\cdot \nabla \eta\Big)\\
    & \quad =\frac{1}{2}\int_{\Omega}\rho_{\varepsilon_k}^a  \eta^2 (u_{\varepsilon_k})_h^2 \,\bigg|_{t=t_1}^{t=t_2} + \int_{t_1}^{t_2}\int_{\Omega}\rho_{\varepsilon_k}^a \Big(\eta^2 (A\nabla u_{\varepsilon_k})_h \cdot \nabla (u_{\varepsilon_k})_h+ 2\eta (u_{\varepsilon_k})_h(A \nabla u_{\varepsilon_k})_h\cdot \nabla \eta\Big)\\
    & \quad =\int_{t_1}^{t_2}\int_{\Omega}\rho_{\varepsilon_k}^a \Big(
    (f_{\varepsilon_k})_h\eta^2 (u_{\varepsilon_k})_h+(F_{\varepsilon_k})_h\cdot\nabla(\eta^2 (u_{\varepsilon_k})_h)\Big),
    \end{aligned}
\end{equation}  
which, rearranging terms, becomes
\begin{equation}\label{eq:step1:epsilon}
\begin{aligned}
\int_{t_1}^{t_2}\int_{\Omega} \rho_{\varepsilon_k}^a \eta^2 (A\nabla u_{\varepsilon_k})_h \cdot \nabla (u_{\varepsilon_k})_h  = &-\frac{1}{2}\int_{\Omega}\rho_{\varepsilon_k}^a  \eta^2 (u_{\varepsilon_k})_h^2 \,\bigg|_{t=t_1}^{t=t_2} - 2\int_{t_1}^{t_2}\int_{\Omega}\rho_{\varepsilon_k}^a \eta (u_{\varepsilon_k})_h (A \nabla u_{\varepsilon_k})_h\cdot \nabla \eta \\     
& + \int_{t_1}^{t_2}\int_{\Omega}\rho_{\varepsilon_k}^a \Big(
    (f_{\varepsilon_k})_h\eta^2 (u_{\varepsilon_k})_h+(F_{\varepsilon_k})_h\cdot\nabla(\eta^2 (u_{\varepsilon_k})_h)\Big)
\end{aligned}
\end{equation}
Using the properties of the Steklov averages (see Remark \ref{rem:Stekelov}), we can take the limit as $h\to0$ in \eqref{eq:step1:epsilon} to obtain
\begin{equation}\label{eq:step1:epsilon:limit}
    \begin{aligned}
        \int_{t_1}^{t_2}\int_{\Omega}\rho_{\varepsilon_k}^a \eta^2 A\nabla u_{\varepsilon_k}\cdot \nabla u_{\varepsilon_k} = &-\frac{1}{2}\int_{\Omega} \rho_{\varepsilon_k}^a \eta^2 u_{\varepsilon_k}^2\bigg{|}_{t=t_1}^{t_2}-2\int_{t_1}^{t_2}\int_{\Omega}\rho_{\varepsilon_k}^a \eta u_{\varepsilon_k} A\nabla u_{\varepsilon_k}\cdot\nabla\eta\\
&+ \int_{t_1}^{t_2}\int_{\Omega}\rho_{\varepsilon_k}^a \Big( f_{\varepsilon_k}\eta^2 u_{\varepsilon_k} + F_{\varepsilon_k}\cdot\nabla(\eta^2 u_{\varepsilon_k})\Big),
    \end{aligned}
\end{equation}
for every $k \in \N$. Now, by testing the equation of $u_h$ with $\eta^2 \chi_{[t_1,t_2]} u_h$ and repeating the very same argument, one shows that
\begin{equation}\label{eq:step1:limit}
\begin{aligned}
        \int_{t_1}^{t_2}\int_{\Omega}|y|^a \eta^2 A\nabla u\cdot\nabla u = &-\frac{1}{2}\int_{\Omega} |y|^a \eta^2 u^2 \, \bigg|_{t=t_1}^{t=t_2} - 2\int_{t_1}^{t_2}\int_{\Omega}|y|^a\eta u A\nabla u\cdot\nabla\eta \\
&+ \int_{t_1}^{t_2}\int_{\Omega}|y|^a \bigg(f\eta^2 u+F\cdot\nabla(\eta^2 u)\bigg),
\end{aligned}
\end{equation}
for a.e. $t_1$ and $t_2$ as above. Recalling that $\D u_{\e_k} \rightharpoonup \D u$ in $L^2(\Omega \times I)$ and using both \eqref{eq:L2locStrongSoln} and $\rho_{\e_k}^a \to |y|^a$ in $L^2(\Omega)$, we find 
\[
-\frac{1}{2}\int_{\Omega} \rho_{\varepsilon_k}^a \eta^2 u_{\varepsilon_k}^2\bigg{|}_{t=t_1}^{t_2}-2\int_{t_1}^{t_2}\int_{\Omega}\rho_{\varepsilon_k}^a \eta u_{\varepsilon_k} A\nabla u_{\varepsilon_k}\cdot\nabla\eta \to -\frac{1}{2}\int_{\Omega} |y|^a \eta^2 u^2 \, \bigg|_{t=t_1}^{t=t_2} - 2\int_{t_1}^{t_2}\int_{\Omega}|y|^a\eta u A\nabla u\cdot\nabla\eta,
\]
as $k \to +\infty$, for a.e. $t_1$ and $t_2$ as above. On the other hand, since in addition $f_{\e_k} \to f$ in $L^2(\Omega\times I)$ and $F_{\e_k} \to F$ in $L^2(\Omega\times I)^{N+1}$, it follows   
\[
\int_{t_1}^{t_2}\int_{\Omega}\rho_{\varepsilon_k}^a \Big( f_{\varepsilon_k}\eta^2 u_{\varepsilon_k} + F_{\varepsilon_k}\cdot\nabla(\eta^2 u_{\varepsilon_k})\Big) \to \int_{t_1}^{t_2}\int_{\Omega}|y|^a \Big(f\eta^2 u+F\cdot\nabla(\eta^2 u)\Big),
\]
as $k \to +\infty$, for a.e. $t_1$ and $t_2$ as above. Consequently,
$$ 
\lim_{k\to+\infty}\int_{t_1}^{t_2}\int_{\Omega}\rho_{\varepsilon_k}^a \eta^2 A\nabla u_{\varepsilon_k}\cdot \nabla u_{\varepsilon_k} = \int_{t_1}^{t_2}\int_{\Omega}|y|^a \eta^2 A\nabla u\cdot \nabla u.
$$
Since $\rho_\varepsilon^a$ is bounded and bounded away from $0$ in $\Omega$ uniformly in $\e$ and 
$u_{\e_k} \rightharpoonup u$ in $L^2(I;H^1(\Omega))$, and $A$ satisfies \eqref{eq:UnifEll}, we may let $\eta \to \chi_\Omega$ and use the triangular inequality to deduce $\nabla u_{\varepsilon_k} \to \nabla u$ in $L^2(\Omega\times (t_1,t_2))$ as $k\to+\infty$. A diagonal argument as above then shows \eqref{eq:L2locStrongNabla}. 

\

\emph{Step 4.} Next, we prove that $u\in L^2(I;H^1(B_1,|y|^a))\cap L^\infty(I;L^2(B_1,|y|^a))$.

By \eqref{estimates-u}, \eqref{eq:L2locStrongSoln} and Fatou's lemma, we have that $u\in L^\infty(I;L^2(B_1,|y|^a))$. 
Indeed, for a.e. $t\in I$, one has
$$
\int_{B_1}|y|^a u^2(z,t)dz\le\liminf_{k}\int_{B_1}\rho_{\varepsilon_k}u_{\varepsilon_k}^2(z,t)dz\le \|u_{\varepsilon_k}\|_{L^\infty(I;L^2(B_1,\rho_{\varepsilon_k}^a))}\le C.
$$ 
To show that $u\in L^2(I;H^1(B_1,|y|^a))$ we distinguish three cases, depending on the value of $a$.

Assume first  $a\ge 0$. Since $|y|^a\le \rho_\varepsilon^a(y)$ for every $\varepsilon\in(0,1)$, one has
$$\|u_\varepsilon\|_{L^2(I;H^1(B_1,|y|^a))}\le\|u_\varepsilon\|_{L^2(I;H^1(B_1,\rho_\varepsilon^a))}\le C,$$
by  \eqref{estimates-u}. Then, the family $\{u_\varepsilon\}_{\e \in (0,1)}$ is uniformly bounded in $L^2(I;H^1(B_1,|y|^a))$ and thus $u \in L^2(I;H^1(B_1,|y|^a))$ by weak convergence. Moreover, if $\{u_\e\}_{\e \in (0,1)}\subset L^2(I;H^1_0(B_1,\rho_\e^a)) \subset L^2(I;H^1_0(B_1,|y|^a))$, then $\{u_\e\}_{\e \in (0,1)} \subset L^2(I;H^1_0(B_1,|y|^a))$ and $u\in L^2(I;H^1_0(B_1,|y|^a))$ by weak convergence.

Second, fix $-1 < a < 0$. In this case $|y|^a$ belongs to the Muckenhoupt class $A_2$ and, since $a<0$, one has $|y|^a\ge1$. Therefore,
$$
\|u_\varepsilon\|_{L^2(I;H^1(B_1))}\le\|u_\varepsilon\|_{L^2(I;H^1(B_1,\rho_\varepsilon^a))}\le C,
$$
and so $u_{\varepsilon_k} \rightharpoonup u$ $L^2(I;H^1(B_1))$ and that $u$ possesses weak gradient. Now, since $u_{\varepsilon_k}\to u$ and $\nabla u_{\varepsilon_k}\to \nabla u$ a.e. in $Q_1$ by \eqref{eq:L2locStrongSoln} and \eqref{eq:L2locStrongNabla}, we may invoke Fatou's lemma again to conclude $u$ and $|\nabla u|$ belong to $L^2(Q_1,|y|^a)$. This shows our claim thanks to Proposition \ref{prop:CharH1a}.

Furthermore, if $\{u_\e\}_{\e \in (0,1)} \subset L^2(I;H^1_0(B_1,\rho_\e^a))$, then $\{u_\e\}_{\e \in (0,1)} \subset L^2(I;H^1_0(B_1))$ and thus there exists a sequence satisfying $u_{\varepsilon_k}\to u$ weakly in $ L^2(I;H^1_0(B_1))$. So, $u\in  L^2(I;H^1(B_1,|y|^a))\cap  L^2(I;H^1_0(B_1))$.

Now, fix $\delta>0$ and consider $\psi\in C_c^\infty(Q_1)$ such that $\|u-\psi\|_{L^2(I;H^1_0(B_1))}\le \bar{\delta}$, where $\bar{\delta} \in (0,1)$ will be chosen in a moment. Let $\hat{y} \in (0,1)$ small. Then
\begin{align*}
\int_{Q_1}|y|^a|\nabla u-\D\psi|^2 &= \int_{\{|y|\ge \hat{y}\}} |y|^a  |\nabla u-\D\psi|^2 +\int_{\{|y|< \hat{y}\}}|y|^a |\nabla u-\D\psi|^2 \\
&\le \hat{y}^a\int_{Q_1}|\nabla u-\D\psi|^2+\delta'(\hat{y}) \le \hat{y}^a\bar{\delta}+\delta'(\hat{y}),
\end{align*}
where $\delta'(\hat{y})\to 0$ as $\hat{y}\to 0$, since the function $|y|^a|\nabla u-\D\psi|^2\in L^1(Q_1)$. Choosing $\bar{\delta}<{\delta'(\hat{y})}/{\hat{y}^a}$ and $\hat{y}$ such that $\delta'(\hat{y})<\delta/2$, we finally obtain $\|u-\psi\|_{L^2(I;H^1_0(B_1,|y|^a))}\le \delta$, that is, $u\in L^2(I;H^1_0(B_1,|y|^a))$ thanks to the arbitrariness of $\delta$.

Finally, let $a \le -1$. In this case, we consider the isometry $\bar{T}_\varepsilon^a$ defined in \eqref{eq:IsoParabolic} and we set $v_\varepsilon := \sqrt{\rho_\varepsilon^a}u_\varepsilon$. By Remark \ref{rem:IsoEll} and \eqref{estimates-u}, the family $\{v_\varepsilon\}_{\e \in (0,1)}$ is uniformly bounded in $L^2(I;H^1(B_1))$ and so $v_{\varepsilon_k}\rightharpoonup v$ weakly in $L^2(I;H^1(B_1))$. Further, by \eqref{eq:L2locStrongSoln}, we have
$$
\int_{Q_1} v\phi\leftarrow\int_{Q_1} v_{\e_k} \phi=\int_{Q_1} \sqrt{\rho_{\e_k}^a} u_{\e_k} \phi\to \int_{Q_1} |y|^{a/2} u\phi,
$$
for every $\phi\in C_c^\infty(Q_1)$, which implies $v=|y|^{a/2}u$ a.e. in $Q_1$. So, noticing that $u = (\bar{T}_0^a)^{-1} v$ and applying Remark \ref{rem:IsoEll} again, we  conclude that $u\in L^2(I;H^1(B_1,|y|^a))$.

Moreover, if  $\{u_\e\}_{\e \in (0,1)} \subset L^2(I;H^1_0(B_1,\rho_\e^a))$ then $\{v_\e\}_{\e \in (0,1)} \subset L^2(I;H^1_0(B_1))$. So, $v_{\e_k}\to v$ weakly in such space, which implies that $u\in L^2(I;H^1_0(B_1,|y|^a))$.

\

\emph{Step 5.} In this last step we show that $u$ satisfies \eqref{solution} (with $\e = 0$) in the weak sense. Let us fix a test function $\phi\in \mathcal{D}_c^\infty(Q_1)$, see Definition \ref{def:TestF} with $\e = 0$. By \eqref{eq:L2locStrongSoln} and \eqref{eq:L2locStrongNabla}, we have both
$$
\begin{aligned}
\rho_{\varepsilon_k}^a (- u_{\varepsilon_k}\partial_t\phi +  A\nabla u_{\varepsilon_k}\cdot\nabla\phi ) \to |y|^a (-u\partial_t \phi + A\nabla u\cdot\nabla\phi) \quad\text{ a.e. in } Q_1
\end{aligned}
$$
and
$$
\rho_{\varepsilon_k}^a (f_{\varepsilon_k}\phi - F_{\varepsilon_k}\cdot\nabla\phi)  \to  |y|^a (f\phi -  F\cdot\nabla\phi) \quad \text{ a.e. in }Q_1,
$$
as $k \to + \infty$. Now, let $E\subset Q_1$ be measurable. By \eqref{eq:UnifEll}, \eqref{estimates-u} and the H\"older inequality, we get
$$
\int_E \rho_{\varepsilon_k}^a \big|- u_{\varepsilon_k}\partial_t\phi  +  A\nabla u_{\varepsilon_k}\cdot\nabla\phi\big| \le C \|u_\varepsilon\|_{L^2(I;H^1(B_1,\rho_{\varepsilon_k}^a))}\|\nabla_{x,t}\phi\|_{L^\infty(Q_1)}\Big(\int_{E\cap\supp(\phi)} \rho^a_{\varepsilon_k}\Big)^{1/2} \le \delta(E),
$$
where $\delta(E) \geq 0$ satisfies $\delta(E) \to 0$ as $|E|\to0$. Indeed, when $a\le-1$, we have $\rho_{\varepsilon_k}^a\le|y|^a\in L^\infty(E\cap\supp(\phi))$, by the definition of $\mathcal{D}_c^\infty(Q_1)$. Instead, when $a>-1$, one has $\rho_{\varepsilon_k}^a \le C|y|^{\min(0,a)} \in  L^1(B_1)$. In particular, it follows that the family $-\rho_{\varepsilon_k}^a u_{\varepsilon_k}\partial_t\phi + \rho_{\varepsilon_k}^a A\nabla u_{\varepsilon_k}\cdot\nabla\phi$ is uniformly integrable and the Vitali's  theorem 
%(\rred{Bogachev, Vladimir I. (2007). Measure Theory Volume I. New York: Springer. pp. 267–271. THEOREM 4.5.4.}) 
yields 
$$
\int_{Q_1}\rho_{\varepsilon_k}^a \left(-u_{\varepsilon_k}\partial_t\phi + A\nabla u_{\varepsilon_k}\cdot\nabla\phi\right)\to\int_{Q_1}|y|^a \left(-u\partial_t \phi
+A\nabla u\cdot\nabla\phi \right),
$$
as $k \to +\infty$. With a very similar argument, we obtain
$$
\int_{Q_1}\rho_{\varepsilon_k}^a\left(f_{\varepsilon_k}\phi-F_{\varepsilon_k}\cdot\nabla\phi\right)\to\int_{Q_1}|y|^a\left(f\phi-F\cdot\nabla\phi
\right),
$$
as $k \to +\infty$, and our statement follows.
\end{proof}
\begin{lem}\label{y-to-rho}
    Let $a\in\mathbb{R}$, $p,q\ge2$, $A$ satisfying \eqref{eq:UnifEll}, $R>0$ and $I_R:=(-R^2,R^2)$. Let $f\in L^p(Q_R,|y|^a)$, $F\in L^q(Q_R,|y|^a)^{N+1}$ and let $u$ be a weak solution to 
    \begin{equation}\label{eqeqeqeq2}
        |y|^a\partial_t u-{\div }(|y|^{a} A\nabla u)=|y|^af+{\div}(|y|^{a} F) \quad \text{ in } Q_R.
    \end{equation}
    Then, for every $r\in(0,R)$, there exist $\{u_\varepsilon\}_{\varepsilon\in(0,1)}$, $\{f_\varepsilon\}_{\varepsilon\in(0,1)}$ and $\{F_\varepsilon\}_{\varepsilon\in(0,1)}$ satisfying the assumptions of Lemma \ref{rho-to-y} in $Q_r$. Moreover, there exists $\varepsilon_k \to 0$ such that $u_{\varepsilon_k} \to u$ in $ L_{loc}^2(I_r;H_{loc}^1(B_{r}\setminus\Sigma))$ as $k \to +\infty$.
\end{lem}
\begin{proof}
By scaling, we may assume $R=1$ and set $I:=(-1,1)$.

\

\emph{Step 1.} Let us fix $r\in(0,1)$ and set $\tilde{B}:=B_r$, $\tilde{Q}:=\tilde{B}\times I$, $B:=B_\frac{1+r}{2}$ and $Q:=B\times I$. Consider a cut-off function $\xi\in C_c^\infty(B_1)$ such that
$$
\supp(\xi) \subset \overline{B}, \quad\xi\equiv1 \text{ in } \tilde{B},\quad 0\le\xi\le1 \text{ in } B_1, \quad |\nabla \xi|\le C_0,
$$
for some $C_0 > 0$ depending on $N$ and $r$, and define $\Tilde{u} := \xi u$. Now, given $\phi\in \mathcal{D}_c^\infty(Q)$, the same computations of Lemma \ref{rho-to-y} show that 
 \begin{align*}
     \int_Q |y|^a\Big(-\Tilde{u}\partial_t\phi + A\nabla\Tilde{u}\cdot\nabla\phi\Big) = \int_Q |y|^a\Big( f\xi\phi-\xi F\cdot\nabla \phi-\phi 
        F\cdot\nabla \xi-\phi A\nabla u\cdot \nabla\xi+u A\nabla \xi\cdot\nabla\phi\Big),
 \end{align*}
and thus, setting
$$
\tilde{f}:=f\xi,\hspace{0.4cm}\Tilde{F}:=F\xi,\hspace{0.4cm}\Tilde{g}:=-F\cdot\nabla \xi-A\nabla u\cdot\nabla\xi,\hspace{0.4cm} \Tilde{G}:=-u A\nabla \xi,
$$
we obtain that $\Tilde{u}$ is a weak solution to
$$
|y|^a\partial_t \Tilde{u}-{\div }(|y|^{a} A\nabla \Tilde{u})=|y|^a(\tilde{f}+\Tilde{g})+{\div}(|y|^{a}(\Tilde{F}+\Tilde{G})) \quad\text{ in }Q,
$$
where we have used that $\tilde{u} \in L^2(I;H^1_0(B,|y|^a))\cap L^\infty(I;L^2(B,|y|^a))$ by construction. Moreover, since $u \in L^2(I;H^1(B_1,|y|^a))$ by definition and $p,q \geq 2$, then $\Tilde{f}\in L^p(Q,|y|^a)$, $\Tilde{F}\in L^q(Q,|y|^a)^{N+1}$, $\Tilde{g}\in L^2(Q,|y|^a)$ and $\Tilde{G}\in L^2(Q,|y|^a)^{N+1}$. Therefore, by Remark \ref{remark 2.3 bis}, it follows that $ \Tilde{u}\in C(\bar{I};L^2(B,|y|^a))$. In particular, $\Tilde{u}_0: = \Tilde{u}_{{|}_{t=-1}} = \xi u_{{|}_{t=-1}} \in L^2(B,|y|^a)$ is well-defined and $\Tilde{u}$ is a weak solution to
\begin{equation}\label{sol-h10}
    \begin{cases}
        |y|^a\partial_t \Tilde{u}-{\div }(|y|^{a} A\nabla \Tilde{u})=|y|^a(\tilde{f}+\Tilde{g})+{\div}(|y|^{a}(\Tilde{F}+\Tilde{G})) &\text{ in }Q,\\
        \Tilde{u}=0&\text{ on }\partial B{\times I},\\
        \Tilde{u}_{{|}_{t=-1}}=\Tilde{u}_0 &\text{ on }B.
    \end{cases}
\end{equation}

\

\emph{Step 2.} In this step, we construct a family of smooth approximations $u_\e$ of $\Tilde{u}$, as in the statement. We distinguish between two cases, depending on the value of $a$.

First, let $a>0$. We define
\begin{align*}
    f_\varepsilon:=\left(\frac{|y|^a}{\rho_\varepsilon^a}\right)^{{1}/{p}}\Tilde{f},\hspace{0.4cm}
F_\varepsilon:=\left(\frac{|y|^a}{\rho_\varepsilon^a}\right)^{{1}/{q}}\Tilde{F},\hspace{0.4cm}
g_\varepsilon:=\left(\frac{|y|^a}{\rho_\varepsilon^a}\right)^{{1}/{2}}\Tilde{g}, \hspace{0.4cm}
G_\varepsilon:=\left(\frac{|y|^a}{\rho_\varepsilon^a}\right)^{{1}/{2}}\Tilde{G},\hspace{0.4cm}
u_{0,\varepsilon}:=\left(\frac{|y|^a}{\rho_\varepsilon^a}\right)^{{1}/{2}}\Tilde{u}_0,
\end{align*}
and consider the family of weak solutions $\{u_\varepsilon\}_{\e \in (0,1)}$ to
\begin{equation}\label{sol-h10-eps}
    \begin{cases}
        \rho_\varepsilon^a\partial_t u_\varepsilon-{\div }(\rho_\varepsilon^a A \nabla u_\varepsilon) = \rho_\varepsilon^a({f_\varepsilon}+{g_\varepsilon})+{\div}(\rho_\varepsilon^a({F_\varepsilon}+{G_\varepsilon})) &\text{ in }Q\\
        u_\varepsilon=0&\text{ on }\partial B{\times{I}},\\
        u_{\varepsilon{|}_{t=-1}}={u}_{0,\varepsilon} &\text{ on }B.
    \end{cases}
\end{equation}
By construction, we have 
\begin{align}\label{eq:estimate:r.h.s.:approximant}
    \begin{aligned}
         \|f_\varepsilon\|_{L^p(Q,\rho_\varepsilon^a)}+\|g_\varepsilon\|_{L^2(Q,\rho_\varepsilon^a)}+\|F_\varepsilon\|_{L^q(Q,\rho_\varepsilon^a)}    +\|G_\varepsilon\|_{L^2(Q,\rho_\varepsilon^a)}+\|u_{0,\varepsilon}\|_{L^2(B,\rho_\varepsilon^a)}\le C,
    \end{aligned}
\end{align}
for some $C > 0$ independent of $\varepsilon$ and $f_\varepsilon\to\Tilde{f}$, $F_\varepsilon\to\Tilde{F}$, $g_\varepsilon\to\Tilde{g}$, $G_\varepsilon\to\Tilde{G}$ a.e. in $Q$ and $u_{0,\varepsilon}\to{\Tilde{u}_0}$ a.e. in $B$. Furthermore, since $a > 0$, we may apply the Lebesgue's dominated convergence theorem to deduce that 
\begin{equation}\label{eq:ConvStronLpRHS}
\begin{aligned}
&f_\e \to \tilde{f} \quad \text{ in } L_{loc}^p((B\setminus\Sigma)\times I), \qquad F_\e \to \tilde{F} \quad \text{ in } L_{loc}^q((B\setminus\Sigma)\times I)^{N+1}, \\
&\, g_\e \to \tilde{g} \quad \text{ in } L_{loc}^2((B\setminus\Sigma)\times I), \qquad G_\e \to \tilde{G} \quad \text{ in } L_{loc}^2((B\setminus\Sigma)\times I)^{N+1},
\end{aligned}
\end{equation}
and $u_{0,\e} \to \tilde{u}_0$ in $L_{loc}^2(B\setminus\Sigma)$ as $\e \to 0$.

The case $a\le0$ is easier: we set
$$
f_\varepsilon:=\Tilde{f},\hspace{0.4cm}F_\varepsilon:=\Tilde{F},\hspace{0.4cm}g_\varepsilon:=\Tilde{g},\hspace{0.4cm}G_\varepsilon:=\Tilde{G},\hspace{0.4cm}u_{0,\varepsilon}:={\Tilde{u}_0}.
$$
Since $\rho_\varepsilon^a\le|y|^a$, we immediately deduce \eqref{eq:estimate:r.h.s.:approximant}, while \eqref{eq:ConvStronLpRHS} is obvious by definition.   

\

\emph{Step 3.} Combining Lemma \ref{liebermann} and \eqref{eq:estimate:r.h.s.:approximant}, we deduce that the family $\{u_\e\}_{\e \in (0,1)}$ is uniformly bounded in $L^2(I;H^1_0(B,\rho_\varepsilon^a)) \cap L^\infty(I;L^2(B,\rho_\varepsilon^a))$. Consequently, by \eqref{eq:estimate:r.h.s.:approximant} again and \eqref{eq:ConvStronLpRHS}, $\{u_\varepsilon\}_{\e \in (0,1)}$, $\{f_\varepsilon + g_\e\}_{\e \in (0,1)}$ and $\{F_\varepsilon + G_\e \}_{\e \in (0,1)}$ satisfy the assumptions of Lemma \ref{rho-to-y} in $Q$ and so there exist $\e_k \to 0$ and $\bar{u}\in L^2(I;H^1_0(B,|y|^a))\cap  C(\bar{I};L^2(B,|y|^a))$ (see Remark \ref{remark 2.3 bis}) such that $u_{\varepsilon_k}\to \bar{u}$ in $L^2_{loc}(I;H^1_{loc}(B\setminus\Sigma))$. Since $u_{0,\e} \to \tilde{u}_0$ in $L_{loc}^2(B\setminus\Sigma)$, $\bar{u}|_{t=-1} = \tilde{u}_0$ in $L^2(B,|y|^a)$ and therefore $\bar{u}$ is a weak solution to \eqref{sol-h10}.

As consequence, we obtain $\bar{u} = \Tilde{u}$ a.e. in $Q$ by uniqueness of $\tilde{u}$ (uniqueness of weak solutions to \eqref{sol-h10} follows by the classical theory of the Cauchy-Dirichlet problem in abstract Hilbert spaces, see \cite{lions}) and our statement follows since $\tilde{u} = u$ a.e. in $\tilde{Q}$ by definition.
\end{proof}
\begin{oss}\label{rem:approx:half}
Let $a>-1$ and $R>r>0$. Then, Lemma \ref{rho-to-y} and Lemma \ref{y-to-rho} hold for weak solutions to \eqref{solution-mezza-palla} in $Q_R^+$. That is, if $\{u_\varepsilon\}_{\varepsilon\in(0,1)}$ is a family of weak solutions to \eqref{solution-mezza-palla},
    such that $u_\e$, $f_\e$, $F_\e$ and $A$ satisfy the same assumptions of Lemma \ref{rho-to-y} in $Q_R^+$, then $u_\e\to u$ in the sense of Lemma \ref{rho-to-y} and $u$ is a weak solution to \eqref{solution-mezza-palla} in $Q_R^+$ with $\e=0$.
Further, if $u$ is a weak solution \eqref{solution-mezza-palla} in $Q_R^+$ with $\e=0$, we can construct families $\{u_{\e}\}_{\e\in(0,1)}$, $\{f_{\e}\}_{\e\in(0,1)}$, $\{F_{\e}\}_{\e\in(0,1)}$ such that the assumptions of Lemma \ref{rho-to-y} in $Q_r^+$ and $u_{\e}\to u$ in the sense of Lemma \ref{y-to-rho}.

Indeed, fixed $\e\in[0,1)$, let us consider a solution $u_\e$ to \eqref{solution-mezza-palla} in $Q_R^+$ 
and let $\phi\in C_c^\infty(Q_1)$ be a test function. Let us define
$$J:=\left(\begin{array}{c|c}
       {I}_n&0  \\ \hline
       0&-1 
  \end{array}\right),\qquad \tilde{A}(x,y,t):=JA(x,-y,t)J,\qquad \tilde{u}_\e(x,y,t):={u_\e}(x,-y,t),$$
  $$\tilde{f}_\e(x,y,t):={f_\e}(x,-y,t),\qquad \tilde{F}_\e(x,y,t):=-{F_\e}(x,-y,t),\qquad \tilde{\phi}(x,y,t):={\phi}(x,-y,t),$$
  for $(x,y,t)\in Q_R^-$.
  By changing variables, 
  \begin{equation}\label{eq:Q1+:Q1-}
  \int_{Q_R^+}\rho_\e^a\big(  -u_\e\phi_t+A\nabla u_\e\cdot\nabla\phi-f_\e\phi+F_\e\cdot\D \phi \big)=\int_{Q_R^-}\rho_\e^a\big(  -\tilde{u}_\e\tilde{\phi}_t+\tilde{A}\nabla \tilde{u}_\e\cdot\nabla\tilde{\phi}-\tilde{f}_\e\tilde{\phi}+\tilde{F}_\e\cdot\D \tilde{\phi} \big),
  \end{equation}
  where $Q_R^-:=Q_R\cap\{y<0\}$. Hence, if we define 
$$\bar{u}_\e:=\begin{cases} u_\e, & \text{ in } Q_R^+\\
\tilde{u}_\e,   &    \text{ in } Q_R^-
\end{cases},\qquad \bar{A}:=\begin{cases} A & \text{ in } Q_R^+\\
\tilde{A}   &    \text{ in } Q_R^-
\end{cases},\qquad\bar{f}_\e:=\begin{cases} f_\e & \text{ in } Q_R^+\\
\tilde{f}_\e   &    \text{ in } Q_R^-
\end{cases},\qquad\bar{F}_\e:=\begin{cases} F_\e & \text{ in } Q_R^+\\
\tilde{F}_\e   &    \text{ in } Q_R^-
\end{cases},$$
we have that $\bar{A}$ is a simmetric matrix satisfying \eqref{eq:UnifEll} and, by the conormal boundary condition in \eqref{solution-mezza-palla},  $\bar{u}_\e$ is a weak solution to
$$\rho_\e^a \partial_t \bar{u}_\e -\div(\rho_\e^a\bar{A}\nabla \bar{u}_\e)=\rho_\e^a\bar{f}_\e+\div(\rho_\e^a\bar{F}_\e), \quad \text{ in }Q_R.$$

Then, Lemmas \ref{rho-to-y} and \ref{y-to-rho} apply to $\bar{u}_\e$ in $Q_R$ and, by definition of $\bar{u}_\e$, are valid for weak solutions to \eqref{solution-mezza-palla} in $Q_R^+$.

\end{oss}

\section{Liouville theorems}\label{section5}
The goal of this section is to prove the Liouville type Theorem \ref{teo liouville 1}, for entire solutions to \eqref{eq-liouville-1}.
These results will be obtained using Caccioppoli inequality and the characterization of the entire even solutions of the associated elliptic problem satisfying certain growth conditions, 
extending the Liouville theorems established in \cite{audritoterracini} (see \cites{SirTerVit21a,SirTerVit21b,TerTorVit22} for the elliptic counterpart).

We begin with the following standard lemma.
\begin{lem} \label{x:derivative:solution}
    Let $a\in\mathbb{R}$, $\varepsilon\in[0,1)$ and let $u$ be an entire solution to
    $$\rho_\varepsilon^{a}\partial_t u-{\div}(\rho_\varepsilon^a\nabla u)=0\hspace{0.4cm}\text{in }\mathbb{R}^{N+1}\times\mathbb{R}.$$
    Then, for every $i=1,\dots,N$, the function $\partial_{x_i}u$ is an entire solution to the same problem.
\end{lem}
The proof combines difference quotients in $x$ and energy estimates, similar to the elliptic setting, see \cite{TerTorVit22}*{Corollary 4.2}. The next lemma was established in \cites{StiTor17,banerjee} for $a \in (-1,1)$ and $\e = 0$. We extend it for all values of $a \in \R$ and $\e \in (0,1)$, with an independent proof.
\begin{lem}\label{y:derivative:solution}
Let $a\in\mathbb{R}$, $\varepsilon\in[0,1)$ and let $u$ be an entire solution to
\begin{equation}\label{eq:SolWholeSpaceSimple}
\rho_\varepsilon^{a}\partial_t u - {\div}(\rho_\varepsilon^a\nabla u) = 0 \quad \text{ in } \mathbb{R}^{N+1}\times\mathbb{R}.
\end{equation}
Then the function $v = \rho_\varepsilon^a\partial_y u$ is an entire solution to     
\begin{equation}\label{eq:SolWholeSpaceSimpleConj}
\rho_\varepsilon^{-a}\partial_t v-{\div}(\rho_\varepsilon^{-a}\nabla v) = 0 \quad \text{ in }\mathbb{R}^{N+1}\times\mathbb{R}.
\end{equation}
\end{lem}
\begin{proof} The case $\varepsilon\in(0,1)$ follows by explicit computations, since weak solutions are smooth.

When $\varepsilon = 0$, we proceed by approximation as follows. Fix $R>0$ and let $I_R:=(-R^2,R^2)$. By Lemma \ref{y-to-rho}, there exist a family of solutions $\{u_\e\}_{\e\in(0,1)}$ to \eqref{eq-epsilon-solution} in $Q_{3R}$ (with $f_\e = 0$ and $F_\e = 0$), uniformly bounded in $L^2(I_{2R};H^1(B_{2R},\rho_\varepsilon^a))$, and a sequence $\e_k \to 0$ such that 
\begin{equation}\label{eq:convLiouviApprox}
u_{\varepsilon_k} \to u \quad \text{ in } L^2_{loc}(I_{2R};H^1_{loc}(B_{2R}\setminus\Sigma)), 
\end{equation}
as $k \to +\infty$. Now, since $\e_k > 0$, the function $v_k := \rho_{\varepsilon_k}^a\partial_y u_{\varepsilon_k}$ is a solution to \eqref{eq-epsilon-solution} in $Q_{2R}$ (with $f_\e = 0$ and $F_\e = 0$), with weight $\rho_{\varepsilon_k}^{-a}$. Further, since $\{u_\e\}_{\e\in(0,1)}$ is uniformly bounded in $L^2(I_{2R};H^1(B_{2R},\rho_\varepsilon^a))$, we have
$$
\int_{ Q_{2R}}\rho_{\varepsilon_k}^{-a}v_{k}^2=\int_{ Q_{2R}}\rho_{\varepsilon_k}^{a}(\partial_y u_{\varepsilon_k})^2\le\int_{ Q_{2R}}\rho_{\varepsilon_k}^{a}|\nabla u_{\varepsilon_k}|^2\le C,
$$
for some $C > 0$ independent of $\e$ and thus, using the Caccioppoli inequality \eqref{caccioppoli.inequality}, it follows
\[
\|v_k\|_{L^\infty(I_R;L^2(B_R,\rho_{\varepsilon_k}^{-a}))} + \|\nabla v_k\|_{L^2(Q_R,\rho_{\varepsilon_k}^{-a})} \leq C,
\]
for some new $C > 0$ independent of $\e$. As a consequence, the family $\{v_k\}_{k\in\N}$ satisfies the assumptions of Lemma \ref{rho-to-y} which, in turn, allows us to conclude that, up to a subsequence, $v_k \to v$ in $L^2_{loc}(I_{R};H^1_{loc}(B_{R}\setminus\Sigma))$, for some weak solution $v$ to \eqref{solution} in $Q_R$ (with $\e = 0$, $f=0$ and $F = 0$). By \eqref{eq:convLiouviApprox}, we deduce $v = |y|^a \partial_y u$ and, since $R>0$ is arbitrary, our statement follows.
\end{proof}
\begin{proof}[Proof of Theorem \ref{teo liouville 1}] First, we point out that it is enough to prove that $u$ is linear and depends only on $x$. Then the second part of the statement automatically follows combining \eqref{growth} with the extra-assumption $\gamma\in[0,1)$.

\

\emph{Step 1.}  By Remark \ref{rem:approx:half}, we notice that the even extension w.r.t. $y$ of $u$ is an entire solution to \eqref{eq:SolWholeSpaceSimple}. Therefore, it is enough to establish our statement for an entire solution $u$ to \eqref{eq:SolWholeSpaceSimple} which is even in $y$ and satisfy \eqref{growth} a.e. in $\mathbb{R}^{N+1}\times\mathbb{R}$. Choosing $r'=R$ and $r=2R$ in the Caccioppoli inequality \eqref{caccioppoli.inequality}, we get
\begin{equation}\label{cacioppoli-liouville}
   \int_{Q_R}\rho_\varepsilon^a|\nabla u|^2 dzdt\le \frac{C}{R^2}\int_{Q_{2R}}\rho_\varepsilon^a u^2 dzdt,
\end{equation}
for some $C > 0$ independent of $\e$ and $R$. We will repeatedly use the above inequality in the next steps. 

\

\emph{Step 2.} In this step, we show that $u$ is linear in $x$. By Lemma \ref{x:derivative:solution}, for every multi-index $\beta\in \mathbb{N}^N$, the function $\partial_{x}^\beta u$ solves \eqref{eq-liouville-1}. Fixed $R>1$, by \eqref{cacioppoli-liouville} and \eqref{growth}, it follows 
$$\int_{{Q}_R}\rho_\varepsilon^a(\partial_{x_i}u)^2\le\int_{{Q}_R}\rho_\varepsilon^a|\nabla u|^2\le\frac{C}{R^2}\int_{{Q}_{2R}}\rho_\varepsilon^a u^2\le \frac{C}{R^2} R^{a^++2\gamma+N+3},$$
for every $i=1,\dots,N$. So, setting
\begin{equation}\label{critical-parameter-liouville}
    \tilde{\gamma} := a^++2\gamma+N+3,
\end{equation}
and iterating, it follows
$$
\int_{{Q}_R}\rho_\varepsilon^a(\partial^\beta_{x}u)^2\le C R^{\tilde{\gamma}-2|\beta|},
$$
for every multi-index $\beta\in\mathbb{N}^N$. Consequently, taking $\beta$ such that $2|\beta|>\Tilde{\gamma}$ and passing to the limit as $R\to+\infty$, we get $\partial_x^{\beta}u = 0$, and therefore we easily obtain that $u$ is polynomial in the variable $x$. By \eqref{growth}, it follows that $u$ must be linear in $x$.

\

\emph{Step 3.} In this step we show that $u$ is independent of $y$. By Lemma \ref{y:derivative:solution}, $v := \rho_\varepsilon^a\partial_yu$ is an entire solution to \eqref{eq:SolWholeSpaceSimpleConj} while, by Lemma \ref{y:derivative:solution} again,
\begin{equation}\label{eq:def:w1}
w_1=\rho_\varepsilon^{-a}\partial_y v=\rho_\varepsilon^{-a}\partial_y(\rho_\varepsilon^a\partial_y u)=\partial_{yy}u+\frac{(\rho_\varepsilon^a)'}{\rho_\varepsilon^a}\partial_y u
\end{equation}
is an entire solution to \eqref{eq:SolWholeSpaceSimple}. So, using \eqref{cacioppoli-liouville} twice, we deduce that
\begin{equation}\label{eq:inequality:w_1}
    \int_{{Q}_R}\rho_\varepsilon^a w_1^2 \leq \int_{{Q}_R}\rho_\varepsilon^{-a}|\nabla v|^2  \le \frac{C}{R^2}\int_{{Q}_{2R}}\rho_\varepsilon^{-a}v^2
    \le  \frac{C}{R^2}\int_{{Q}_{2R}}\rho_\varepsilon^{a}|\nabla u|^2\le \frac{C}{R^4}\int_{{Q}_{4R}}\rho_\varepsilon^{a}u^2\le C R^{\tilde{\gamma}-4}.
\end{equation}
Setting 
\begin{equation}\label{w_j-iterazione in y in liouville}
    w_{j+1}:=\partial_{yy}w_{j}+\frac{{(\rho_\varepsilon^a)'}}{{\rho_\varepsilon^a}}\partial_y w_{j},
\end{equation}
and noticing that $w_{j+1}$ is an entire solution to \eqref{eq:SolWholeSpaceSimple} for $j\in\mathbb{N}_{\ge 1}$, we may iterate the argument above to show the existence of $k\in\mathbb{N}$ such that $\tilde{\gamma}-4{k}<0$ and 
$$
\int_{{Q}_R}\rho_\varepsilon^a w_{{k}}\le C R^{\tilde{\gamma}-4{k}}.
$$
Hence, taking the limit as $R\to+\infty$, we obtain $w_k = 0$, that is
$$\partial_{yy}w_{{k-1}}+\frac{(\rho_\varepsilon^a)'}{\rho_\varepsilon^a}\partial_y w_{{k-1}}=0.$$
This ODE can be explicitly solved:
\begin{equation}\label{eq:ODESolStepk-1}
w_{k-1} = c_{2k-1}(x,t)\int_0^y \rho_\varepsilon^{-a}(s)ds+c_{2k-2}(x,t),
\end{equation}
where $c_{2k-1}(x,t)$ and $c_{2k-2}(x,t)$ are unknown functions, linear in $x$. Now, let us define
\begin{equation}\label{eq:def:g_i}
\begin{cases}
       g_1(y)=\int_0^y \rho_\varepsilon^{-a}(s)ds,\\
       g_2(y)={\int}_0^y \rho_\varepsilon^{-a}(s)\int_0^s\rho_\varepsilon^{a}(\tau)d\tau\\
       %g_2(y)=\int_0^y \rho_\varepsilon^{a}(s)(\int_0^s \rho_\varepsilon^{-a}(t)dt)ds,\\
       g_i(y)=\int_0^y \rho_\varepsilon^{-a} (s)\int_0^s \rho_\varepsilon^a(\tau) g_{i-2}(\tau)d\tau, & \text{ for }i\in\mathbb{N}_{\ge3},
       \end{cases}
\end{equation}
which are linked by the relationship
$$ \rho_\varepsilon^{-a}\partial_y(\rho_\varepsilon^{a}\partial_y g_{i})=g_{i-2},\quad\text{for }i\in\mathbb{N}_{\ge3}.  $$
An iterative argument combined with \eqref{eq:ODESolStepk-1} and \eqref{w_j-iterazione in y in liouville} shows that
$$
w_j = c_{2j}(x,t)+\sum_{i=1}^{2(k-j)-1}g_{i}(y)c_{2j+i}(x,t),
$$
for every $j=1,\dots, k-1$, and thus, by \eqref{eq:def:w1}, 
$$
u=c_0(x,t)+\sum_{i=1}^{2k-1}g_i(y)c_i(x,t),
$$
where $c_i(x,t)$ are unknown functions, linear in $x$.

We claim that $c_i\equiv0$ for any $i=1,\dots,2k-1$, which implies that $u$ doesn't depends on $y$.  

First, since $g_i(y)$ are odd functions for odd $i$, one has that $c_i(x,t)\equiv0$ for odd $i$, being $u$ an even function in $y$.
Moreover, for every $i\ge1$ the functions $g_{2i}$ are asymptotically equivalent to $b_iy^{2i}$ for $y\to+\infty$, where $b_i\in\mathbb{R}$. Indeed, by using twice de l'Hôpital rule and by observing that 
$$\lim_{y\to+\infty}\frac{\rho_\varepsilon(y)}{y}=1,$$ we have that
$$\lim_{y\to+\infty}\frac{g_2(y)}{y^2}=\lim_{y\to+\infty}\frac{\rho_\varepsilon^{-a}(y)g_1(y)}{2y}=\lim_{y\to+\infty}\frac{\rho_\varepsilon^{-a}(y)}{y^{-a}}\frac{\int_0^y\rho_\varepsilon^{a}(s)ds}{2y^{1+a}}=\lim_{y\to+\infty}\frac{1}{2(1+a)}\frac{\rho_\varepsilon^a(y)}{y^a}=\frac{1}{2(1+a)}.$$
By using an inductive argument and \eqref{eq:def:g_i}, we can prove that
$$\lim_{|y|\to+\infty}\frac{g_{2i}(y)}{y^{2i}}=b_i,\quad\text{where } b_i={\prod_{m=1}^i\frac{1}{2m(2m-1+a)}}.$$
Hence, $g_{2i}$ is asymptotically equivalent to $b_{i}y^{2i}$ for $y\to+\infty$. This immediately implies that $c_{2i}\equiv0$ for every $i\ge1$, by the parabolic sub-quadratic growth condition \eqref{growth}. Then, $u$ does not depend on $y$ and it is linear in $x$. Using the equation satisfied by $u$, we have that $\partial_t u=0$, hence the thesis is proved, that is, $u=u(x)$ is a linear function.
\end{proof}

\begin{oss}\label{liouville-matrice-non-costante}
Let us highlight that when $a=0$ (and therefore $\rho_\e^a=1$), Theorem \ref{teo liouville 1} remains valid for entire solutions to the heat equation $\partial_t u-\Delta u=0$ in $\mathbb{R}^{N+2}$ and the proof above works in this setting as well, with minor changes. Furthermore, Theorem \ref{teo liouville 1} still holds for entire solutions $u$ to 
 \begin{equation*}\label{eq-liouville-remark}
       \begin{cases}
       \rho_\varepsilon^{a}\partial_t u-{\div}(\rho_\varepsilon^a A \nabla u)=0&\text{in }\mathbb{R}^{N+1}_+\times\mathbb{R},\\
       \rho_\varepsilon^a( A\nabla u)\cdot e_{N+1}=0&\text{on }\partial\mathbb{R}^{N+1}_+\times\mathbb{R},
   \end{cases}
   \end{equation*}
where $A$ is a constant symmetric positive definite matrix (and $u$ satisfies \eqref{growth}). Under such assumptions, $u$ must to be a linear function depending only on $z$. This is a standard result, which immediately follows by a change of coordinates: since $A$ is a symmetric positive definite matrix, we can consider the change of variables $z'=A^{1/2}z$, which allows us to reduce to the case $A = I$.

\end{oss}

\section{Hölder estimates}\label{section6}
In this section we prove the following uniform H\"older bounds.
\begin{teo}\label{teo C^0-alpha}
    Let $N\ge1$, $a>-1$, $p>\frac{N+3+a^+}{2}$, $q>N+3+a^+$, $\alpha\in(0,1)\cap(0,2-\frac{N+3+a^+}{p}]\cap(0,1-\frac{N+3+a^+}{q}]$. Let $A$ be a continuous matrix satisfying \eqref{eq:UnifEll}. As $\varepsilon\to0^+$ let $\{u_\varepsilon\}$ be a family of solutions to
    \begin{equation}\label{eq-epsilon-solution-C-alpha}
    \begin{cases}
        \rho_\varepsilon^{a}\partial_t u_\varepsilon-{\div }(\rho_\varepsilon^{a} A\nabla u_\varepsilon)=\rho_\varepsilon^{a} f_\varepsilon+{\div}(\rho_\varepsilon^{a} F_\varepsilon) &{\rm in  }\hspace{0.1cm}Q_1^+,\\
        \rho_\varepsilon^a(A\nabla u_\varepsilon+F_\varepsilon)\cdot e_{N+1}=0&{\rm on  }\hspace{0.1cm}\partial^0Q_1^+.
    \end{cases}
    \end{equation}
     Then, there exists a constant $C>0$, depending on $N$, $a,$ $\lambda$, $\Lambda$, $p$, $q$ and $\alpha$ such that
     \begin{equation}\label{eq:c1alpha:estimate}
     \|u_\varepsilon\|_{C^{0,\alpha}_p(Q_{1/2}^+)}\le C\left(
    \|u_\varepsilon\|_{L^2(Q_1^+,\rho_\varepsilon^a)}+
    \|f_\varepsilon\|_{L^p(Q_1^+,\rho_\varepsilon^a)}+
    \|F_\varepsilon\|_{L^q(Q_1^+,\rho_\varepsilon^a)}
    \right).
     \end{equation}
\end{teo}

\begin{proof}
Without loss of generality we can assume that there exists a constant $C>0$, which is uniform in $\varepsilon\to0^+$, such that
    $$
    \|u_\varepsilon\|_{L^2(Q_1^+,\rho_\varepsilon^a)}+
    \|f_\varepsilon\|_{L^p(Q_1^+,\rho_\varepsilon^a)}+
    \|F_\varepsilon\|_{L^q(Q_1^+,\rho_\varepsilon^a)}
    \le C.$$
Otherwise \eqref{eq:c1alpha:estimate} is trivially verified.

\

    \emph{Step 1: Contradiction argument and blow-up sequences.}  Consider a cut-off function $\eta\in C_c^\infty(Q_1^+)$ such that
    $$\eta\equiv1\hspace{0.1cm}\text{in }Q_{1/2}^+,\hspace{0.3cm}0\le\eta\le1,\hspace{0.3cm}\supp(\eta)=Q^+_{3/4}.\hspace{0.3cm}$$
    By smoothness of $\eta$, it immediately follows that $\eta\in C^{0,1}_p(Q_1^+)$; that is, there exists a constant $M>0$, which depends only on $N$, such that
    $$|\eta(P)-\eta(Q)|\le Md_p(P,Q),\quad \text{ for every } P=(z,t),Q=(\xi,\tau)\in Q_1^+, $$
   where $d_p(\cdot,\cdot)$ is the parabolic distance, which is defined in \eqref{eq:par:dist}.
    
    We argue by contradiction. Let us suppose that there exist $p>\frac{N+3+a^+}{2}$, $q>N+3+a^+$, $\alpha\in(0,1)\cap (0,2-\frac{N+3+a^+}{p}]\cap(0,1-\frac{N+3+a^+}{q}]$ and a sequence of solutions $\{u_{k}\}_k:=\{u_{\varepsilon_k}\}_k$ as $\varepsilon_k\to0^+$ to \eqref{eq-epsilon-solution-C-alpha}, such that 
    $$L_k:=[\eta u_{k}]_{C^{0,\alpha}_p(Q_1^+)}=\sup_{\substack{P,Q\in Q_1^+\\P\not=Q}}\frac{|(\eta u_k)(P)-(\eta u_k)(Q)|}{d_p(P,Q)^\alpha}\to+\infty.$$
    Now, by the definition of the parabolic Hölder seminorm of $u_k$, we can take two sequences of points $P_k=(z_k,t_k),\bar{P}_k=(\xi_k,\tau_k)\in Q^+_{3/4}$ such that
    $$\frac{|(\eta u_k)(P_k)-(\eta u_k)(\bar{P}_k)|}{d_p(P_k,\bar{P}_k)^\alpha}\ge \frac{1}{2}L_k\to+\infty.$$
    Defining $r_k:=d_p(P_k,\bar{P}_k)$, one has that $r_k\to 0$ as $k\to+\infty$. Indeed, by the local uniform boundedness of solutions, see Proposition \ref{moser-3.5}, one has
    $$\infty\leftarrow L_k\le \frac{4\|\eta u_k\|_{L^\infty(Q^+_1)}}{r_k^\alpha}\le Cr_k^{-\alpha}.$$
    Let $\bar{r}:=4/5$. For $k$ large let us define the blow-up domains
    $$Q(k):=\frac{B^+_{\Bar{r}}-{z_k}}{r_k}\times\frac{(-\Bar{r}^2-t_k,\Bar{r}^2-t_k)}{r_k^2.},$$
    and set $Q^\infty:=\lim_
    {k\to+\infty}Q(k)$ along an appropriate subsequence. We define two blow-up sequences as
    \begin{align}
     \begin{aligned}\label{v_k,w_k_degenere}
     &v_k(z,t):=\frac{\eta(r_k z+ z_k,r_k^2 t+t_k)}{L_k r_k^{\alpha}}(u_k(r_k z+ z_k,r_k^2 t+t_k)-u_k( z_k,t_k)),\\
     &w_k(z,t):=\frac{\eta( z_k,t_k)}{L_k r_k^{\alpha}}(u_k(r_k z+ z_k,r_k^2 t+t_k)-u_k(z_k,t_k)),
\end{aligned}
\end{align}
    for $(z,t)\in Q(k)$.     Then, we distinguish two cases:
    
    \textbf{Case 1:} $$\frac{y_k}{r_k}=\frac{d_p(P_k,\Sigma)}{r_k}\to+\infty,$$ as $k\to+\infty$. In this case we have $Q^\infty=\mathbb{R}^{N+2}$.
    
    \textbf{Case 2:} $$\frac{y_k}{r_k}=\frac{d_p(P_k,\Sigma)}{r_k}\le C,$$ uniformly in $k$. In this case, one has $\frac{y_k}{r_k}\to l$, up to pass to a subsequence and so $Q^\infty=\mathbb{R}^{N}\times\{y\ge l\}\times\mathbb{R}$.
    
    \
    
    \emph{Step 2: Estimate of the parabolic Hölder seminorm of $v_k$.}  Let us fix a compact set $K\subset Q^\infty$. Then, $K\subset Q(k)$ for any $k$ large. For every $P=(z,t)$, $Q=(\xi,\tau)\in K$, $P\not= Q$, we have
   \begin{align*}
    |v_k(z,t)-v_k(\xi,\tau)| &\le\frac{|(\eta u_k)(r_k z+ z_k,r_k^2 t+t_k)-(\eta u_k)(r_k \xi+ z_k,r_k^2 \tau+t_k)|}{L_k r_k^{\alpha}}\\
    &+\frac{|u_k( z_k,t_k)||\eta (r_k z+ z_k,r_k^2 t+t_k)-\eta (r_k \xi+ z_k,r_k^2 \tau+t_k)|}{L_k r_k^{\alpha}}\\
    &\le d_p(P,Q)^\alpha+\frac{\|u_k\|_{L^\infty(Q^+_{3/4})}M d_p((r_k z,r_k^2 t),(r_k \xi,r_k^2 \tau))}{L_k r_k^\alpha}\\
     &\le  d_p(P,Q)^\alpha+\frac{C M r_k^{1-\alpha} d_p(P,Q)}{L_k}.
\end{align*}
Then, as $k\to+\infty$
\begin{equation}\label{stima-hölder-per-C^0,alpha}
    \frac{|v_k(P)-v_k(Q)|}{d_p(P,Q)^\alpha}\le 1+o(1).
\end{equation}

\

\emph{Step 3: The sequences $v_k$ and $w_k$ converge to the same limit $w$.} 
Notice that $v_k(0)=0$ for every $k$. Then, by \eqref{stima-hölder-per-C^0,alpha}, we have that $\|v_k\|_{C^{0,\alpha}_p(K)}$ is uniformly bounded for every compact subset $K\subset Q^\infty$. By the Arzelà-Ascoli theorem, we can pass to a subsequence $v_k$ satisfying $v_k\to w$ uniformly in $K$ and, taking the limit in \eqref{stima-hölder-per-C^0,alpha}, one has $w\in C^{0,\alpha}_p(K)$ with $\|w\|_{C^{0,\alpha}_p(K)}\le1$. 
Moreover, by a countable compact exhaustion of $Q^\infty$, we have that $w$ is globally $C^{0,\alpha}_p$-continuous in $Q^\infty$, that is
\begin{equation}\label{eq:est:w:C^0,alpha}
[w]_{C^{0,\alpha}_p(Q^\infty)}\le 1.
\end{equation}
Furthermore, fixed $K\subset Q^\infty$ compact, for every $P=(z,t)\in K$ one has
\begin{align*}
    |v_k(P)-w_k(P)|&=\frac{|(u_k(r_k z+ z_k,r_k^2 t+t_k)-u_k( z_k,t_k))(\eta_k(r_k z+ z_k,r_k^2 t+t_k)-\eta_k( z_k,t_k))|}{L_k r_k^\alpha}\\
    &\le \frac{2\|u_k\|_{L^\infty(Q^+_{4/5})}r_k M d_p(P,0)}{L_k r_k^\alpha}\rightarrow 0.
\end{align*}
In other words, the sequences $v_k$ and $w_k$ have the same asymptotic behavior as $k\to+\infty$ on $K\subset Q^\infty$ and this implies that $w_k\to w$ uniformly in $K$. 

\

\emph{Step 4: $w$ is not constant.} 
First, $w(0)=0$, since $v_k(0)=0$ for every $k$.
Let us consider the sequence of points $$S_k:=\left(\frac{\xi_k-z_k}{r_k},\frac{\tau_k-t_k}{r_k^2}\right)\in Q(k).$$
Since $d_p(S_k,0)=1$ for any $k$, 
we have $S_k\to\Bar{S}$, up to consider a subsequence. Then, as $k\to+\infty$
\begin{align*}
    |v_k(S_k)|&=\left|\frac{\eta(\bar{P}_k)(u_k(\bar{P}_k)-u_k(P_k))}{L_k r_k^\alpha}\right|\\
    &=\left|\frac{(\eta u_k)(\bar{P}_k)- (\eta u_k)(P_k)+(\eta u_k)(P_k)-\eta (\bar{P}_k)u_k(P_k)}{L_k r_k^\alpha}\right|\\
    &\ge\left|\frac{(\eta u_k)(\bar{P}_k)- (\eta u_k)(P_k)}{L_k r_k^\alpha}\right|-\left|\frac{u_k(P_k)(\eta (\bar{P}_k)-\eta(P_k))}{L_k r_k^\alpha}\right|\\
    &\ge\frac{1}{2}-\frac{\|u_k\|_{L^\infty(Q^+_{3/4})}Mr_k}{L_k r_k^\alpha}=\frac{1}{2}+o(1).
\end{align*}
Then, as $k\to+\infty$, we obtain that $w(\Bar{S})\ge\frac{1}{2}$; that is, $w$ is not constant.

\

\emph{Step 5: $w$ is an entire solution to a homogeneous equation with constant coefficients.} 
First, we observe that, defining $A_k(z,t):=A(r_k z+z_k,r_k^2 t+t_k)$ and $(\Bar{z},\Bar{t}):=\lim_{k\to+\infty}(z_k,t_k)$, by continuity we can define $\bar{A}:=\lim_{k\to+\infty}A_k(z,t)=A(\Bar{z},\Bar{t})$, which is a constant coefficients simmetric matrix satisfying \eqref{eq:UnifEll}. 

Let us consider $\phi\in C_c^\infty(Q^\infty)$, such that $\supp(\phi)\subset Q(k)$ for any $k$ large, and define $\Tilde{\phi}(z,t):=\phi(\frac{z-z_k}{r_k},\frac{t-t_k}{r_k^2})\in C_c^\infty(Q_1^+)$. Since $u_k$ is a solution to \eqref{eq-epsilon-solution-C-alpha}, by explicit computations, we have
\begin{align*}
    &-\int_{Q(k)}\rho_{\varepsilon_k}^a(r_k y+ y_k) w_k\partial_t \phi +
    \int_{Q(k)}\rho_{\varepsilon_k}^a(r_k y+ y_k) A_k\nabla w_k\cdot\nabla \phi  \\
    &= \frac{(\eta u_k)(z_k,t_k)}{L_k r_k^\alpha}\int_{Q(k)}\rho_{\varepsilon_k}^a(r_k y+ y_k) \partial_t \phi 
    +
    \frac{\eta(z_k,t_k)}{L_k r_k^\alpha}\Big(
    -\int_{Q_1^+}\rho_{\varepsilon_k}^a(y) u_k \partial_t \Tilde{\phi}
    +\int_{Q_1^+}\rho_{\varepsilon_k}^a(y) A \nabla u_k\cdot\nabla \Tilde{\phi}
    \Big)r_k^{-N-1}\\
    &= \frac{\eta(z_k,t_k)r_k^{-N-1-\alpha}}{L_k}\int_{Q_1^+}\rho_{\varepsilon_k}^a(y) (f_{\varepsilon_k}\Tilde{\phi}-F_{\varepsilon_k}\cdot\nabla\Tilde{\phi})\\
    &=\frac{\eta(z_k,t_k)r_k^{2-\alpha}}{L_k}\int_{Q(k)}\rho_{\varepsilon_k}^a(r_k y+ y_k) f_{\varepsilon_k}(r_k z+z_k,r_k^2 t+t_k)\phi  \\  
    &+\frac{\eta(z_k,t_k)r_k^{1-\alpha}}{L_k}\int_{Q(k)}\rho_{\varepsilon_k}^a(r_k y+ y_k)F_{\varepsilon_k} (r_k z+z_k,r_k^2 t+t_k)\cdot\nabla\phi.
\end{align*}
So, $w_k$ is a solution in $Q(k)$ to
\begin{align}
\begin{aligned}\label{equazione-w_k-c^alpha}
    \rho_{\varepsilon_k}^{a}(r_k \cdot+y_k)\partial_t w_k-{\div}(\rho_{\varepsilon_k}^{a}(r_k \cdot+ y_k)A_k\nabla w_k)
    &=\rho_{\varepsilon_k}^{a}(r_k \cdot+y_k)\frac{\eta(z_k,t_k)r_k^{2-\alpha}}{L_k}f_{\varepsilon_k}(r_k\cdot+z_k,r_k^2\cdot+t_k)\\
    &+\frac{\eta(z_k,t_k)r_k^{1-\alpha}}{L_k}{\div}(\rho_{\varepsilon_k}^{a}(r_k \cdot+ y_k)F_{\varepsilon_k}(r_k\cdot+z_k,r_k^2\cdot+t_k)).
\end{aligned}
\end{align}
Notice that in \textbf{Case 2} the function $w_k$ satisfies a conormal boundary condition on the hyperplane $\{y=\frac{y_k}{r_k}\}$ too.

Next, we normalize the equation \eqref{equazione-w_k-c^alpha} in the following way: let us define $\Gamma_k:=(\varepsilon_k,y_k,r_k)$ and $\nu_k:=|\Gamma_k|$, which is bounded from above, since $r_k\to0$, $\varepsilon_k\to0$ and $y_k\to\Bar{y}\in[0,1)$. Let 
$$\Tilde{\Gamma}_k:=\frac{\Gamma_k}{\nu_k}=\left(\frac{\varepsilon_k}{\nu_k},\frac{y_k}{\nu_k},\frac{r_k}{\nu_k}\right)=(\Tilde{\varepsilon_k},\Tilde{y}_k,\Tilde{r}_k).$$
Since $|\Tilde{\Gamma_k}|=1$ for every $k$, up to consider a subsequence,
$\Tilde{\Gamma}_k\to\Tilde{\Gamma}=(\Tilde{\varepsilon},\Tilde{y},\Tilde{r})$. Denoting 
$$\Tilde{\rho}_k^a(y):=\frac{\rho_{\varepsilon_k}^a(r_k y+y_k)}{\nu_k^a}=(\Tilde{\varepsilon}_k^2+(\Tilde{r}_k y+\Tilde{y}_k)^2)^{a/2},$$
and
$$\Tilde{\rho}^a(y):=(\Tilde{\varepsilon}^2+(\Tilde{r} y+\Tilde{y})^2)^{a/2},$$
we have that $\Tilde{\rho}_{k}^a\to\tilde{\rho}^a$ a.e. in $Q^\infty$. By multiplying the equation \eqref{equazione-w_k-c^alpha} by $\nu_k^{-a}$ we get that $w_k$ solves
\begin{align}
\begin{aligned}\label{equazione-w_k-c^alpha-normalized}
    \tilde{\rho}_k^{a}\partial_t w_k-{\div}(\tilde{\rho}_k^{a}A_k\nabla w_k)
    &=\tilde{\rho}_k^{a}\frac{\eta(z_k,t_k)r_k^{2-\alpha}}{L_k}f_{\varepsilon_k}(r_k\cdot+z_k,r_k^2\cdot+t_k)\\
    &+\frac{\eta(z_k,t_k)r_k^{1-\alpha}}{L_k}{\div}(\tilde{\rho}_k^{a}F_{\varepsilon_k}(r_k\cdot+z_k,r_k^2\cdot+t_k)).
\end{aligned}
\end{align}
We claim that the r.h.s. of \eqref{equazione-w_k-c^alpha-normalized} vanishes in a distributional sense as $k\to+\infty$. 
Indeed, fixed  $\phi\in C_c^\infty(Q^\infty)$  with $\supp(\phi)\subset Q(k)$ for any $k$ large, we have the following estimate
\begin{align*}
   &\Big|\int_{\supp(\phi)}\rho_{\varepsilon_k}^a(r_k y+ y_k)f_{\varepsilon_k}(r_k z+z_k,r_k^2 t+t_k)\phi(z,t)dzdt
    \Big|\\
    & \le\|\phi\|_{L^\infty(\mathbb{R}^{N+2})}\Big(\int_{\supp(\phi)}\rho_{\varepsilon_k}^a(r_k y+ y_k)|f_{\varepsilon_k}(r_k z+z_k,r_k^2 t+t_k)|^pdzdt
    \Big)^{1/p}\cdot\Big(\int_{\supp(\phi)}\rho_{\varepsilon_k}^a(r_k y+ y_k)dzdt\Big)^{1/p'}\\
    & \le C\|\phi\|_{L^\infty(\mathbb{R}^{N+2})}\Big(\int_{Q_1^+}(\varepsilon_k^2+ \xi_{N+1}^2)^{a/2}|f_{\varepsilon_k}(\xi,\tau)|^p r_k^{-(N+3)}d\xi d\tau
    \Big)^{1/p} \nu_k^{a/p'}\\
    & \le C\|f_{\varepsilon_k}\|_{L^p(Q_1^+,\rho_\varepsilon^a)}r_k^{-\frac{N+3}{p}}\nu_k^{\frac{a}{p'}} \le C r_k^{-\frac{N+3}{p}}\nu_k^{\frac{a}{p'}}.
\end{align*}
So, we can estimate the first member of the r.h.s. of \eqref{equazione-w_k-c^alpha-normalized} as follows
\begin{align*}
&\frac{\eta(z_k,t_k)r_k^{2-\alpha}\nu_k^{-a}}{L_k}\Big|\int_{\supp(\phi)}\rho_{\varepsilon_k}^a(r_k y+ y_k)f_{\varepsilon_k}(r_k z+z_k,r_k^2 t+t_k)\phi(z,t)dzdt
    \Big|\\    
    &\le C \nu_k^{-a} \frac{\eta(z_k,t_k)r_k^{2-\alpha}}{L_k}r_k^{-\frac{N+3}{p}}\nu_k^{\frac{a}{p'}}\le C r_k^{2-\alpha-\frac{N+3+a^+}{p}}\left(\frac{r_k^{a^+}}{\nu_k^a}\right)^{1/p}\to0,
\end{align*}
since $r_k\le\nu_k$ and  $\alpha<2-\frac{N+3+a^+}{p}$. Similarly, the second term of the r.h.s. of \eqref{equazione-w_k-c^alpha-normalized} vanishes as well.

Finally, we prove that the l.h.s. of \eqref{equazione-w_k-c^alpha-normalized} converges in the following sense
\begin{equation}\label{eq:C^alpha:conv:lhs}
\int_{Q^\infty}\tilde{\rho}_k^a (-w_k\partial_t \phi+A_k \nabla w_k\cdot\nabla\phi)\to\int_{Q^\infty}\tilde{\rho}^a (-w\partial_t \phi+\bar{A} \nabla w\cdot\nabla\phi).
\end{equation}
Let us fix $R>0$ such that $\supp({\phi})\subset Q_R\cap Q^\infty$ and observe that $Q^\infty=B^\infty\times\mathbb{R}$. Since $\{w_k\}$ is uniformly bounded in $L^\infty(Q_{2R}\cap Q^\infty)$ one has that $\{w_k\}$ is uniformly bounded in $L^2(Q_{2R}\cap Q^\infty,\tilde{\rho}_k^a)$. Then, by using the Caccioppoli inequality \eqref{caccioppoli.inequality}, we get that $\{w_k\}$ is uniformly bounded in $L^2(-R^2,R^2;H^1(B_R\cap B^\infty,\Tilde{\rho}_k^a))\cap L^\infty(-R^2,R^2;L^2(B_R\cap B^\infty,\Tilde{\rho}_k^a))$. Using the a.e. convergences $A_k(z,t)\to \bar{A}$ and $\tilde{\rho}_k^a\to\Tilde{\rho}^a$, we are able to apply Lemma \ref{rho-to-y}, with minor changes, and the convergence \eqref{eq:C^alpha:conv:lhs} holds.

This convergence, combined with the previous ones, tells us that $w$ is an entire solution to
\begin{equation}\label{equazione omogenea C-alpha Case1}
\Tilde{\rho}^a \partial_t w-{\div}(\Tilde{\rho}^a \bar{A}\nabla w)=0,\hspace{0.4cm}\text{in }\mathbb{R}^{N+2},
\end{equation}
in \textbf{Case 1} while, in \textbf{Case 2}, $w$ is an entire solution to
\begin{equation}\label{equazione omogenea C-alpha Case2}
       \begin{cases}
       \Tilde{\rho}^a \partial_t w-{\div}(\Tilde{\rho}^a \bar{A}\nabla w)=0,&\text{in }\mathbb{R}^{N+1}_+\times\mathbb{R},\\
       \Tilde{\rho}^a \bar{A}\nabla w\cdot e_{N+1}=0&\text{on }\mathbb{R}^{N}\times\{y=l\}\times\mathbb{R}.
   \end{cases}
   \end{equation}

\

\emph{Step 6: Liouville theorems.} Summarizing, we have that $w$ solves \eqref{equazione omogenea C-alpha Case1} or \eqref{equazione omogenea C-alpha Case2}, is globally $C^{0,\alpha}_p$-continuous in $Q^\infty$ and is not constant. By the global $C^{0,\alpha}_p$-continuity \eqref{eq:est:w:C^0,alpha}, it follows that 
$$|w(z,t)|\le|w(z,t)-w(0,0)|+|w(0,0)|\le (|z|^2+|t|)^{\alpha/2}.$$
In \textbf{Case 1}, since $\frac{y_k}{r_k}\to+\infty$, we have
\begin{align*}
    \tilde{\rho}_k^a(y)&=\left(
\frac{1}{\nu_k^2}\left(\varepsilon_k^2+y_k^2\left(\frac{r_k}{y_k}y+1\right)^2\right)\right)^{a/2} =
\left(\tilde{\varepsilon}_k^2+\tilde{y_k}^2\left(\frac{r_k}{y_k}y+1\right)^2\right)^{a/2}\to(\Tilde{\varepsilon}^2+\Tilde{y}^2)^{a/2},
\end{align*}
which is a positive constant. Then, 
by the classical Liouville theorem for the heat equation, see Remark \ref{liouville-matrice-non-costante}, and the above growth condition,
 the solution $w$ must be constant and this is a contradiction. 
 
In \textbf{Case 2}, $\frac{y_k}{r_k}\le C$, uniformly in $k$ and $\frac{\Tilde{y_k}}{\Tilde{r_k}}=\frac{y_k}{r_k}\to \frac{\Tilde{y}}{\Tilde{r}}=l$. Up to consider a translation of $\frac{\Tilde{y}}{\Tilde{r}}=l$, we can assume $\Tilde{y}=0$ and then $\Tilde{\rho}^a(y)=(\Tilde{\varepsilon}^2+\Tilde{r}^2y^2)^{a/2}$. There are three possibilities:
\begin{itemize}
    \item $\tilde{\varepsilon}=0,\tilde{r}\not=0$, $\Tilde{\rho}^a(y)=|y|^a$.
    \item $\tilde{\varepsilon}\not=0,\tilde{r}=0$, $\Tilde{\rho}^a(y)=1.$
    \item $\tilde{\varepsilon}\not=0,\tilde{r}\not=0$, $\Tilde{\rho}^a(y)=(1+y^2)^{a/2}$, up to a dilation of $\frac{\tilde\varepsilon}{\Tilde{r}}$.
\end{itemize}
In any case, we can invoke Liouville Theorem \ref{teo liouville 1} in $\mathbb{R}^{N+1}_+\times\mathbb{R}$ and by Remark \ref{liouville-matrice-non-costante} we obtain again a contradiction.
\end{proof}

\section{Hölder estimates for the gradient}\label{section7}

\begin{teo}\label{teo-C1 alpha epsilon}
    Let $N\ge1$, $a>-1$, $p>N+3+a^+$, $\alpha\in (0,1-\frac{N+3+a^+}{p})$. Let $A\in C^{0,\alpha}_p(Q_1^+)$ be a matrix satisfying \eqref{eq:UnifEll}. As $\varepsilon\to0$ let $\{u_\varepsilon\}$ be a family of solutions to
    \begin{equation}\label{eq-epsilon-solution-C1-alpha}
        \begin{cases}
        \rho_\varepsilon^{a}\partial_t u_\varepsilon-{\div }(\rho_\varepsilon^{a} A\nabla u_\varepsilon)=\rho_\varepsilon^{a}f_\varepsilon+{\div}(\rho_\varepsilon^{a} F_\varepsilon) &{\rm in  }\hspace{0.1cm}Q_1^+,\\
        \rho_\varepsilon^a(A\nabla u_\varepsilon+F)\cdot e_{N+1}=0&{\rm on  }\hspace{0.1cm}\partial^0Q_1^+.
    \end{cases}
    \end{equation}
Then, there exists a constant $C>0$ depending on $N$, $a$, $\lambda$, $\Lambda$, $p$, $\alpha$ and $\|A\|_{C^{0,\alpha}_p(Q_{1}^+)}$ such that
    $$\|u_\varepsilon\|_{C^{1,\alpha}_p(Q_{1/2}^+)}\le C\left(
    \|u_\varepsilon\|_{L^2(Q_1^+,\rho_\varepsilon^a)}+
    \|f_\varepsilon\|_{L^p(Q_1^+,\rho_\varepsilon^a)}+
    \|F_\varepsilon\|_{C^{0,\alpha}_p(Q_1^+)}
    \right).$$
\end{teo}
\begin{proof}
To simplify the notation, let $\partial_i :=\partial_{x_i} $ for $i=1,\dots,N$ and $\partial_{N+1}:=\partial_y$.  As in Theorem \ref{teo C^0-alpha}, without loss of generality, we can assume that there exists $C>0$, which is uniform in $\varepsilon\to0^+$, such that
    $$
    \|u_\varepsilon\|_{L^2(Q_1^+,\rho_\varepsilon^a)}+
    \|f_\varepsilon\|_{L^p(Q_1^+,\rho_\varepsilon^a)}+
    \|F_\varepsilon\|_{C^{0,\alpha}_p(Q_1^+)}
    \le C.$$
    
    \
    
    \emph{Step 1: Contradiction argument and blow-up sequences.} Consider a cut-off function $\eta\in C_c^\infty(Q_1^+)$ such that
    $$\eta\equiv1\hspace{0.1cm}\text{in }Q_{1/2}^+,\hspace{0.3cm}0\le\eta\le1,\hspace{0.3cm}\supp(\eta)=Q^+_{3/4}.\hspace{0.3cm}$$
    By smoothness of $\eta$, it immediately follows that $\eta\in C^{1,1}_p(Q_1^+)$; that is, there exists a constant $M>0$, which depends only on $N$, such that $\|\eta\|_{C^{1,1}_p(Q_1^+)}\le M$.

    By contradiction, let us suppose that there exist  $p>N+3+a^+$, $\alpha\in(0,1-\frac{N+3+a^+}{p})$ and a sequence of solutions $\{u_k\}:=\{u_{\varepsilon_k}\}$ as $\varepsilon_k\to0^+$ to \eqref{eq-epsilon-solution-C1-alpha}, such that
    $$\| \eta u_k\|_{C_p^{1,\alpha}(Q_{1}^+)}\to+\infty.$$
Define
  \begin{align*}
      L_k:=\max\Big\{\{[\partial_{i}(\eta u_k)]_{C^{0,\alpha}_p(Q^+_1)}:i=1,\dots,N+1\}, [\eta u_k]_{C_t^{0,\frac{1+\alpha}{2}}(Q^+_1)}\Big\},
  \end{align*}
and distinguish two cases: first, let us suppose that there exists $i\in\{1,\dots,N+1\}$ such that $L_k=[\partial_{i}(\eta u_k)]_{C^{0,\alpha}_p(Q^+_{1})}$ 
(later we will deal  with the second case, when $L_k=[\eta u_k]_{C_t^{0,\frac{1+\alpha}{2}}(Q^+_1)}$). Notice that it cannot be $\|\nabla(\eta u_k)\|_{L^\infty(Q_{1}^+)}\to+\infty$ and $[\eta u_k]_{C_p^{1,\alpha}(Q_{1}^+)}$ remains bounded, since the functions $\eta u_k$ are identically zero outside $Q_{3/4}^+$, for every $k$.

Next, we take two sequences of points $P_k=(z_k,t_k), \bar{P}_k=(\xi_k,\tau_k)\in Q^+_{3/4}$ such that 
$$\frac{|\partial_{i}(\eta u_k)(P_k)-\partial_{i}(\eta u_k)(\bar{P}_k)|}{d_p(P_k,\bar{P}_k)^\alpha}\ge \frac{1}{2}L_k\to+\infty. $$
Let $r_k:=d_p(P_k,\bar{P}_k)$, $\hat{z}_k:=(\hat{x}_k,\hat{y}_k)\in B_{3/4}^+$ be a sequence of points which will specify below. Let $\Bar{r}:=4/5$. For $k$ large let us define 
$$Q(k):=\frac{B^+_{\Bar{r}}-{\hat{z}_k}}{r_k}\times\frac{(-\Bar{r}^2-t_k,\Bar{r}^2-t_k)}{r_k^2.},$$
and set $Q^\infty:=\lim_{k\to+\infty}Q(k)$. We define two blow-up sequences as follows
\begin{align}
\begin{aligned}\label{v_k,w_k_degenerate}
&v_k(z,t):=\frac{\eta(r_k z+\hat{z}_k,r_k^2 t+t_k)}{L_k r_k^{1+\alpha}}(u_k(r_k z+\hat{z}_k,r_k^2 t+t_k)-u_k(\hat{z}_k,t_k)),\\
    &w_k(z,t):=\frac{\eta(\hat{z}_k,t_k)}{L_k r_k^{1+\alpha}}(u_k(r_k z+\hat{z}_k,r_k^2 t+t_k)-u_k(\hat{z}_k,t_k)),
\end{aligned}
\end{align}
for $(z,t)\in Q(k)$.
Then, we distinguish two cases:

    \textbf{Case 1:} $$\frac{y_k}{r_k}=\frac{d_p(P_k,\Sigma)}{r_k}\to+\infty,$$ 
as $k\to+\infty$. Since $y_k$ is uniformly bounded, we have that $r_k\to 0$ and $Q^\infty=\mathbb{R}^{N+2}$. In this case we set $\hat{z}_k=z_k$.

    \textbf{Case 2:} 
    $$\frac{y_k}{r_k}=\frac{d_p(P_k,\Sigma)}{r_k}\le C,$$ 
uniformly in $k$. We set $\hat{z}_k=(x_k,0)$ and we will show later that also in this case $r_k\to0$, which implies that $Q^\infty=\R^{N+1}_+\times\R$.

\

    \emph{Step 2: Parabolic Hölder estimates.}
    Let us fix a compact set $K\subset Q^\infty$. Then, $K\subset Q(k)$ for any $k$ large. For every $P=(z,t), Q=(\xi,\tau)\in K$, $P\not= Q$ and for every $j=1,\dots,N+1$, we have 
\begin{align*}
   |\partial_{j} v_k(P)-\partial_{j} v_k(Q)|&\le \frac{|\partial_{j}(\eta u_k)(r_k z+\hat{z}_k,r_k^2 t+t_k)-\partial_{j}(\eta u_k)(r_k \xi+\hat{z}_k,r_k^2 \tau+t_k)|}{L_kr_k^\alpha}\\
    &+\frac{|u_k(\hat{z}_k,t_k)||\partial_{j} \eta(r_k z+\hat{z}_k,r_k^2 t+t_k)-\partial_{j}\eta(r_k \xi+\hat{z}_k,r_k^2 \tau+t_k)|}{L_k r_k^\alpha}\\
    &\le  \frac{[\partial_j(\eta u_k)]_{C^{0,\alpha}_p(Q_1^+)}d_p(P,Q)^\alpha}{L_k}+\frac{\|u_k\|_{L^\infty(Q_{3/4}^+)}r_k^{1-\alpha} M d_p(P,Q)}{L_k}\\
    &\le  d_p(P,Q)^\alpha+\frac{CM}{L_k},
\end{align*}
since $[\partial_j(\eta u_k)]_{C^{0,\alpha}_p(Q_1^+)}\le L_k$, $r_k\le C$, $\|u_k\|_{L^\infty(Q^+_{3/4})}\le C$ and $d_p(P,Q)^{1-\alpha}\le C$ in $K$. By dividing the previous inequality  by $d_p(P,Q)^\alpha$ and using $L_k\to+\infty$, we get
\begin{equation}\label{5.1-degenerate}
    \sup_{\substack{P,Q\in K\\P\not=Q}}\frac{ |\partial_{j} v_k(P)-\partial_{j} v_k(Q)|}{d_p(P,Q)^\alpha}\le 1+o(1).
\end{equation}
as $k\to+\infty$. On the other hand, for every $P=(z,t),Q=(z,\tau)\in K$, $t\not= \tau$, we have that
\begin{align*}
    | v_k(P)- v_k(Q)|&\le \frac{|(\eta u_k)(r_k z+\hat{z}_k,r_k^2 t+t_k)-(\eta u_k)(r_k z+\hat{z}_k,r_k^2 \tau+t_k)|}{L_kr_k^{1+\alpha}}\\
    &+\frac{|u_k(\hat{z}_k,t_k)||\eta(r_k z+\hat{z}_k,r_k^2 t+t_k)-\eta(r_k z+x_k,r_k^2 \tau+t_k)|}{L_k r_k^{1+\alpha}}\\
    &\le  |t-\tau|^\frac{1+\alpha}{2}+\frac{\|u_k\|_{L^\infty(Q^+_{3/4})}r_k^{1-\alpha} M |t-\tau|}{L_k},
\end{align*}
so  
\begin{equation}\label{5.1-degenrate time}
    \sup_{\substack{(z,t),(z,\tau)\in K\\t\not=\tau}}\frac{|v_k(z,t)-v_k(z,\tau)|}{|t-\tau|^{\frac{1+\alpha}{2}}}\le 1+o(1).
\end{equation}
Putting together this inequality with \eqref{5.1-degenerate}, we obtain the uniform boundedness of $[v_k]_{C^{1,\alpha}_p(K)}$, noticing that these considerations are valid in both \textbf{Case 1} and \textbf{Case 2}.

\

\emph{Step 3: Convergence of blow-ups.} For $P=(z,t)\in Q(k)$, let us define
\begin{equation}\label{eq:bar:v}
\Bar{v}_k(P):=v_k(P)-\nabla v_k(0)\cdot z,\hspace{0.4cm}\Bar{w}_k(P):=w_k(P)-\nabla w_k(0)\cdot z.
\end{equation}
Notice that $\Bar{v}_k(0)=0=\Bar{w}_k(0)$ and $|\nabla \Bar{v}_k(0)|=0=|\nabla \Bar{w}_k(0)|$.
For every $K\subset Q^\infty$ compact, since $[\bar{v}_k]_{C^{1,\alpha}_p(K)}=[{v}_k]_{C^{1,\alpha}_p(K)}$, we have that $\|\bar{v}_k\|_{C^{1,\alpha}_p(K)}$ is uniformly bounded. Then, we can apply the Arzelá-Ascoli theorem and infer that $\bar{v}_k\to\Bar{v}$ in $C^{1,\gamma}_p(K)$, for any $\gamma<\alpha$. Now, passing to the limit in \eqref{5.1-degenerate} and in \eqref{5.1-degenrate time} and by a countable compact exhaustion of $Q^\infty$, we obtain that the limit function $\Bar{v}$ satisfies
\begin{equation*}
[\bar{v}]_{C^{1,\alpha}_p(Q^\infty)}\le C,
\end{equation*}
that is, $\bar{v}$ is globally $C^{1,\alpha}_p$-continuous in $Q^\infty$.

Next, we want to show that the sequence $\{\bar{w}_k\}$ converges uniformly to $\bar{v}$ on compact sets. Let us fix $K\subset Q^\infty$, such that $K\subset Q(k)$ for any $k$ large. Since $\nabla \bar v_k(0)=\nabla \bar w_k(0)$, for every $P=(z,t)\in K$, we have
\begin{align*}
    |\Bar{v}_k(P)-\Bar{w}_k(P)|&=|v_k(P)-w_k(P)|\\
    &\le\frac{|\eta(r_k z+\hat{z}_k,r_k^2 t+t_k)-\eta(\hat{z}_k,t_k)|\cdot |u_k(r_k z+\hat{z}_k,r_k^2 t+t_k)-u_k(\hat{z}_k,t_k)|}{L_k r_k^{1+\alpha}}\\
    &\le \frac{Cr_k d_p(P,0)\cdot Mr_k^\alpha d_p(P,0)^{\alpha}}{L_k r_k^{1+\alpha}} =\frac{CM d_p(P,0)^{1+\alpha}}{L_k}\to 0,
    \end{align*}
as $k\to+\infty$, by the properties of $\eta$ and the Theorem \ref{teo C^0-alpha}, which ensures local uniform bound of $u_k$ in $C^{0,\alpha}_p$-space. This implies that $\Bar{w}_k\to\Bar{v}$ uniformly in $K$.

\

\emph{Step 4: $\nabla\Bar{v}$ is not constant.} Let us define two sequences of points as
$$S_k:=\left(\frac{\xi_k-\hat{z}_k}{r_k},\frac{\tau_k-t_k}{r_k^2}\right),\quad\Bar{S_k}:=\left(\frac{z_k-\hat{z}_k}{r_k},0\right)\in Q(k).$$
In \textbf{Case 1}, one has $\hat{z}_k=z_k$, then $S_k\to S\in Q^\infty$, up to consider a subsequence, and $\Bar{S}_k=0$. Let $i\in\{1,\dots,N+1\}$ be the one that realizes the maximum of $L_k$. We can compute, as $k\to+\infty$
\begin{align*}
    |\partial_{i}\bar{v}_k(S_k)-\partial_{i}\bar{v}_k(\bar{S}_k)|&=|\partial_{i}{v}_k(S_k)-\partial_{i}v_k(0)|\\
    &=\frac{|\partial_{i}(\eta u_k)(\bar{P}_k)-\partial_{i}(\eta u_k)(P_k)-u_k(P_k)(\partial_{i}\eta(\bar{P}_k)-\partial_{i}\eta(P_k))|}{L_kr_k^\alpha}\\
    &\ge \frac{1}{2}-\frac{\|u_k\|_{L^\infty(Q^+_{3/4})}M r_k^{1-\alpha}}{L_k}=\frac{1}{2}+o(1).
\end{align*}
Then, as $k\to+\infty$, we  obtain that $|\partial_{i}\Bar{v}(S)-\partial_{i}\Bar{v}(0)|\ge\frac{1}{2}$, which implies that $\nabla \bar{v}$ is not constant.

Instead, in \textbf{Case 2}, we have $\bar{S}_k=\frac{y_k}{r_k}e_{n+1},$ which converge to a point $\Bar{S}$, up to consider a subsequence, by the fact that $\frac{y_k}{r_k}\le C$ uniformly in $k$. The sequence $S_k$ can be written as 
$$S_k=\left(\frac{\xi_k-{z}_k}{r_k},\frac{\tau_k-t_k}{r_k^2}\right)+\frac{y_k}{r_k}e_{N+1},$$
and still converges, up to a subsequence, to a point $S\in Q^\infty$. So, also in this case, we have
\begin{align*}
    |\partial_{i}\bar{v}_k(S_k)-\partial_{i}\bar{v}_k(\bar{S}_k)|=|\partial_{i}{v}_k(S_k)-\partial_{i}v_k(0)|
    %&=\frac{1}{L_kr_k^\alpha}|\partial_{x_i}(\eta u_k)(\bar{P}_k)-\partial_{x_i}(\eta u_k)(P_k)-u_k(P_k)(\partial_{x_i}\eta(\bar{P}_k)-\partial_{x_i}\eta(P_k))|\\
    \ge \frac{1}{2}+o(1),
\end{align*}
which allows us to conclude that $\Bar{v}$ has non constant gradient exactly as in \textbf{Case 1}.

\

\emph{Step 5: $r_k\to0$ in  \textbf{Case 2}.} By contradiction, let us suppose that, up to consider a subsequence, $r_k\to\Tilde{r}>0$ in \textbf{Case 2}. Then, if $K\subset Q^\infty$ is a fixed compact set, we have
$$\sup_{P\in K}|{v_k(P)}|\le 2\frac{\|\eta\|_{L^\infty(Q^+_1)}\|u_k\|_{L^\infty(Q^+_{3/4})}}{L_k r_k^{1+\alpha}}\to0,$$
as $k\to+\infty$, which means that $v_k\to0$ uniformly on compact sets of $Q^\infty$. For every $P=(z,t)\in K$, by using the convergence $\Bar{v}_k\to\bar{v}$ obtained in \emph{Step 3}, one has 
$$\Bar{v}(P)=\lim_{k\to+\infty}\nabla v_k(0)\cdot z.$$
We claim that the sequence $\{\nabla v_k(0)\}_k$ is bounded. Indeed, assume by contradiction that there exists $j\in\{1,\dots,N+1\}$ such that $\{\partial_{j}v_k(0)\}$ is unbounded. Fix $R>0$ sufficiently small such that $Q_R^+$ is contained in $Q^\infty$. Then
$$|\Bar{v}(R e_j)|=R \, \lim_{k\to+\infty}|\nabla v_k(0)\cdot e_j|=R|\partial_{j}v_k(0)|\to+\infty,$$
which is in contradiction to the fact $\bar{v}\in C^{1,\alpha}_p(Q_R^+)$ and hence bounded in $Q_R^+$. Hence, up to consider a subsequence, we have that $\nabla v_k(0)\to \nu\in\mathbb{R}^{N+1}$ and $\Bar{v}(z,t)=\nu\cdot z$, which is in contradiction to the fact that $\Bar{v}$ has non constant gradient. So, we have shown that $r_k\to0$ also in \textbf{Case 2}, which implies that$Q^\infty=\mathbb{R}^{N+1}_+\times\mathbb{R}$.

\

\emph{Step 6: $\bar{v}$ is an entire solution to a homogeneous equation with constant coefficients.} First, we look at the equation satisfied by $\Bar{w}_k$ in $Q(k)$. As in Theorem \ref{teo C^0-alpha}, let us define $\nu_k=|(\varepsilon_k,\hat{y}_k,r_k)|$ and $(\tilde{\varepsilon}_k,\tilde{\hat{y}}_k,\tilde{r}_k)=(\frac{\varepsilon_k}{\nu_k},\frac{\hat{y}_k}{\nu_k},\frac{r_k}{\nu_k})$, which converges, up to consider a subsequence, to $(\Tilde{\varepsilon},\Tilde{y},\Tilde{r})$. Defining 
$$\tilde{\rho}^a_k(y)=\frac{\rho_{\varepsilon_k}^a(r_ky+\tilde{y}_k)}{\nu_k^a}=(\Tilde{\varepsilon}_k^2+(\Tilde{r}_k y+\Tilde{\hat{y}}_k)^2)^{a/2},$$ 
and 
$$\Tilde{\rho}^a(y):=(\Tilde{\varepsilon}^2+(\Tilde{r} y+\Tilde{y})^2)^{a/2},
$$
we have that $\Tilde{\rho}_k^a\to\Tilde{\rho}^a$ a.e. in $ Q^\infty$.

Let us fix $\phi\in C_c^\infty(Q^\infty)$, with $\supp(\phi)\subset Q(k)$ for any $k$ large. Then,
\begin{align}\label{equazione C^1-alpha per le w_k}
\begin{aligned}
    &\int_{\supp(\phi)}\tilde{\rho}_{k}^a(y)\Big(-\Bar{w}_k\partial_t\phi+A(r_k z+\hat{z}_k,r_k^2 t+t_k)\nabla\bar{w}_k\cdot\nabla\phi\Big)\\
    &=\frac{\eta(\hat{z}_k,t_k)r_k^{1-\alpha}\nu_k^{-a}}{L_k }\int_{\supp(\phi)}\rho_{\varepsilon_k}^a(r_ky+\hat{y}_k)f_{\varepsilon_k}(r_k z+\hat{z}_k,r_k^2 t+t_k)\phi\\
    &-\frac{\eta(\hat{z}_k,t_k)\nu_k^{-a}}{L_k r_k^\alpha}\int_{\supp(\phi)}\rho_{\varepsilon_k}^a(r_ky+\hat{y}_k)\Big( F_{\varepsilon_k}(r_k z+\hat{z}_k,r_k^2 t+t_k)-F_{\varepsilon_k}(\hat{z}_k,t_k)
    \Big)\cdot\nabla\phi\\
    &-\frac{\eta(\hat{z}_k,t_k)\nu_k^{-a}}{L_k r_k^\alpha}\int_{\supp(\phi)}\rho_{\varepsilon_k}^a(r_ky+\hat{y}_k)\Big( A(r_k z+\hat{z}_k,r_k^2 t+t_k)-A(\hat{z}_k,t_k)
    \Big)\nabla u_k(\hat{z}_k,t_k)\cdot\nabla\phi\\
    &+\frac{\eta(\hat{z}_k,t_k)\nu_k^{-a}}{L_k r_k^\alpha}\int_{\supp(\phi)}\rho_{\varepsilon_k}^a(r_ky+\hat{y}_k)\Big( A(\hat{z}_k,t_k)\nabla u_k(\hat{z}_k,t_k)+F_{\varepsilon_k}(\hat{z}_k,t_k)
    \Big)\cdot\nabla\phi.
    \end{aligned}
\end{align}
Next, we show that the r.h.s. of \eqref{equazione C^1-alpha per le w_k} vanishes in a distributional sense as $k\to+\infty$. The first member can be estimate exactly as in Theorem \ref{teo C^0-alpha}, and by the hypothesis on $p$ and $\alpha$, we obtain the desired convergence to zero. The second can be bounded as follows
\begin{align*}
    &\Big|\frac{\eta(\hat{z}_k,t_k)\nu_k^{-a}}{L_k r_k^\alpha}\int_{\supp(\phi)}\rho_{\varepsilon_k}^a(r_ky+\hat{y}_k)\Big( F_{\varepsilon_k}(r_k z+\hat{z}_k,r_k^2 t+t_k)-F_{\varepsilon_k}(\hat{z}_k,t_k)
    \Big)\cdot\nabla\phi \Big|\\
    &\le\frac{\nu_k^{-a}\|\nabla\phi\|_{L^\infty(Q^\infty)} }{L_k r_k^a}  \int_{\supp(\phi)}\rho_{\varepsilon_k}^a(r_ky+\hat{y}_k)C r_k^\alpha (|z|+|t|^{1/2})^{\alpha}
    \le \frac{C}{L_k}\to 0,
\end{align*}
as $k\to+\infty$. In the previous inequalities, we have used the uniform boundedness of $F_{\varepsilon_k}$ in $C^{0,a}_p$-space and the estimate $\int_{\supp(\phi)}\Tilde{\rho}_k^a\le C$.

Next, we show that the fourth member vanishes. First, we can rewrite it as 
\begin{align}\label{estaiate-boundary.term}
\begin{aligned}
    &\frac{\eta(\hat{z}_k,t_k)\nu_k^{-a}}{L_k r_k^\alpha}\int_{\supp(\phi)}\rho_{\varepsilon_k}^a(r_ky+\hat{y}_k)\Big( A(\hat{z}_k,t_k)\nabla u_k(\hat{z}_k,t_k)+F_{\varepsilon_k}(\hat{z}_k,t_k)
    \Big)\cdot\nabla\phi \\
    &=\frac{\eta(\hat{z}_k,t_k)\nu_k^{-a}}{L_k r_k^\alpha}\int_{\supp(\phi)}{\div}\Big(\rho_{\varepsilon_k}^a(r_ky+\hat{y}_k)(A(\hat{z}_k,t_k)\nabla u_k(\hat{z}_k,t_k)+F_{\varepsilon_k}(\hat{z}_k,t_k))\phi \Big)\\
    &+\frac{\eta(\hat{z}_k,t_k)\nu_k^{-a}}{L_k r_k^\alpha}\int_{\supp(\phi)}\partial_y\big(\rho_{\varepsilon_k}^a(r_ky+\hat{y}_k)\big)(A(\hat{z}_k,t_k)\nabla u_k(\hat{z}_k,t_k)+F_{\varepsilon_k}(\hat{z}_k,t_k))\cdot e_{N+1}\phi .
    \end{aligned}
\end{align}
By using the divergence theorem, we can rewrite the first member in \eqref{estaiate-boundary.term} as
\begin{align*}
    &\int_{\supp(\phi)}{\div}\bigg(\rho_{\varepsilon_k}^a(r_ky+\hat{y}_k)(A(\hat{z}_k,t_k)\nabla u_k(\hat{z}_k,t_k)+F_{\varepsilon_k}(\hat{z}_k,t_k))\phi \bigg)dzdt\\
    &={\int_{\partial\{\supp(\phi)\}}}\bigg(\rho_{\varepsilon_k}^a(r_ky+\hat{y}_k)(A(\hat{z}_k,t_k)\nabla u_k(\hat{z}_k,t_k)+F_{\varepsilon_k}(\hat{z}_k,t_k))\phi \bigg)d\sigma,
\end{align*}
and observe that this is equal to zero. In fact, in \textbf{Case 1} we have $Q^\infty=\mathbb{R}^{N+2}$ and $\phi$ has compact support. Instead, in \textbf{Case 2}, since $\hat{z}_k$ lies on the flat boundary, the term vanishes by the conormal boundary condition satisfied by $u_k$.

The second term in the r.h.s. of \eqref{estaiate-boundary.term} vanishes too. In \textbf{Case 2} it is identically zero since $(A\nabla u_k+F_{\varepsilon_k})(\hat{z}_k,t_k)=0$ by the conormal boundary condition.
Let us consider \textbf{Case 1} and recall that $\hat{y}_k=y_k$ and $\frac{r_k}{y_k}\to0$ as $k\to+\infty$. Then, on compact subsets of $\mathbb{R}^{N+2}$, one has the following estimate
\begin{align*}
    \big|\nu_k^{-a}\partial_y[\rho_{\varepsilon_k}^a(r_k y+{y_k})]\big|&=\Big|\nu_k^{-a} ar_k \rho_{\varepsilon_k}^a(r_k y+{y_k})\frac{r_k y+{y}_k}{\varepsilon_k^2+(r_k y+{y}_k)^2}\Big| \le a\Tilde{\rho}_k^a(y)\frac{r_k}{y_k}\left|\frac{\frac{r_k}{y_k}y+1}{\frac{\varepsilon_k^2}{y_k^2}+\left(\frac{r_k}{y_k}y+1\right)^2}\right|\le C \frac{r_k}{y_k}.
\end{align*}
Next, let $(\zeta_k,t_k)=(x_k,0,t_k)$ be the projection of $(\hat{z}_k,t_k)=(z_k,t_k)$ on the hyperplane $\{y=0\}$.
By the conormal boundary condition, we have that $[\eta(A\nabla u_k+F_{\varepsilon_k})](\zeta_k,t_k)\cdot e_{N+1}=0$, so
\begin{align*}
    &\frac{\eta(\hat{z}_k,t_k)\nu_k^{-a}}{L_k r_k^\alpha}\int_{\supp(\phi)}\partial_y\big(\rho_{\varepsilon_k}^a(r_ky+\hat{y}_k)\big)(A\nabla u_k+F_{\varepsilon_k})(\hat{z}_k,t_k)\cdot e_{N+1}\phi(z,t)dzdt\\
     &=\frac{\nu_k^{-a}}{L_k r_k^\alpha}\int_{\supp(\phi)}\partial_y\big(\rho_{\varepsilon_k}^a(r_ky+\hat{y}_k)\big)\big[\eta(A\nabla u_k+F_{\varepsilon_k})\big](\hat{z}_k,t_k)\cdot e_{N+1}\phi(z,t)dzdt     \\
      &-\frac{\nu_k^{-a}}{L_k r_k^\alpha}\int_{\supp(\phi)}\partial_y\big(\rho_{\varepsilon_k}^a(r_ky+\hat{y}_k)\big)\big[\eta(A\nabla u_k+F_{\varepsilon_k})\big](\zeta_k,t_k)\cdot e_{N+1}\phi(z,t)dzdt.
\end{align*}
We can estimate
\begin{align*}
    &\Big|\big[\eta(A\nabla u_k+F_{\varepsilon_k})\big](\hat{z}_k,t_k)-\big[\eta(A\nabla u_k+F_{\varepsilon_k})\big](\zeta_k,t_k)\Big|\\
    &\le \Big|A \nabla (\eta u_k)(\hat{z}_k,t_k)-A \nabla (\eta u_k)(\zeta_k,t_k)\Big|+\Big|u_kA\nabla\eta (\hat{z}_k,t_k)-u_kA\nabla\eta (\zeta_k,t_k)\Big|  +\Big|\eta F_{\varepsilon_k}(\hat{z}_k,t_k)-\eta F_{\varepsilon_k}(\zeta_k,t_k)\Big|\\
    & \le C L_k y_k^\alpha.
\end{align*}
We remark here that in order to estimate the second term above we have used the uniform $C^{0,\gamma}_p$ regularity of $u_k$s for some chosen $\gamma\geq\alpha$. Finally, we obtain that
\begin{align*}
    \frac{\eta(\hat{z}_k,t_k)\nu_k^{-a}}{L_k r_k^\alpha}\int_{\supp(\phi)}\partial_y\big(\rho_{\varepsilon_k}^a(r_ky+\hat{y}_k)\big)(A\nabla u_k+F_{\varepsilon_k})(\hat{z}_k,t_k)\cdot e_{N+1}\phi(z,t)dzdt
    \le C\left(\frac{r_k}{y_k}\right)^{1-\alpha}\to 0.
\end{align*}
So, also in \textbf{Case 1} we obtain that the fourth member of the r.h.s. of \eqref{equazione C^1-alpha per le w_k} vanishes. 

To conclude, we prove that the third member of \eqref{equazione C^1-alpha per le w_k} goes to zero as $k\to+\infty$. Notice that
\begin{align*}  
&\Big|\eta(\hat{z}_k,t_k)\Big(A(r_k z+\hat{z}_k,r_k^2 t+t_k)-A(\hat{z}_k,t_k)
    \Big)\nabla u_k(\hat{z}_k,t_k)\Big|\\
    &=\Big| \Big(A(r_k z+\hat{z}_k,r_k^2 t+t_k)-A(\hat{z}_k,t_k)
    \Big)\nabla(\eta u_k)(\hat{z_k},t_k) -\big(A(r_k z+\hat{z}_k,r_k^2 t+t_k)-A(\hat{z}_k,t_k)
    \big)\nabla\eta(\hat{z}_k,t_k)u_k(\hat{z}_k,t_k)\Big|\\
    &\le r_k^\alpha\|\nabla(\eta u_k)\|_{L^\infty(Q^+_{3/4})}+r_k^\alpha \|\nabla \eta\|_{L^\infty(Q^+_{3/4})}\|u_k\|_{L^\infty(Q^+_{3/4})}\le C r_k^\alpha L_k,
    \end{align*}
where we have used the following parabolic Hölder interpolation inequality, see \cite{Lie96}*{Proposition 4.2}
$$\|\nabla(\eta u_k)\|_{L^\infty(Q^+_{3/4})}\le C\Big(\|\eta u_k\|_{L^\infty(Q^+_{3/4})}+[\eta u_k]_{C^{1,\alpha}_p(Q^+_{3/4})} \Big)\le C (1+L_k).$$
Then, in order to make vanish the full term we need to reason in two steps: first, one proves a uniform $C^{1,\beta}$ estimate with a given suboptimal $\beta\in(0,\alpha)$. In fact, in this case the third term vanishes as follows
\begin{align*}
    \bigg|\frac{\eta(\hat{z}_k,t_k)\nu_k^{-a}}{L_k r_k^\beta}\int_{\supp(\phi)}\rho_{\varepsilon_k}^a(r_ky+\hat{y}_k)&\Big( A(r_k z+\hat{z}_k,r_k^2 t+t_k)-A(\hat{z}_k,t_k)
    \Big) \nabla u_k(\hat{z}_k,t_k)\cdot\nabla\phi(z,t)dzdt\bigg|\\
    \le& r_k^{\alpha-\beta}\int_{\supp(\phi)}\Tilde{\rho}_k^a(y)\|\nabla \phi\|_{L^\infty(Q^\infty)}dzdt\le C r_k^{\alpha-\beta}\to0,
\end{align*}
as $k\to+\infty$. Then one can procede with the suboptimal exponent $\beta$ up to the end of the present proof. This provides uniform boundedness of the sequence $\nabla u_k$. Then, restarting the proof with the optimal $\alpha$ and the additional information above, in the previous computation we get
\begin{align*}
    \Big|\frac{\eta(\hat{z}_k,t_k)\nu_k^{-a}}{L_k r_k^\alpha}\int_{\supp(\phi)}\rho_{\varepsilon_k}^a(r_ky+\hat{y}_k)&\Big ( A(r_k z+\hat{z}_k,r_k^2 t+t_k)-A(\hat{z}_k,t_k)
    \Big)\nabla u_k(\hat{z}_k,t_k)\cdot\nabla\phi(z,t)dzdt\Big|\\
    \le& \frac{C}{L_k}\|\nabla u_k\|_{L^\infty(Q_{3/4}^+)},
\end{align*}
which converges to zero. Putting together all previous information, we have proved that the r.h.s. in \eqref{equazione C^1-alpha per le w_k} vanishes as $k\to+\infty$. 

Let $(\Bar{z},\Bar{t})=\lim_{k\to+\infty}(\Hat{z_k},t_k)$ and $\bar{A}:=\lim_{k\to+\infty} A(r_kz+\hat{z}_k,r_k^2t+t_k)$. Arguing as in Theorem \ref{teo C^0-alpha}, we can prove the convergence of the l.h.s. of  \eqref{equazione C^1-alpha per le w_k} in the following sense
$$\int_{\supp(\phi)}\tilde{\rho}_k^a \big(-w_k\partial_t \phi+A(r_kz+\hat{z}_k,r_k^2t+t_k) \nabla w_k\cdot\nabla\phi\big)\to\int_{\supp(\phi)}\tilde{\rho}^a \big(-\bar{v}\partial_t \phi+\bar{A} \nabla \bar{v}\cdot\nabla\phi\big),$$
and obtain that $\Bar{v}$ is an entire solution to
\begin{equation}\label{eq:barv:entire1}
\partial_t\bar{v}-{\div}(\bar{A}\nabla\Bar{v})=0\hspace{0.4cm}\text{in }\mathbb{R}^{N+2},
\end{equation}
in \textbf{Case 1} and $\Bar{v}$ is an entire solution to
\begin{equation}\label{eq:barv:entire2}
       \begin{cases}
       \Tilde{\rho}^{a}\partial_t\bar{v}-{\div}(\Tilde{\rho}^a\bar{A}\nabla\Bar{v})=0&\text{in }\mathbb{R}^{N+1}_+\times\mathbb{R},\\
       \Tilde{\rho}^a \bar{A}\nabla \Bar{v}\cdot e_{N+1}=0&\text{on }\mathbb{R}^{N}\times\{0\}\times\mathbb{R},
   \end{cases}
   \end{equation}
in \textbf{Case 2}.

\

\emph{Step 7: Liouville theorems.} Since $\Bar{v}$ is globally $C^{1,\alpha}_p$-continuous in $Q^\infty$, it follows that
\begin{align*}
    |2\Bar{v}(z,t)|&\le|2\Bar{v}(z,t)-\Bar{v}(0,t)-\nabla \bar{v}(0,t)\cdot z-\Bar{v}(z,0)| +|\bar{v}(0,t)|+|\nabla \bar{v}(0,t)\cdot z|+|\bar{v}(z,0)|\\
    &\le|\Bar{v}(z,t)-\Bar{v}(0,t)-\nabla \bar{v}(0,t)\cdot z|+|\Bar{v}(z,t)-\Bar{v}(z,0)|+C+C|z|+C\\
    &\le C|z|^{1+\alpha}+C|t|^{\frac{1+\alpha}{2}}+C(1+|z|)\\
    &\le C(1+(|z|^2+|t|))^{\frac{1+\alpha}{2}},
\end{align*}
The estimate above exploit a first-order expansion in the spacial variable $z$ for $t$ fixed. However, the constant $C>0$ can be chosen independently from the point $(z,t)$.

Hence, as in Theorem \ref{teo C^0-alpha}, by the growth condition above, we can apply the Liouville Theorem \ref{teo liouville 1} in both \textbf{Case 1} and \textbf{Case 2}, keeping in mind Remark \ref{liouville-matrice-non-costante}, and obtain that $\Bar{v}$ is a linear function, independent of $t$, in contradiction with the fact that $\nabla \bar v
$ is not constant.

\

\emph{Step 8: The case $L_k=[\eta u_k]_{C_t^{0,\frac{1+\alpha}{2}}(Q^+_1)}$.} In this case, the argument is similar with minor differences.
As above, we take two sequences of points $P_k=(z_k,t_k),\Bar{P}_k=(z_k,\tau_k)\in Q^+_{3/4}$, such that
\begin{equation}\label{eq:explosion:time:seminorm}
\frac{|(\eta u_k)(z_k,t_k)-(\eta u_k)(z_k,\tau_k)|}{|t_k-s_k|^\frac{1+\alpha}{2}}\ge\frac{1}{2}L_k\to+\infty.
\end{equation}
Defining $r_k:=d_p(P_k,\bar{P}_k)=|t_k-\tau_k|^{1/2}$, by \eqref{eq:explosion:time:seminorm} and the local uniform boundedness of solutions, see Proposition \ref{moser-3.5}, we get $r_k\to0$.

We define two blow-up sequences $v_k$ and $w_k$ as in \eqref{v_k,w_k_degenerate}, centered in the new blow-up sequence $P_k$, defined on the domains $Q(k)$, which are the same as above and set $Q^\infty:=\lim_{k\to +\infty}Q(k)$. 

Since $[\partial_j(\eta u_k)]_{C^{0,\alpha}_p(Q_1^+)}\le L_k$, for every $j=1,\dots,N+1$, we obtain that the estimates \eqref{5.1-degenerate} and \eqref{5.1-degenrate time} holds; that is, $[v_k]_{C^{1,\alpha}_p(K)}\le C$, uniformly in $k$, for every compact set $K\subset Q^\infty$. Defining $\bar{v}_k$ and $\bar{w}_k$ as in \eqref{eq:bar:v}, we can use the Arzelá-Ascoli theorem to obtain that $\bar{v}_k \to \bar{v}$ in $C^{1,\gamma}_p(K)$, for any $\gamma\in(0,\alpha)$, $\bar{w}_k\to \bar{v}$ uniformly on $K$ and that $\bar{v}$ is globally $C^{1,\alpha}_p$-continuous in $Q^\infty$.

The crucial difference between this case and the previous one is in \emph{Step 4}: in this case we claim that $\bar{v}$ is non constant in the variable $t$. Indeed, we have that
\begin{align*}
    &\Big|\Bar{v}_k\left(0,\frac{t_k-\tau_k}{r_k^2}\right)-\bar{v}_k(0,0)\Big|=\Big|{v}_k\left(0,\frac{t_k-\tau_k}{r_k^2}\right)\Big| = \frac{|\eta{(z_k,\tau_k)}(u_k(z_k,\tau_k)-u_k(z_k,t_k))|}{L_k r_k^{1+\alpha}}\\
    &\ge \frac{|(\eta u_k)(z_k,\tau_k)-(\eta u_k)(z_k,t_k)|}{L_k r_k^{1+\alpha}}-\frac{|(\eta(z_k,
    \tau_k)-\eta (z_k,t_k))u_k(z_k,t_k)|}{L_k r_k^{1+\alpha}}\\
    &\ge\frac{1}{2}-\frac{Mr_k^{1-\alpha}\|u_k\|_{L^\infty(Q_{3/4}^+)}}{L_k}=\frac{1}{2}+o(1),
\end{align*}
as $k\to+\infty$.
Observing that $\frac{t_k-\tau_k}{r_k^2}\to\Bar{t}\not=0$, up to consider a subsequence, we can take the limit as $k\to+\infty$ in the previous computation to obtain $|\bar{v}(0,\bar{t})-\bar{v}(0,0)|\ge\frac{1}{2}$; that is, $\bar{v}$ is non constant in $t$. 

With the same argument of \emph{Step 6} we can prove that $\bar{v}$ is an entire solution to \eqref{eq:barv:entire1} or \eqref{eq:barv:entire2}. Moreover, since $\bar{v} $ is globally $C^{1,\alpha}_p$-continuous in $Q^\infty$, then it satisfies a parabolic sub-quadratic growth condition. Hence by the Liouville theorem (Theorem \ref{teo liouville 1}), we get that $\bar{v}$ should be a linear function not depending on $t$ and this is a contradiction.
\end{proof}

Combining the uniform estimates obtained for the regularized problems and Lemma \ref{y-to-rho} in the half cylinder (see Remark \ref{rem:approx:half}), we obtain our main Theorem \ref{teo:C^1,alpha} as a byproduct.

\begin{proof}[Proof of Theorem \ref{teo:C^1,alpha}]
Let $u$ be a weak solution to \eqref{eq:1} in $Q_1^+$ in the sense of Definition \ref{def.solution-mezza-palla}. By Lemma \ref{y-to-rho} and Remark \ref{rem:approx:half}, we can find sequences $\{ {u_{\varepsilon_k}} \}_k$, $\{f_{\varepsilon_k}\}_k$, $\{F_{\varepsilon_k}\}_k$ as $\e_k\to0^+$, such that every $u_{\e_k}$ is a solution to
\[
\begin{cases}
\rho_{\varepsilon_k}^a \partial_t u_{\e_k} - \div (\rho_{\e_k}^{a} A\nabla u_{\e_k}) = \rho_{\e_k}^a f_{\e_k} + \div(\rho_{\e_k}^{a} F_{\e_k}) \quad &\text{ in }  Q_{3/4}^+\\
\rho_{\e_k}^a\left(A\nabla u_{{\e_k}}+F_{\e_k}\right)\cdot e_{N+1} = 0 \quad &\text{ in }\partial^0Q_{3/4}^+,
    \end{cases}
\]
and $u_{\varepsilon_k}\to u$ in $ L_{loc}^2(I_{3/4};H_{loc}^1(B_{3/4}\setminus\Sigma))$ as $\e_k\to0^+$. Furthermore, $f_{\e_k}$ and $F_{\e_k}$ satisfy the assumptions of Theorem \ref{teo C^0-alpha} (respectively of Theorem \ref{teo-C1 alpha epsilon}): this implies uniform boundedness of the $C^{0,\alpha}_p(Q_{1/2}^+)$-norm of $u_{\e_k}$ (respectively of the $C^{1,\alpha}_p(Q_{1/2}^+)$-norm). Then, by the Arzelà-Ascoli Theorem and by the a.e. convergences $u_{\varepsilon_k}\to u$ and $\nabla u_{\varepsilon_k}\to \nabla u$, we obtain that the estimates \eqref{eq:C0,alpha} and \eqref{eq:C1,alpha} hold true.

Finally, in the $C^{1,\alpha}_p$ case, the boundary condition \eqref{eq:boundary} follows by the $C^1(Q_{1/2}^+)$-convergence $u_{\e_k}\to u$.
\end{proof}

\subsection{Weights degenerating on curved characteristic manifolds}\label{sec:curve}
In this last section, we show how to extend the $C^{1,\alpha}$ regularity estimates to weak solutions of a class of equations having weights vanishing or exploding on curved characteristic manifolds $\Gamma$, as in \eqref{eq:1:curve}. Let us begin with the notion of weak solutions to \eqref{eq:1:curve}.

\begin{defi}\label{def.solution-curve}
Let $a>-1$ and $N\ge1$. Let $\varphi\in C^{1,\alpha}(B_1\cap\{y=0\})$ be the parametrization defined in \eqref{phi}, $\delta\in C^{1,\alpha}(\Omega^+\cap B_1)$ satisfying \eqref{delta}, $f \in L^2((\Omega^+\cap B_1)\times(-1,1),\delta^a)$ and $F \in L^2((\Omega^+\cap B_1)\times(-1,1),\delta^a)^{N+1}$. We say that $u$ is a weak solution to \eqref{eq:1:curve} 

if $u \in L^2(I_1;H^1(\Omega^+\cap B_1,\delta^a)) \cap L^\infty(I_1;L^2(\Omega^+\cap B_1,\delta^a))$ and satisfies
\[
-\int_{(\Omega^+\cap B_1)\times(-1,1)}\delta^a u\partial_t\phi \, dzdt + \int_{(\Omega^+\cap B_1)\times(-1,1)}\delta^a A\nabla u\cdot\nabla\phi \, dzdt = \int_{(\Omega^+\cap B_1)\times(-1,1)}\delta^a( f\phi -F\cdot\nabla\phi )\, dzdt,
\]
for every $\phi\in C_c^\infty(Q_1)$.
\end{defi}
\begin{proof}[Proof of Corollary \ref{cor:C^1,alpha}]
Since the proof is very similar to the one of Theorem \ref{teo:C^1,alpha}, we just sketch it highlighting the main differences.

\

\emph{Step 1. Reducing to flat characteristic manifolds by a local diffeomorphism.} Let us consider a classical diffeomorphism which straighten the hypersurface $\Gamma$,
      \[
      \Phi(x,y)=(x,y+\varphi(x)),
      \]
      which is of class $C^{1,\alpha}$ and then $C^{1,\alpha}_p$ extending constantly in the time variable. Actually, $\Phi^{-1}$ locally flattens $\Gamma$ to $\Sigma$. In fact, there exists a small radius $R>0$ such that $\Phi (B_R\cap\{y>0\})\subseteq B_1\cap\{y>\varphi(x)\}$, $\Phi(0)=\Phi^{-1}(0)=0$ and $\Phi (B_R\cap\{y=0\})\subseteq B_1\cap\{y=\varphi(x)\}$. The Jacobian associated to $\Phi$ is
       $$J_\Phi(x)=\left(\begin{array}{c|c}
       {I}_N & 0  \\ \hline
       (\D\varphi(x))^T &1 \\ 
  \end{array}\right), \quad \text{with } |\det J_\Phi|\equiv1.$$
      Up to a dilation, one has that $\tilde{u}:=u\circ(\Phi(x),t)$ is a weak solution to
\[
\begin{cases}
{\tilde\delta^a\partial_t\tilde{u}}-\div(\tilde\delta^a\tilde{A}\nabla \tilde{u})=\tilde\delta^a\tilde{f}+\div(\tilde\delta^a\tilde{F}),&\text{in }Q_1^+,\\
\displaystyle{\lim_{y\to0^+}}\tilde{\delta}^a(\tilde{A}\nabla\tilde{u}+\tilde{F})\cdot e_{N+1}=0 & \text{on }\partial^0Q_1^+.
\end{cases}   
\]   
      where $\tilde{\delta}=\delta\circ\Phi$, $\tilde{f}=f\circ(\Phi(x),t)$ and $\tilde{F}= J_\Phi^{-1} F\circ(\Phi(x),t)$ and  $\tilde{A}=(J_\Phi^{-1})(A\circ(\Phi(x),t))(J_\Phi^{-1})^T$.
      
      By \cite{TerTorVit22}*{Lemma 2.3}, $\tilde{\delta} \in C^{1,\alpha}(B_1^+)$ and satisfies
      $$\tilde{\delta}>0 \text{ in }B_1^+, \qquad \tilde{\delta}=0 \text{ on }\partial^0 B_1^+,\qquad \partial_y\tilde{\delta}>0\text{ on }\partial^0 B_1^+, \qquad\frac{\tilde{\delta}}{y}\in C^{0,\alpha}(B_1^+),\qquad \frac{\tilde{\delta}}{y}\ge\mu>0 \text{ in } \overline{B_1^+},$$
      where the last nondegeneracy condition is a consequence of the assumption $|\nabla\delta|\geq c_0>0$. Now, noticing that $\tilde{u}$ is a weak solution to 
      \begin{equation}
\begin{cases}\label{eq:curved}
y^a\left(\frac{\tilde{\delta}}{y}\right)^a\partial_t\tilde{u}-\div(y^a\bar{A}\nabla \tilde{u})=y^a\bar{f}+\div(y^a\bar{F}),&\text{ in }Q_1^+,\\
\displaystyle{\lim_{y\to0^+}}y^a(\bar{A}\nabla\tilde{u}+\bar{F})\cdot e_{N+1}=0 & \text{on }\partial^0Q_1^+.
\end{cases}   
      \end{equation}  
where $\bar{A}=\tilde{A}(\tilde{\delta}/y)^a\in C^{0,\alpha}_p(Q_1^+)$, $\bar{f}=\tilde{f}(\tilde{\delta}/y)^a \in L^p(Q_1^+,y^a)$ and $\bar{F}=\tilde{F}(\tilde{\delta}/y)^a \in C^{0,\alpha}_p(Q_1^+)$, we are taken back to an equation with the standard degenerate or singular weight $y^a$ as in \eqref{eq:1}, but with a new nondegenerate term $(\tilde{\delta}/y)^a$ in front of the time derivative.

\

\emph{Step 2. Regularity for flat characteristic manifolds with an extra term in front of the time derivative.} In what follows we show that our regularity theory applies with minor changes to weak solutions to \eqref{eq:curved}; that is, where an extra term $b$ appears in front of the time derivative in the parabolic equation. The term needs to be uniformly continuous in $B_1^+$ and bounded away from zero $b\ge\mu>0$. In the present case $b(z):=(\tilde{\delta}(z)/y)^a$, which is even H\"older continuous.

First, the energy results obtained in Sections \ref{section2}, \ref{section3}, \ref{section4} can be easily extended just using the fact that the positive term $b$ is bounded and bounded away from zero. These bounds ensure invariance of the norms involved in the functional setting.

Let us focus on the only difference, respect to the proof of Theorem \ref{teo-C1 alpha epsilon}; that is, the $C^{1,\alpha}_p$ $\e$-stable regularity of solutions with regularized weights $\rho_\e^a$ (the proof of Theorem \ref{teo C^0-alpha}, the $\e$-stability for the $C^{0,\alpha}_p$ estimate, is analogous): in order to prove that the blow-up sequence $\{\bar{w}_k\}$ (see \eqref{eq:bar:v}) converges to an entire solution with constant coefficients, one considers the limit in the equation \eqref{equazione C^1-alpha per le w_k} satisfied by $\bar{w}_k$ with the necessary modifications for the present case. The l.h.s. converges in the following sense: by using the same considerations of Lemma \ref{rho-to-y} we have that
\[
\int_{\supp(\phi)}\tilde{\rho}_{k}^a(y)\Big(-b(r_kz+\hat{z}_k)\Bar{w}_k\partial_t\phi+A(r_k z+\hat{z}_k,r_k^2 t+t_k)\nabla\bar{w}_k\cdot\nabla\phi\Big)
\to
\int_{\supp(\phi)}\tilde{\rho}^a \big(-\bar{b}\bar{v}\partial_t \phi+\bar{A} \nabla \bar{v}\cdot\nabla\phi\big),
\]
where $\bar{b}=\lim_{k\to+\infty}b(r_kz+\hat{z}_k)$ is a positive constant and $\bar{A}=\lim_{k\to+\infty}A(r_kz+\hat{z}_k,r_k^2t+t_k)$ is a constant coefficient matrix. Therefore, the contradiction argument ends up again with the use of the Liouville Theorem \ref{teo liouville 1}. Finally, by Lemma \ref{rho-to-y}, with the same considerations done in the proof of Theorem \ref{teo:C^1,alpha}, the statement follows.
\end{proof}

\section*{Acknowledgement}
The authors are research fellow of Istituto Nazionale di Alta Matematica INDAM group GNAMPA. A.A and G.F. are supported by the GNAMPA-INDAM project \emph{Teoria della regolarit\`a per problemi ellittici e parabolici con diffusione anisotropa e pesata}, CUP\_E53C22001930001. S.V. is supported by the MUR funding for Young Researchers - Seal of Excellence SOE\_0000194 \emph{(ADE) Anomalous diffusion equations: regularity and geometric properties of solutions and free boundaries}, and supported also by the PRIN project 2022R537CS \emph{$NO^3$ - Nodal Optimization, NOnlinear elliptic equations, NOnlocal geometric problems, with a focus on regularity}  by the GNAMPA-INDAM project \emph{Regolarit\`a e singolarit\`a in problemi di frontiere libere}, CUP\_E53C22001930001.

%%%%%%%%%%%%%%%%%%%%%%%%%%%%%%%%%%%%%%%%%%%%%%%%%%%%%%%%%%%%%%%%%%%%%%%%%%%%%%%%%%%%%%%%%%%%%

%REFERENCES

%%%%%%%%%%%%%%%%%%%%%%%%%%%%%%%%%%%%%%%%%%%%%%%%%%%%%%%%%%%%%%%%%%%%%%%%%%%%%%%%%%%%%%%%%%%%%

%
%
%
%
%

\end{document}